\documentclass[a4paper]{amsart}

\pdfoutput=1

\usepackage[utf8]{inputenc}
\usepackage[T1]{fontenc}
\usepackage{lmodern}
\usepackage{color}
\usepackage{amsthm, amssymb, amsmath, amsfonts, mathrsfs, dsfont, esint}
\usepackage[colorlinks=true, pdfstartview=FitV, linkcolor=blue, citecolor=blue, urlcolor=blue,pagebackref=false]{hyperref}

\usepackage{microtype}

\newcommand{\ignore}[1]{}
\newcommand{\oh}{{\textstyle\frac{1}{2}}}
\newtheorem{definition}{Definition}

\newtheorem{theorem}{Theorem}
\newtheorem{remark}{Remark}
\newtheorem{lemma}{Lemma}
\newtheorem{corollary}{Corollary}
\providecommand{\dx}{\, \mathrm{d} x}

\renewcommand{\le}{\leqslant}
\renewcommand{\ge}{\geqslant}

\renewcommand{\subset}{\subseteq}

\numberwithin{equation}{section}

\title[Green's function for elliptic systems]{Green's function for elliptic systems: existence and \\
Delmotte-Deuschel bounds }
\author{Joseph G. Conlon, Arianna Giunti, Felix Otto}

\address[Joseph G. Conlon]{University of Michigan, Department of Mathematics, Ann Arbor, MI 48109-1109, USA}

\address[Arianna Giunti]{Max Planck Institute for Mathematics in the Sciences, Leipzig, Germany}

\address[Felix Otto]{Max Planck Institute for Mathematics in the Sciences, Leipzig, Germany}

\begin{document}
\begin{abstract}
 This paper is divided into two parts: in the main deterministic part, we prove that for an open domain $D \subset \mathbb{R}^d$ with $d \geq 2$, for every (measurable) uniformly elliptic tensor field $a$ and for almost every point $y \in D$, there exists a unique Green's function centred in $y$ associated to the vectorial
operator $-\nabla \cdot a\nabla $ in $D$. 
This result  implies  the existence of the fundamental solution for elliptic systems when  $d>2$, i.e. the Green function for $-\nabla \cdot a\nabla$ in $\mathbb{R}^d$.
In the second part, we introduce a shift-invariant ensemble $\langle\cdot \rangle$ over the set of uniformly elliptic tensor fields, and infer for the fundamental solution $G$ some pointwise bounds for $\langle |G(\cdot; x,y)|\rangle$, $\langle|\nabla_x G(\cdot; x,y)|\rangle$ and $\langle |\nabla_x\nabla_y G(\cdot; x,y)|\rangle$.
These estimates scale optimally in space and provide a generalisation to systems of the bounds obtained by Delmotte and Deuschel for the scalar case.
\end{abstract}
\maketitle
\section{Introduction}
In this work we shall be concerned with the study of the Green function for the second order vectorial operator in divergence form $-\nabla \cdot a \nabla$, on a general open domain $D \subset \mathbb{R}^d$ with $d \geq 2$.  
More precisely, let $G_D(a;\cdot ,y), \ y \in D$ denote the Green function centred in $y$ and corresponding to $-\nabla \cdot a\nabla$ in $D$.
We assume that $a$ is \textit{any} measurable uniformly elliptic tensor field on $\mathbb{R}^d$: Our notion of ellipticity is slightly weaker than the more standard strong ellipticity, and this allows us to   include
the setting of linearised elasticity. We also need to impose  an extra condition on the domain $D$ in the case $d=2$, namely that it has at least one bounded direction. 

The paper is structured in two parts. In the first \textit{deterministic} part we provide an existence and uniqueness result for Green functions. That is we prove  for every $a$ and almost every $y \in D$ 
the function $G_D(a; \cdot , y)$ exists (in fact, in the case of systems it is a tensor field). 
In the case $d>2$ this implies the existence of the full-space Green function, i.e. of $G(a;\cdot,y)=G_{\mathbb{R}^d}(a;\cdot,y)$. In the second \textit{stochastic} part we introduce a shift-invariant probability measure 
on the coefficient fields $a$ (or, equivalently, an ensemble $\langle \cdot \rangle$), and consider when $d > 2$ the random field given by $G(a; \cdot, y)$. In this random setting we establish for $G$ optimal pointwise moment bounds:
If $\langle\cdot\rangle$ denotes expectation with respect to the ensemble  and $\lambda$ is the ellipticity ratio of $a$, we prove that
\begin{align}\label{DD1}
\langle |G(\cdot;x,y)|\rangle \leq \frac{C\mbox{\scriptsize $(\lambda, d)$}}{|x-y|^{d-2}},
\end{align}
with similar estimates for $\nabla_x G$, and  $\nabla_x\nabla_yG$.

In the scalar case it is a well-known result (see e.g. Gr\"uter and Widman \cite{gw},  Littman, Stampacchia, and Weinberger \cite{lsw}) that for any measurable and strongly uniformly elliptic $a$, the Green function exists and 
has optimal pointwise decay, e.g. as the Green function associated to the Laplacian (c.f. also the r.h.s. in \eqref{DD1}). This bound on the decay is a consequence of the De Giorgi-Nash-Moser theory, which does not hold in the case of systems. 
Moreover, when working with systems the existence of a Green's function is itself not ensured  for this class of (possibly very rough) coefficient fields: A famous example of De Giorgi \cite{DeG}, disproving both the Liouville property and the $C^{\alpha}$-regularity theory for $a$-harmonic functions, also implies that there are strongly elliptic tensor fields and points $y \in \mathbb{R}^d$ for which a Green's function centred in $y$ does not exist.

Under additional regularity assumptions on the coefficient fields and/or on the domain $D$, the existence of the Green function has been proved also for systems: For any bounded and $C^1$ domain $D$, Fuchs \cite{f} establishes
existence of the Green function for strongly elliptic continuous coefficient fields $a$, and optimal pointwise bounds under the stronger assumption of H\"older continuity of $a$. Subsequently, Dolzmann and M\"uller \cite{dm}
improve the previous result by obtaining for continuous $a$ not only the existence of the Green function, but also optimal decay properties. 
In a series of works, Hoffman and Kim \cite{hk} and Kim and collaborators
(see e.g. \cite{kk} and \cite{cdk}) considerably weaken the assumptions on the domain $D$ and on the regularity of $a$ ( both in the elliptic and in the corresponding parabolic setting): In \cite{hk}, they establish the existence of the Green
function  for an arbitrary open domain $D \subset \mathbb{R}^d$ with $d > 2$, provided that the coefficient field is such that $a$-harmonic functions satisfy an interior H\"older continuity estimate (e.g. if $a$ is VMO).
In \cite{kk}, Kang and Kim (see also Cho, Dong and Kim \cite{cdk} for the case $d=2$) further develop the previous theory and in addition provide a necessary and sufficient condition on $a$ in order to have for the Green function
an optimal pointwise bound. We also mention that a result similar to \cite{kk} has been proved by Auscher and Tchamitchian \cite{at} in the parabolic case via the introduction of a criterion (the Dirichlet Property (D)) for the parabolic
Green function to have Gaussian bounds.

In this paper we adopt a different approach: Instead of restricting the class of coefficient fields $a$ by further regularity assumptions, we show that the 
"bad'' cases as the one of De Giorgi's example are \textit{exceptional} for any coefficient field $a$. The main idea consists of extending our definition of the Green function to a two-variable object $G_D(a; \cdot , \cdot)$ 
which solves the equation almost surely in $y$: With this understanding, we may establish  $L^2$ a-priori bounds in $(x,y)$ on  the gradients $\nabla_x G$, $\nabla_y G$ and the mixed derivatives $\nabla_x\nabla_yG$.
By an approximation procedure, we then obtain the existence statement.

The optimal stochastic estimates (c.f. \eqref{DD1} ), obtained in the second part of this work, extend the inequalities established by Delmotte and Deuschel \cite{dd} for scalar equations to elliptic systems:
Their methodology relies on the theory of De Giorgi, Nash and Moser for uniformly elliptic and parabolic PDEs in divergence form and therefore does not generalise to elliptic systems.
Stochastic estimates on the (whole space) Green function have been largely used in the context of stochastic homogenization for elliptic PDE's in divergence form, in particular to quantify the decay of the \textit{homogenisation error},
i.e. the difference between the solution of the heterogeneous equation and the solution of the homogenised equation. 
Inspired by the work of Naddaf and Spencer \cite{ns2} on Gradient Gibbs Measures, the third author and Gloria and the third author, Gloria and Neukamm (see e.g. \cite{go1}, \cite{GNO}) provide optimal estimates for the fluctuations (variance) of the corrector
by controlling the decay of the moments of the gradients and mixed derivatives Green function, i.e. $\langle |\nabla_x G(\cdot ; x, y)|^{2p} \rangle^{\frac 1 p}$ and $\langle |\nabla_x\nabla_y G(\cdot ; x, y)|^{2p} \rangle^{\frac 1 p}$ for $p \geq 1$.
There, an important role is played by the assumption on the ensemble of coefficient fields to satisfy a quantification of ergodicity in the form of a Spectral Gap.
In \cite{MO}, Marahrens and the third author rely on Logarithmic Sobolev's inequalities to upgrade the bounds obtained by Delmotte and Deutschel for any moment of $G$, $\nabla_x G$ and $\nabla_x\nabla_y G$
and infer optimal estimates on the fluctuation of the homogenisation error.
The result of this paper should therefore allow to obtain the optimal quantitative results cited above also in the case of stochastic homogenisation of elliptic systems. 
We also mention that in \cite{BG}, Bella and the second author upgraded as well $\eqref{DD1}$ to a bound for any moment in probability of $G$, its gradient and its mixed derivatives.

Estimates \eqref{DD1} immediately imply optimal decay bounds for the averaged Green function $\langle G( a; \cdot, \cdot) \rangle$.
It is an interesting exercise to compare the methodology used in the present paper with the methodology used by the first author and Naddaf \cite{cn2} to prove in the scalar case pointwise estimates on the averaged Green function.
and their derivatives. While in this work we infer the stochastic bounds on the Green function directly from the deterministic existence result for $G(a; \cdot, \cdot)$, in \cite{cn2} a major part is played by the Fourier representation 
of the averaged Green function, which is a generalisation of the Fourier representation of the Green function for an elliptic PDE with constant coefficients. 
Naddaf and the first author then obtain Fourier space estimates strong enough to imply the pointwise estimates on averaged Green functions. We remark that the method in \cite{cn2} does not make use of the scalar structure and therefore may be
applied also to obtain a-priori estimates in the vectorial case. In the last section of this paper we indeed summarise how our main estimate in the proof of \eqref{DD1} can be obtained using this Fourier method.

We conclude this introduction by remarking that the issue of the regularity of averaged Green functions for PDEs with random coefficients plays an important role in statistical mechanics (in fact, \cite{dd} belongs itself to this context).
In particular, it appears to have first come up in the study of the equilibrium statistical mechanics of the Coulomb dipole gas.
Correlation inequalities for the dipole gas on the integer lattice were first obtained by Gawedzki  and Kupiainen \cite {gk} by means of a complicated multi-scale induction argument known as the renormalisation group method
\cite{br2}.  A major drawback to the implementation of the renormalisation group method is that it always requires  smallness in some parameters.
In the case of the dipole gas this implies that the density of the gas must be extremely small, and with no reasonable estimate on how large the density is allowed to be. In \cite{ns} Naddaf and Spencer pioneered an alternative
approach to estimating correlation functions for the dipole gas which was based on convexity theory.  Their starting point was the observation that a  correlation function closely related to the charge-charge correlation function
for the dipole gas is equal to the integral over time of an averaged Green function for a  {\it parabolic} PDE in divergence form with random coefficients.  
One of the main results of \cite{gk}  follows immediately from  this identity by application of a discrete version of the Aronson bounds  \cite{aronson}  for Green  functions of parabolic PDE. In addition, the argument applies for gases with density of order  $1$.
The Aronson bounds  make use of the main ingredient of Nash's argument for the De Giorgi-Nash-Moser theory, and are thus restricted to the scalar setting.  

An important intuition in the study of the Coulomb gas is the notion of {\it screening}. That is the interaction between two particles of the gas is decreased by the presence of the other particles.  
In the case of two dipoles centred at $x,y\in\mathbb{R}^d$ the interaction  behaves like $1/|x-y|^d$ at large distances.  Hence one expects that for a dilute Coulomb gas of dipoles the charge-charge correlation between two dipoles
at $x,y$ also behaves like $1/|x-y|^d$ at large distances.  In the Naddaf-Spencer representation the charge-charge correlation function is approximately given by the averaged second derivative $\nabla_x\nabla_yG$ of
the Green function evaluated at $(x,y)$.  Hence one is motivated to expect pointwise bounds on averages of second derivatives of Green functions for certain parabolic PDE in divergence form with random coefficients,
a conjecture formulated by Spencer \cite{sp} and proven in \cite{dd}.

\section{Notation and setting}\label{Not}
In this section we introduce the elliptic systems' setting we are interested in and the notion of associated Green's function for a general open domain $D\subset \mathbb{R}^d$, $d\geq 2$. In particular, we want to justify the scalar notation which is used throughout the paper.
In order to do so, we first introduce our problem in a more general setting: In the rest of this section we consider a Hilbert space $Y$ with dim$Y:=m<\infty$.
We denote by $z y$ and $z\cdot y$ respectively the inner product in Y and the natural one induced over $Y^d$.
In the same spirit, we write $|z|= (z z)^{\frac{1}{2}}$ and $|y|=(y\cdot y)^{\frac{1}{2}}$ for $z\in Y$ and $y\in Y^d$.

\bigskip

{\bf Coefficient field.}\ \ A coefficient field $a$ is a map
\begin{align}
a\ : \ \ &\mathbb{R}^d\rightarrow\nolinebreak\mathcal{L}(Y^d,Y^d)\notag\\
&x\mapsto a(x): Y^d\rightarrow Y^d.\notag
\end{align}
Let $\Omega$ be the set of all symmetric and elliptic coefficient fields, i.e. all maps $a$ such that
\begin{equation}\label{Sym}
 \forall x\in \mathbb{R}^d,\ \forall \ \xi, \xi'\in Y^d \ \ \xi\cdot a(x)\xi' = \xi'\cdot a(x)\xi ,
\end{equation}
\begin{equation}\label{bdd}
 \forall x\in \mathbb{R}^d,\ \forall \ \xi\in Y^d \ \ \ |a(x)\xi|\leq |\xi| .
\end{equation}
\begin{equation}\label{KC}
\exists \lambda\in(0,1) \ : \ \  \forall \ \zeta\in C^{\infty}_0(\mathbb{R}^d,Y) \ \ \ \int \nabla{\zeta}\cdot a\nabla{\zeta}\geq \lambda\int |\nabla{\zeta}|^2,
\end{equation}
We stress that (\ref{KC}) is a weaker condition than the uniform ellipticity assumption
\begin{equation} \label{StE}
\exists \lambda\in(0,1)  \ : \ \ \forall \ a.e.\ \  x\in X, \ \forall \ \xi\in Y^d \ \ \ \xi\cdot a(x)\xi \geq \lambda |\xi|^2 \  ,
\end{equation}
and it includes a larger class of elliptic systems to which belongs also the case of linearised elasticity. 
In this paper we mainly consider coefficient fields $a \in \Omega$, thus elliptic in the more general sense (\ref{KC}).

\medskip

{\bf Generalised Green's function.} For an open domain $D\subset \mathbb{R}^d$ with $ d \geq 2$ and a given $a\in\Omega$, we refer to the map $G_D(a;\cdot,\cdot) :\mathbb{R}^d\times \mathbb{R}^d \rightarrow \mathcal{L}(Y,Y)$ as a \textit{Green's function},
if there exists an $\alpha \in (0, d)$\footnote{In fact, we show that such $\alpha$ can be chosen to be any $\alpha \in (d-2 ,d)$ and thus be optimal in the sense that it corresponds to the constant coefficients case.} and $R>0$ such that for every $z\in \mathbb{Z}^d$
\begin{align}
\int_{|y-z|< 2R}\int_{|x-z|< 2R} |x-y|^{\alpha} |\nabla G_D(a; x, y)|^2 < +\infty,\label{WB} \\
\int_{|y-z|< R}\int_{|x-z|> 2R} |\nabla G_D(a; x, y)|^2 < +\infty, \label{WB2}
\end{align}
and for almost every $y\in \mathbb{R}^d$ the application $G(a, \cdot, y)$ satisfies
\begin{align}\label{TG}
\begin{cases}
 -\nabla\cdot a\nabla G_D(a;\cdot, y) =\delta(\cdot-y) \mbox{ \hspace{0.5cm} in $D$}\\
 \ \ G_D(a;\cdot , y) = 0  \hspace{2.76cm} \mbox{in $\partial D$},
\end{cases}
\end{align}
in the sense that $G_D(a; \cdot , y) = 0$  almost everywhere outside $D$ or vanishes at infinity for $D=\mathbb{R}^d$, and that for any $\xi\in Y$, $|\xi|=1$ it holds for every $\zeta\in C^\infty_0(D; Y)$
\begin{equation}\label{G}
\int \nabla\zeta(x)\cdot a(x)\nabla\bigl(G(a;x, y)\xi\bigr) = \zeta(y)\xi.
\end{equation}
We note that if we obtain estimates over $G(a;\cdot,y)\xi$, independent of $\xi$, then we automatically deduce the same bounds for $G(a; \cdot, y)$ itself.
Therefore, as long as we estimate uniformly in $\xi$, instead of (\ref{G}) we can adopt the formal notation
\begin{align}\label{E}
\int \nabla\zeta(x)\cdot a(x)\nabla G(a;x, y) = \zeta(y).
\end{align}
Given $G(a; \cdot, \cdot)$ as defined before, we may also consider $\nabla_y G_{D}(a; \cdot , y)$ which, for almost every $y\in \mathbb{R}^d$, is formally a solution (with the same understanding of (\ref{TG}) ) of 
\begin{align}\label{TGy}
\begin{cases}
 -\nabla\cdot a\nabla \nabla_y G_{D}(a; \cdot , y)= \nabla_y\delta(\cdot-y) \mbox{\ \ \hspace{0.8cm}  in D}\\
 \ \ \nabla_y G_{D}(a; \cdot , y) = 0  \ \mbox{ \hspace{3.3cm} in $\partial D$,}
\end{cases}
\end{align}
where the distribution $\nabla_y\delta(\cdot-y)$ acts on any $\zeta \in C^\infty_0 (D)$ as
\begin{align*}
\int \nabla_y\delta(x-y) \zeta(x) = \nabla \zeta (y).
\end{align*}
Throughout the rest of the paper we fix $Y=\mathbb{R}^m$ with the canonical inner product and use the previous scalar-like notation.
When no ambiguity occurs, we write $\nabla G_D$, $\nabla\nabla G_D$ for the gradient $\nabla_xG(a;x,y)$ and the mixed derivatives $\nabla_x\nabla_yG_D(a;x,y)$. In the case $D=\mathbb{R}^d$, we also use the notation $G(a; \cdot, \cdot)=G_{\mathbb{R}^d}(a; \cdot, \cdot)$. 
In the estimates carried out in this paper,  $\lesssim $ stands for $\leq C$ with a constant depending exclusively on the dimension d and the ellipticity ratio $\lambda$ and thus in particular independent of the choice of the domain $D$; 
similarly, $\lesssim_D $ stands for $\leq C$ with $C$ depending on $d$, $\lambda$ and also on the domain $D$: Unless stated otherwise, the dependence of $C$ on the domain is merely through the size of the smallest bounded direction of $D$.

\medskip

We remark that our definition of Green's function guarantees that for every coefficient field $a\in \Omega$ and for every open domain $D \subset \mathbb{R}^d$ with $d \geq 2$, $G_D(a; \cdot ,\cdot)$  is \textit{unique}. 
More precisely, we have the following
\begin{lemma}\label{U}
Let $a \in \Omega$ and let $D$ be an open domain in  $\mathbb{R}^d$, with $d \geq 2$. Then, $G_D(a, \cdot, \cdot)$ is unique (in the sense of $L^1_{loc}(\mathbb{R}^d \times \mathbb{R}^d))$. 
\end{lemma}
The proof of this result in the appendix is very similar to \cite{GO}, Subsection A.3, Step 4.

\bigskip

{\bf Random coefficient fields.}\ \ We restrict our definition of $\Omega$ as
\begin{align*}
\Omega:= \{\mbox{ measurable \ \ } a: \mathbb{R}^d\rightarrow \mathcal{L}(Y^d,Y^d), \mbox{ satisfying (\ref{Sym}), (\ref{KC}) and (\ref{bdd}) }  \},
\end{align*}
where the measurability is considered with respect to the coarsest $\sigma$-algebra $\mathcal{F}$ such that $\forall \xi\in Y^d$, $|\xi|=1$, the evaluation
\begin{align}\label{F}
a \rightarrow \int  ( a(x)\xi ) \chi(x) dx 
\end{align}
is measurable for every $\chi \in C^\infty_0(\mathbb{R}^d)$ ( where $\mathbb{R}$ is equipped with the usual Borel $\sigma$-algebra). 
We define a random coefficient field by endowing the couple $(\Omega, \mathcal{F})$ with a probability measure $\mathds{P}$, or equivalently by considering an {\it ensemble} $\langle\cdot\rangle$ over symmetric, uniformly elliptic coefficient fields $a$.
We assume the ensemble $\langle \cdot \rangle$ to be {\it stationary}, namely that $\forall z\in \mathbb{R}^d$ the coefficient fields $\{\mathbb{R}^d\ni x\rightarrow a(x)\}$ and $\{\mathbb{R}^d\ni x\rightarrow a(x+z)\}$ have the same distribution, and to be {\it stochastically continuous}, in the sense that for every $x\in \mathbb{R}^d$ and $\delta >0$
\begin{align*}
 \lim_{h\downarrow 0}\langle I[ |a(x+h)-a(x)| > \delta ] \rangle =0.
\end{align*}
This last condition ensures that the map $\Omega\times \mathbb{R}^d\ni (a,z)\ \rightarrow \ a(\cdot + z ) \in \Omega$ is  measurable  with respect to the product topology of $\Omega\times \mathbb{R}^d$.\\
With this additional structure, we can consider the random map $G: \Omega \ni a \rightarrow G(a; \cdot, \cdot)$ .
We also remark that, by definition (\ref{F}), $\mathcal{F}$ is countably generated and therefore for every $p \in [1; +\infty)$, the space
\begin{align*}
L^p( \Omega, \mathcal{F}, \mathds{P}):= \biggl\{ \zeta: \Omega \rightarrow \mathbb{R} \ \ \text{measurable}, \ : \ \langle \ |\zeta|^p \ \rangle^{\frac 1 p} := \mathbb{E}_{\mathds{P}}[\  |\zeta|^p \  ]^{\frac 1 p} < +\infty  \biggr\}
\end{align*}
is separable.

\section{Main result and remarks}
Throughout the paper, as a basic assumption, we consider the domain \mbox{$D \subset \mathbb{R}^d$} to be open and such that
\begin{align}\label{D}
|\partial D |= 0.
\end{align}
This condition basically ensures that if a function $u$ is zero almost everywhere outside $D$, and admits weak derivatives up to order k, then the derivatives are as well almost surely zero outside the domain: This will prove to be useful when defining
the approximate problem for (\ref{TG}), cf.(\ref{PDa}), which calls for a higher order operator and thus Dirichlet boundary conditions also for the derivatives.

In this paper we mainly provide two existence results for the Green function in a domain  $D \subset \mathbb{R}^d$ with $d \geq 2$. As introduced in Section \ref{Not}, for a given $a\in \Omega$, we treat the Green function for a domain $D$ as an object $G_D(a; \cdot, \cdot)$ in two space variables $(x, y) \in \mathbb{R}^d \times \mathbb{R}^d$, which satisfies for almost every singularity point $y\in \mathbb{R}^d$ the equation (\ref{TG}).
It is with this generalised definition of Green function that we manage to prove its existence and uniqueness (cf. Lemma \ref{U}) and also to obtain optimal estimates for the $L^2$-norm in both the space variables of $G_D(a; \cdot, \cdot)$, its gradient and its mixed derivatives.

In the first theorem, we show that if the domain $D$ is open, and if $d=2$ also bounded in at least one direction, then for every coefficient field $a\in\Omega$ the Green function $G_D(a; \cdot, \cdot)$ exists; 
In particular, this result also includes the existence of the fundamental solution, i.e. the Green function for $D= \mathbb{R}^d$, with $d>2$.
In the case of open domains bounded in at least one direction and strongly elliptic (cf. (\ref{StE}) coefficient fields $a\in \Omega$, we also provide in Corollary \ref{Cor1} an improvement of the estimates obtained in Theorem \ref{Teo1}, namely that the off-diagonal $L^2$-norms of $G_D(a; \cdot, \cdot)$, $G_D(a; \cdot, \cdot)$ and $\nabla\nabla G_D(a; \cdot, \cdot)$ decay exponentially in the unbounded directions.

Corollary \ref{Cor2} deals with the introduction of a stationary ensemble on the space of coefficient fields $\Omega$ and provides in the case of systems a generalisation, at the level of the first moments in probability, of the stochastic bounds obtained by Delmotte and Deuschel in \cite{dd} for the scalar case.

More precisely, we prove the following statements
\begin{theorem}{\label{Teo1}}
Let $d\geq 2$ and $D\subset \mathbb{R}^d$ be a general open domain satisfying (\ref{D}). Then
\begin{itemize}
\item[(a)] If $d > 2$, for every $a\in \Omega$ there exists the Green function $G_D(a;\cdot,\cdot)$ and it satisfies for every $z \in \mathbb{R}^d$, $R>0$, $\alpha > d-2$,
\begin{align}
\int_{|y-z|< R}& \int_{|x-z|< R}\hspace{-0.5cm} |x-y|^{\alpha}\bigl(|\nabla G_D(a; x, y)|^2 \notag\\
&\hspace{4cm}+  |\nabla_y G_D(a; x, y)|^2\bigr) \lesssim R^{2+\alpha},\label{Teo1B}\\
\int_{|y-z|< R}& \int_{|x-z| > 8R}\hspace{-0.5cm}|\nabla G_D(a; x, y)|^2 \lesssim R^{2},\label{Teo1A2}
\end{align}
and for every $z \in \mathbb{R}^d$, $R>0$ and $1 \leq p < \frac{d}{d-2}$ and $1 \leq q < \frac{d}{d-1}$
\begin{align}
\int_{|y-z|< R}& \int_{|x-z|< R} |G_D(a; x, y)|^p \lesssim R^{(2-p)d + 2p},\label{Teo1C}\\
\int_{|y-z|< R} &\int_{|x-z|< R} \hspace{-1cm}|\nabla G_D(a; x, y)|^q + |\nabla_yG_D(a; x, y)|^q \lesssim R^{(2-q)d + q},\label{Teo1D}\\
\int_{|y-z|< R}& \int_{|x-z|> 2R} |\nabla\nabla G_D(a; x, y)|^2 \lesssim 1. \label{Teo1A}
\end{align}
\item[(b)] If $d=2$ and $D$ is bounded in at least one direction, then for every $a\in \Omega$ the Green function $G_D(a; \cdot, \cdot)$ exists as well and satisfies the bounds (\ref{Teo1B})-(\ref{Teo1A2}), (\ref{Teo1A}) and (\ref{Teo1C})-(\ref{Teo1D}). All the constants,
with the exception of (\ref{Teo1A}), depend also on the size of the smallest bounded direction of $D$ and the bound (\ref{Teo1A2}) holds for radii $R \gtrsim_D 1$.
\end{itemize}
\end{theorem}
\begin{corollary}\label{Cor1}
Let $D \subset \mathbb{R}^d$ with $d \geq 2$ be an open domain satisfying (\ref{D}) and bounded in at least one direction. Then for every $a\in \Omega$ satisfying also (\ref{StE}), there exists a constant $C_1 \lesssim_D 1$ such that 
for every $z\in \mathbb{R}^d$ and $R \gtrsim_D 1$ it holds
\begin{align}
\int_{|y-z|< R}\int_{|x-z|> 4R} &|\nabla\nabla G_{D}(a; x,y)|^2 \lesssim \exp{(-\frac{R}{C_1})},\label{Teo1F}\\
\int_{|y-z|< R}\int_{|x-z|> 4R} & |\nabla G_{ D}(a; x,y)|^2 + |G_{ D}(a; x,y)|^2 \lesssim_D \exp{(-\frac{R}{C_1})}.\label{Teo1E}
\end{align}
\end{corollary}
\begin{corollary}{\label{Cor2}}
Let $d > 2$ and let $\langle \cdot \rangle$ be a stationary ensemble on $\Omega$. 
Then, the Green function $G(a; \cdot, \cdot)$ for the whole space $\mathbb{R}^d$ satisfies for almost every $ x,y \in \mathbb{R}^d$ the annealed pointwise bounds
\begin{align}
\langle |G(\cdot\ &;x,y)|\rangle \lesssim |x-y|^{2-d},\label{C2A}\\
\langle |\nabla G(& \cdot \ ;x,y)|\rangle\lesssim |x-y|^{1-d},\label{C2B}\\
\langle |\nabla\nabla G&(\cdot \ ;x,y)|\rangle\lesssim |x-y|^{-d}.\label{C2C}
\end{align}
\end{corollary}
We recall that in all the previous inequalities $\lesssim$ and $\lesssim_D$ stand for $\leq C$ with the constant C respectively depending on $d$ and $\lambda$ or on $d$,$\lambda$ and the size of the smallest bounded component of $D$.

\medskip 

In the following remark we argue that the bounds (\ref{C2B}) and (\ref{C2C}) require the expectation $\langle \cdot \rangle$:
\begin{remark}\label{R.1}
\begin{itemize}
\item[(i)] For $d > 2$, a coefficient field $a \in \Omega$ and an associated Green's function $G(a; \cdot ,y)$ on $\mathbb{R}^d$, the bound 
\begin{align}\label{R0}
|\nabla G(a; x, y)| \lesssim |x-y|^{1-d} \ \ \ \text{for a. e. $x, y \in \mathbb{R}^d$,}
\end{align}
implies that any finite energy a-harmonic application u is (locally) bounded. More precisely, the local boundedness of a-harmonic applications is also implied if assuming instead of (\ref{R0})
the weaker $L^2$- bound 
\begin{align}\label{R1}
\int_{|x-y| > R}|\nabla G(a; x, y)|^{2} \lesssim R^{2-d} \ \ \text{for every $R>0$ and a.e. $y\in\mathbb{R}^d$.}
\end{align}
While in the scalar case the bound (\ref{R1}) holds (\cite{go1}, Lemma 2.9), we cannot expect it to be true for every coefficient field $a\in \Omega$ in the case of systems. The following example of De Giorgi \cite{DeG} shows indeed
that in $d>2$, the unbounded function $u:\mathbb{R}^d\rightarrow Y$, with $Y=\mathbb{R}^d$, given by
\begin{align}\label{r.7}
u(x)=\frac{x}{|x|^{\gamma}}, \text{\ \ \ $\gamma:=\frac{d}{2}\biggl(1-\frac{1}{\sqrt{(2d-2)^2+1}}\biggr)> 1$}
\end{align}
is locally of finite energy and a-harmonic for the symmetric and elliptic coefficient field
\begin{align}\label{r.9}
\xi\cdot a(x)\xi= \xi:\xi + \bigl((d-2)Tr(\xi)+ d(\frac{x}{|x|}\cdot\xi\frac{x}{|x|})\bigr), \ \ \ \xi\in \mathbb{R}^{d\times d}.
\end{align}
\item[(ii)] Assuming that both (\ref{R1}) and
\begin{align}\label{R2}
\int_{|x-y| > R}|\nabla\nabla G(a; x, y)|^{2} \lesssim R^{-d} \text{\ for every $R>0$ and a.e.  $y\in \mathbb{R}^d$}
\end{align}
hold, implies that any a-harmonic application u is also locally Lipschitz. Hence, also in the scalar case both conditions (\ref{R1}) and (\ref{R2}) cannot be true for every coefficient field $a\in \Omega$ (\cite{PS}, Example 3).
\item[(iii)]  For $\alpha \in (0,1)$, even a suboptimal assumption on the decay of (\ref{R2}) as 
\begin{equation}\label{R3}
\int_{|x-y| > R}|\nabla\nabla G(a; x, y)|^{2}  \lesssim R^{-d +\alpha},
\end{equation}
cannot hold for every coefficient field $a\in \Omega$ both in the scalar and in the systems' case. Indeed, (\ref{R3}) implies a Liouville property for a-harmonic functions, namely that for $\beta\in (\frac{\alpha}{2}, 1)$ for any a-harmonic u on $\mathbb{R}^d$
\begin{align}\label{L}
\lim_{R\rightarrow +\infty}R^{-1+\beta}\bigl(\fint_{|x|<R}|u|^2 \bigr)^{\frac{1}{2}}=0 \ \ \Leftrightarrow \text{\ \ \  u is constant},
\end{align}
where
\begin{align*}
\fint_{|x| < R} |u|^2 = \frac{1}{\bigl|\{|x|< R \} \bigr|} \int_{|x|<R}|u|^2 \simeq R^{-d} \int_{|x|<R}|u|^2.
\end{align*}
It can be shown that in the scalar case (\ref{L}) doesn't even hold for uniformly elliptic and smooth coefficient fields (\cite{FO}, Proposition 21): For every $\varepsilon >0$, there exists indeed a smooth $a\in \Omega$ and an a-harmonic function u such that $\bigl(\fint_{|x|<R}|u|^2 \bigr)^{\frac{1}{2}} \lesssim R^{\varepsilon}$.
Moreover, in the case of systems, De Giorgi's example (\ref{r.7}) shows that a-harmonic functions can also (non trivially) vanish at infinity.
\end{itemize}
We postpone the proofs of (i), (ii) and (iii) to the Appendix. 
\end{remark}

\bigskip

This paper is organised as follows: In Section \ref{ST1} we give the argument for \mbox{Theorem \ref{Teo1}}, part (a) and (b).
The core of the proof for part (a), i.e. when $d>2$, is an $L^2$-off-diagonal bound for $\nabla\nabla G_D$ and $\nabla G_D$, in both space variables $x$ and $y$
and depending only on the dimension and the ellipticity ratio. It is mainly obtained through a duality argument \`a la Avellaneda-Lin (\cite{AL1}, Theorem 13) on standard energy estimates for solutions of $-\nabla \cdot a\nabla u = \nabla \cdot g$, combined with
an inner-regularity estimate for a-harmonic functions in the spirit of Lemma 4 of \cite{BGO}. We stress here that this result is inspired by Lemma 2 of \cite{BGO} and provides the new and pivotal ingredient for the first fundamental estimate
for $G_D$. This may be considered as the key ingredient for the whole argument of Theorem \ref{Teo1}.
Sobolev's inequality allows to extend the previous estimates also for $\nabla_y G_D$ and $G_D$. Finally, with the aid of rescaling and dyadic decomposition arguments, from the off-diagonal estimate on $\nabla\nabla G_D$ 
we also infer bounds for $\nabla G_D$ and $G_D$ close to the singularity $x=y$.

In dimension $d=2$, it is well known that the fundamental solution, i.e. the Green function for $D= \mathbb{R}^d$, does not exist. We indeed restrict our attention to domains $D$ which have at least one bounded direction: By 
substituting the scale-invariant Sobolev's inequality, which holds only for $d> 2$, with Poincar\'e-type inequalities, we may extend the arguments of part (a) to the two-dimensional case. 
We point out that the appeal to Poincar\'e-type inequalities introduces in the estimates for $G_D$ and its derivatives a dependence also on the minimal bounded direction of $D$.

We stress that our assumptions on $\Omega$ include in this set also very rough coefficient fields for which the existence of $G_D$ is not \textit{a priori} guaranteed. Therefore, we need to first approximate the problem (\ref{TG}),
carry out and adapt to the approximate solution the aforementioned a priori bounds on $G_D$, and then argue by standard weak-compactness of $W^{1,q}_{loc}$- spaces.
We approximate (\ref{TG}) through an $\varepsilon$-perturbation of the operator $-\nabla\cdot a \nabla$ with the hyper-elliptic term\footnote{This choice of $\mathcal{L}_n$ instead of the more standard $(-\Delta)^{n}$ will prove to
be more convenient when testing the equation and integrating by parts, as it avoids mixed derivatives and thus simplifies the calculations.}
\begin{align}\label{LL}
\mathcal{L}_n:= \sum_{i=1}^d (-\partial^2_i)^{n},
\end{align}
and thus consider for $\varepsilon > 0$, $a\in \Omega$ and $y\in \mathbb{R}^d$ the problem
\begin{align}\label{PDa}
\begin{cases}
 -\nabla\cdot a\nabla G_{\varepsilon,D}(a; \cdot, y) + \varepsilon\mathcal{L}_n G_{\varepsilon,D}(a; \cdot, y)=\delta(\cdot-y) \mbox{\ \ \hspace{0.2cm}  in D}\\
 \ \ G_{\varepsilon,D}(a; \cdot ,y) = 0  \mbox{ \hspace{5.2cm} in $\partial D$.}
\end{cases}
\end{align}
The assumption (\ref{D}) on the domain D and our understanding of the boundary conditions, i.e. that $G_{\varepsilon, D}$ vanishes almost surely outside $D$ or vanishes at infinity for $D=\mathbb{R}^d$, imply that the same boundary conditions hold also for the higher-order derivatives up to index $n-1$. For $D=\mathbb{R}^d$, the Dirichlet conditions on the boundary turn into the requirement for every $\partial^{\alpha}u$ with $0\leq |\alpha| \leq n-1$, to vanish at infinity.
For $n > \frac d 2$, Riesz's representation theorem ensures the existence of a unique weak solution $G_{\varepsilon,D}$ for every $\varepsilon>0$, $a\in \Omega$ and singularity point $y\in \mathbb{R}^d$. Moreover, assuming $n > \frac d 2 +1$ also implies that there exists a unique $\nabla_yG_{\varepsilon, D}(a; \cdot, y)$ which solves the approximate problem for (\ref{TGy})
\begin{align}\label{PDay}
\begin{cases}
 -\nabla\cdot a\nabla \nabla_y G_{\varepsilon, D}(a; \cdot , y) + \varepsilon\mathcal{L}_n\nabla_y G_D(a; \cdot , y)= \nabla_y\delta(\cdot-y) \mbox{\ \ \hspace{0.2cm}  in D}\\
 \nabla_y G_{\varepsilon, D}(a; \cdot , y) = 0  \ \mbox{\hspace{6.1cm} in $\partial D$.}
\end{cases}
\end{align}

In Section \ref{Corollaries} we provide the proof of Corollary \ref{Cor1} and Corollary \ref{Cor2}; In the first corollary we show that in the case of domains with at least one bounded direction and strongly elliptic coefficient fields
 we improve the estimates of Theorem \ref{Teo1} again by a duality argument which this time relies on a refinement of the standard energy estimate for solutions of $-\nabla \cdot a\nabla u = \nabla \cdot g$ in domains which have a bounded direction.
While the arguments of Theorem \ref{Teo1} and Corollary \ref{Cor1} are purely deterministic, in Corollary \ref{Cor2} we introduce a stationary ensemble $\langle \cdot \rangle$ on $\Omega$ and focus our attention on the fundamental solution $G$ in $d>2$ seen as a random map. The stationarity assumption on $\langle \cdot \rangle$ provides an improvement of the estimates on $G$ by upgrading the bounds of Theorem \ref{Teo1} from space-averaged in both variables $x$ and $y$ to annealed in $a$ but pointwise in $y$. An a priori estimate for locally a-harmonic functions allows us to conclude the argument and obtain estimates averaged in a, but pointwise in x and y.

In the last section we present an alternative partial proof for Corollary \ref{Cor2} which makes use of the Fourier techniques developed in \cite{cn2} and relies on a representation formula for the Fourier transform of the 
Green function. Finally, in the Appendix we give a self-contained proof of all the auxiliary results which are used in the arguments.
\section{Proof of Theorem 1}\label{ST1}
Let $a\in \Omega$ and $D \subset \mathbb{R}^d$ be a generic open domain satisfying (\ref{D}), with $d \geq 2$. For a fixed $y\in \mathbb{R}^d$ and $\varepsilon>0$, we consider the approximate problem for (\ref{TG}) introduced in (\ref{PDa}), i.e.
\begin{align*}
\begin{cases}
 -\nabla\cdot a\nabla G_{\varepsilon,D}(a; \cdot, y) + \varepsilon\mathcal{L}_n G_{\varepsilon,D}(a; \cdot, y)=\delta(\cdot-y) \mbox{\ \ in $D$}\\
 \ \ G_{\varepsilon,D}(a; \cdot, y) = 0 \ \mbox{ \hspace{4.75cm} in $\partial D$,}
\end{cases}
\end{align*}
where $\mathcal{L}_n$ is as in definition (\ref{LL}) and with $n$ a fixed odd integer such that $n > \frac d 2 +1$.\\
\begin{definition}\label{R.6} Let $R>0$, and $g \in [ L^{2}(\{ |x| < R \}) ]^d$. We say that $u \in L^1_{loc}(\mathbb{R}^d)$ is a weak solution of
\begin{align*}
\begin{cases}
 -\nabla\cdot a\nabla u + \varepsilon\mathcal{L}_n u= \nabla \cdot g \mbox{\ \ \ \hspace{0.18cm} in $\{|x| < R\} \cap D$}\\
 \ \ u( \cdot) = 0  \ \mbox{ \hspace{2.75cm} in $\{|x| < R\} \cap \partial D$,}
\end{cases}
\end{align*}
if
\begin{itemize}
\item[(i)] for all $i=1, ... , d$, there exist weak $\partial_i u, \partial_i^n u \in L^2(|x| < R)$;
\item[(ii)] $u = 0$ almost everywhere in $\{|x| < R\}\setminus D$;
\item[(iii)] for all $v$ satisfying (i) and compactly supported in $\{|x| < R\}$, it holds
\begin{align*}
\int \nabla v\cdot a \nabla u + \varepsilon \sum_{i=1}^d \int \partial_i^nv \ \partial_i^n u =  -\int \nabla v \cdot g.
\end{align*}
\end{itemize}
\end{definition}
Analogously we may consider solutions on outer domains by substituting in Definition \ref{R.6} the domain $\{|x| < R \}$ with $\{|x| > R\}$.

\medskip

We start with two variants of Lemma 4 of \cite{BGO}; while this last one is a statement for ensembles of locally a-harmonic functions, the following Lemma \ref{BGO2} takes into account the new perturbation term $\mathcal{L}_n$ and the more general domain $D$.
If $d >2$, then Lemma \ref{BGO3} is a further generalisation to the case of functions solving $-\nabla \cdot a \nabla u + \varepsilon \mathcal{L}_n u=0$ on outer domains. We postpone the proofs of Lemma \ref{BGO2} and Lemma \ref{BGO3} to the Appendix.
\begin{lemma}\label{BGO2}
For a radius $R > 0$ and $a\in \Omega$, we consider a $\sigma$-finite measure $\mu$ on functions $u$ satisfying in the sense of Definition \ref{R.6}
\begin{align}\label{Ex.39a}
\begin{cases}
 -\nabla\cdot a\nabla u + \varepsilon\mathcal{L}_n u=0 \mbox{\ \ \ \hspace{0.35cm} in $\{|x| < 2R\}\cap D$}\\
 \ \ u = 0  \ \ \mbox{\hspace{2.76cm} in $\{|x| < 2R\}\cap \partial D$,}
\end{cases}
\end{align}
with $\varepsilon \geq 0$.
Then we have
\begin{equation}\label{b210a}
\int \int_{|x|<R}|\nabla u|^2 \dx \,\mathrm{d}\mu \lesssim
\sup_{F}{\int |F u|^2 \mathrm{d}\mu},
\end{equation}
where the supremum runs over all linear functionals $F$ bounded in the sense of
\begin{equation}\label{b29a}
|Fv|^2\leq \int_{|x|\le 2R}|\nabla v|^2 \dx,
\end{equation}
with v satisfying (i) and (ii) of Definition \ref{R.6}.
\end{lemma}
\begin{lemma}\label{BGO3}
Let $d>2$. For a radius $R > 0$ and $a\in \Omega$, we consider a \mbox{$\sigma$-finite} measure $\mu$ on functions $u$ with finite Dirichlet energy in $\{|x| >R\}$ and satisfying in the sense of Definition \ref{R.6}
\begin{align}\label{Ex.39}
\begin{cases}
 -\nabla\cdot a\nabla u + \varepsilon\mathcal{L}_n u=0 \mbox{\ \ \ \hspace{0.35cm} in $\{|x| > R\}\cap D$}\\
 \ \ u = 0  \ \ \mbox{\hspace{3cm} in $\{|x| > R\}\cap \partial D$,}
\end{cases}
\end{align}
with $\varepsilon \geq 0$.
Then we have
\begin{equation}\label{b210}
\int \int_{|x|> 4R}|\nabla u|^2 \dx\, \mathrm{d}\mu \lesssim
\sup_{F}{\int |F u|^2 \mathrm{d}\mu},
\end{equation}
where the supremum runs over all linear functionals $F$ bounded in the sense of
\begin{equation}\label{b29}
|Fv|^2\leq \int_{|x| > R}|\nabla v|^2 \dx,
\end{equation}
with v satisfying (i) and (ii) in the sense Definition \ref{R.6} where the set $\{|x| < R\}$ is substituted by $\{|x| > R \}$.
\end{lemma}
Analogously to Theorem \ref{Teo1}, $\lesssim$ means $\le C$ with a generic $C=C(d,\lambda)$.

\medskip

We remark  that the inequalities
\begin{align*}
\int \int_{|x|<R}|\nabla u|^2 \dx\, \mathrm{d}\mu \lesssim
{\int \bigl( \sup_{F} |F u|^2 \bigr)\mathrm{d}\mu},\\
\int \int_{|x| > 4R }|\nabla u|^2 \dx\, \mathrm{d}\mu \lesssim
{\int \bigl( \sup_{F} |F u|^2 \bigr)\mathrm{d}\mu},
\end{align*}
hold trivially by conditions (\ref{b29a}) and (\ref{b29}) and duality (in $L^2$). Roughly speaking, Lemma \ref{BGO2} and Lemma \ref{BGO3} state that the previous inequalities remain true also if we exchange in the r.h.s. the order of the integration in $\mu$ and the supremum over the functionals $F$.

\medskip

We may refer to the result of Lemma 4 of \cite{BGO}, which corresponds to Lemma \ref{BGO2} with $D= \mathbb{R}^d$ and $\varepsilon=0$, as a \textit{compactness} statement for ensembles of locally a-harmonic functions. Indeed, as we show in the appendix, inequality (\ref{b210a}) actually follows by an inner regularity estimate which allows to control the energy of an a-harmonic function u in an interior domain by the $L^2$-norm on $\{|x| < 2R\}$ of $(-\Delta_N)^{-\frac l 2} u$ for any even $l\in \mathbb{N}$. Here, $-\Delta_N$ denotes the Laplacian with Neumann boundary conditions. This last estimate basically implies that in the space of locally a-harmonic functions, the local $W^{1,2}$-norm (the strongest norm which is meaningful to consider for weak solutions of a variable-coefficient and uniformly elliptic second-order operator) is actually equivalent to much weaker norms, provided we consider a slightly bigger domain. Therefore, in this sense we may say that the space of !
 locally a-harmonic functions is ``close'' to being finite-dimensional, in which case all the norms are equivalent.
The previous lemmas state similar compactness results in the case of the approximate operator $-\nabla \cdot a \nabla + \varepsilon \mathcal{L}_n$.

\bigskip

{\sc Proof of Theorem \ref{Teo1}.} Throughout the whole proof we assume $D \subset \mathbb{R}^d$ to be a generic open domain satisfying (\ref{D}) which is  also bounded in at least one direction if $d=2$.

\medskip

{{ \sc Step} 1: Construction of the approximate family $\{ G_{\varepsilon, D}(a; \cdot ,y)\}_{\varepsilon \downarrow 0}$}. 
We start by showing that for every $a \in \Omega$, $D \subset \mathbb{R}^d$, $\varepsilon >0$ and $y\in \mathbb{R}^d$ there exist $G_{\varepsilon,D}(a, \cdot, y)$ and $\nabla_yG_{\varepsilon}(a; \cdot, y)$, unique weak solutions\footnote{We give the precise definition along the proof of this step.}
 respectively for the problem (\ref{PDa}) and (\ref{PDay}).
 
Given the Hilbert spaces
\begin{align*}
\mathcal{X}:= \{u\in L^{2^*}(\mathbb{R}^d) \ &\ | \ \ \nabla u, \partial_i^nu \in L^2(\mathbb{R}^d), \ i=1, ..., d 
\ \ \text{and \ }   u = 0 \mbox{ \ outside $ D$}\}
\end{align*}
for $d > 2$ and 
\begin{align*}
\mathcal{X}:= \{u\in H^n(\mathbb{R}^d) \ &\ | \ \  u = 0 \mbox{ \ outside $ D$}\}
\end{align*}
for $d=2$, the bilinear form $B: \mathcal{X} \times \mathcal{X} \rightarrow \mathbb{R}$
\begin{align*}
B(u,v)= \int \nabla{v}\cdot a\nabla{u} + \varepsilon \sum_{i=1}^d\int \partial_i^n v \, \partial_i^{n}u
\end{align*}
is bounded thanks to (\ref{bdd}) and coercive in the sense of 
\begin{align}\label{Ex.8a}
B(u,u)\gtrsim \varepsilon \sum_{i=1}^d\int_D |\partial_i^n u|^2+ \lambda\int_D |\nabla{u}|^2,
\end{align}
thanks to (\ref{KC}).
\medskip

Let us first consider the case $d>2$: Sobolev's embedding implies that
\begin{align}\label{Ex.8d}
B(u,u) = 0 \Leftrightarrow u = 0,
\end{align}
i.e. $B$ is non-degenerate. We now argue that $B$ satisfies for every $u \in \mathcal{X}$
\begin{align}\label{Ex.8}
B(u,u)\gtrsim \varepsilon \bigl(\sup_D |u|\bigr) ^2.
\end{align}
Thanks to the coercivity condition (\ref{Ex.8a}), inequality (\ref{Ex.8}) is implied by
\begin{align*}
\lambda\int_D |\nabla{u}|^2+ \sum_{i=1}^d\int_D |\partial_i^n u|^2 \gtrsim \biggl(\sup_D |u|\biggr) ^2,
\end{align*}
which can be restated by passing to dual (i.e. Fourier) variables $k$ as
\begin{align}\label{Ex.9}
\int (|k|^2+ \sum_{i=1}^d|k_i|^{2n} ) |\hat{u}|^2 dk \gtrsim \biggl(\int |\hat{u }|dk \biggr)^2.
\end{align}
By Cauchy-Schwarz's inequality, it holds
\begin{align}\label{Ex.9a}
\int|\hat{u}|dk  &\leq \left(\int\frac{dk}{|k|^2+ \sum_{i=1}^d|k_i|^{2n}}\right)^\frac{1}{2}\left(\int(|k|^2+ \sum_{i=1}^d|k_i|^{2n})|\hat{u}|^2dk\right)^\frac{1}{2}\notag\\
 &\stackrel{ }{\lesssim} \left(\int\frac{dk}{|k|^2+ |k|^{2n}}\right)^\frac{1}{2}\left(\int(|k|^2+ \sum_{i=1}^d|k_i|^{2n})|\hat{u}|^2dk\right)^\frac{1}{2}.
\end{align}
As the conditions $n > \frac d 2 + 1 > \frac d 2$ and $d > 2$ imply that 
\begin{align*}
\int\frac{dk}{|k|^2+ |k|^{2n}} < +\infty,
\end{align*}
from (\ref{Ex.9a}) we obtain (\ref{Ex.9}) and thus (\ref{Ex.8}). Inequality (\ref{Ex.8}) in particular yields that for the linear functional $F v := v(y)$, we have for every $u \in \mathcal{X}$
\begin{align*}
B(u,u) \stackrel{(\ref{Ex.8})}{\gtrsim} \varepsilon \bigl(\sup_D |u|\bigr) ^2 \gtrsim \varepsilon |F u|^2,
\end{align*}
which implies by Riesz's representation theorem that there exists a unique $G_{D,\varepsilon}(a; \cdot , y) \in \mathcal{X}$ weakly solving (\ref{PDa}). 
As we have shown that for every $a\in \Omega$, the map $D \ni y \rightarrow G_{\varepsilon, D}(a, \cdot ,y)$ is well defined, we now also show that for every $a \in \Omega$ and $y\in D$ there exists  $\nabla_y G_{\varepsilon, D}(a; \cdot , y)$, unique (weak) solution\footnote{The fact
that the solution of (\ref{PDay}) is actually the $y$-gradient of $G_D(a; \cdot, \cdot)$ is rigorously proven by first showing that, on the one hand, by symmetry (\ref{Ex.28}), the difference quotients $\frac{1}{h} \bigl( G_{D,\varepsilon}(a; \cdot , \cdot+he_i)- G_{D,\varepsilon}(a; \cdot , \cdot+he_i)\bigr)$
are uniformly bounded in $\mathcal{X}$ for $h <<1$, and thus weakly converge up to subsequences. On the other hand, letting $ h \rightarrow 0^+$ in the equation solved by $\frac{1}{h} \bigl( G_{D,\varepsilon}(a; \cdot , y+he_i)- G_{D,\varepsilon}(a; \cdot , y+he_i)\bigr)$, we recover (\ref{PDay}) and conclude the argument by uniqueness of the solution.} of (\ref{PDay}).
We appeal again to Riesz's representation theorem : In this case, the linear functional that we need to bound
with $B$ is given by $\tilde F( v ) := \nabla v(y)$. Once again, thanks to the coercivity condition (\ref{Ex.8a}) we conclude the argument if we show that
\begin{align}\label{Ex.8c}
\lambda\int_D |\nabla u|^2+ \sum_{i=1}^d\int_D |\partial_i^n u|^2 \gtrsim \biggl(\sup_D |\nabla u|\biggr) ^2,
\end{align}
or equivalently, by passing in Fourier variables, that
\begin{align*}
\int (|k|^2+ \sum_{i=1}^d|k_i|^{2n} ) |\hat{u}|^2 dk \gtrsim \biggl( \int |k||\hat{u }|dk \biggr)^2.
\end{align*}
Similarly to (\ref{Ex.9a}), we estimate
\begin{align*}
\int|k||\hat{u}|dk  &\leq \left(\int\frac{|k|^2}{|k|^2+ \sum_{i=1}^d|k_i|^{2n}}dk\right)^\frac{1}{2}\left(\int(|k|^2+ \sum_{i=1}^d|k_i|^{2n})|\hat{u}|^2dk\right)^\frac{1}{2}\notag\\
 &\stackrel{ }{\lesssim} \left(\int\frac{dk}{1+ |k|^{2(n-1)}}\right)^\frac{1}{2}\left(\int(|k|^2+ \sum_{i=1}^d|k_i|^{2n})|\hat{u}|^2dk\right)^\frac{1}{2}
\end{align*}
and this time we appeal to $ n - 1 > \frac d 2$ to ensure that
\begin{align*}
\int\frac{dk}{1+ |k|^{2(n-1)}} < +\infty.
\end{align*}

\medskip

Let us now assume that $d=2$: Also in this case, by (\ref{bdd}) and (\ref{KC}),
B is respectively bounded and coercive in the sense of (\ref{Ex.8a}). Our assumption on $D$ and the Dirichlet boundary conditions allow us to appeal to Poincar\'e's inequality
\begin{align}\label{Poi}
\int |u|^2 \lesssim_D \int |\nabla u|^2,
\end{align}
and infer that $B$ is non degenerate in the sense of (\ref{Ex.8d}). We note that once we prove that  also in this case $B$ satisfies inequalities (\ref{Ex.8}) and (\ref{Ex.8c}), we may argue analogously to the case $d>2$ and 
conclude that there exist unique $G_{\varepsilon}(a, \cdot, y)$ and $\nabla_yG_{\varepsilon}(a, \cdot, y)$ solving in $D$ respectively (\ref{PDa}) and (\ref{PDay}) . The argument used above for (\ref{Ex.8c}) is still valid; to show (\ref{Ex.8}), we observe that by Poincar\'e's inequality and (\ref{Ex.8a}) it is enough to prove that
\begin{align*}
\int |u|^2 +  \sum_{i=1}^d\int_D |\partial_i^n u|^2 \gtrsim \biggl(\sup_D | u|\biggr) ^2.
\end{align*}
We rewrite the previous inequality in Fourier variables as
\begin{align*}
\int (1+ \sum_{i=1}^d|k_i|^{2n} ) |\hat{u}|^2 dk \gtrsim \biggl( \int |\hat{u }|dk \biggr)^2
\end{align*}
and estimate
\begin{align*}
\int|\hat{u}|dk  &\leq \left(\int\frac{dk}{(1+ \sum_{i=1}^d|k_i|^{2n})}\right)^\frac{1}{2}\left(\int(1+ \sum_{i=1}^d|k_i|^{2n})|\hat{u}|^2dk\right)^\frac{1}{2}\notag\\
 &\stackrel{ }{\lesssim} \left(\int\frac{dk}{(1+ |k|^{2n})}\right)^\frac{1}{2}\left(\int(1 + \sum_{i=1}^d|k_i|^{2n})|\hat{u}|^2dk\right)^\frac{1}{2}.
\end{align*}
Relying on our assumption $n > \frac d 2 +1 > \frac d 2 $, we have that
\begin{align*}
\int\frac{dk}{1+ |k|^{2n}} < +\infty,
\end{align*}
and thus we infer (\ref{Ex.8}).

\medskip

In addition, we note that uniqueness and the symmetry of the operator $- \nabla\cdot \nolinebreak a \nabla + \varepsilon \mathcal{L}_n$, cf. (\ref{Sym}), yield for all $a\in \Omega$, $z\in \mathbb{R}^d$, $R> 0$, $y \in \mathbb{R}^d$ and almost every $x \in \mathbb{R}^d$ that
\begin{align}
G_{\varepsilon,D}(a;x,y)&=G_{\varepsilon, D}(a;y,x),\label{Ex.28}\\
G_{\varepsilon, D}(a(\cdot+z);x,y)&=G_{\varepsilon, D+z}(a;x+z,y+z),\label{Ex.29}\\ 
G_{\varepsilon, D} (a; R {x}, R{y}) &= R^{2-d}G_{R^{2-2n}\varepsilon,R^{-1}D}(a(R \, \cdot ); {x},  y). \label{Ex.24}
\end{align}
Moreover, we claim that for every compactly supported $f \in L^2(\mathbb{R}^d)$ and every $g \in L^{2}( \mathbb{R}^d)^d$, if $u \in \mathcal{X}$ solves 
\begin{align}\label{approx.pb}
\begin{cases}
-\nabla\cdot a \nabla u + \varepsilon \mathcal{L}_n u = f + \nabla \cdot g \ \ \ \ \mbox{ in $D$}\\
\ \ u= 0 \ \ \ \ \mbox{\hspace{3.5cm} in $\partial D$,}
\end{cases}
\end{align}
in the sense of Definition 1, then we have the representation formula 
\begin{align}\label{representation.f}
u(x) = \int G_{\varepsilon, D}(a; y, x) f(y) \, dy - \int \nabla_y G_{\varepsilon, D}(a; y, x) \cdot g(y) \, dy.
\end{align}
By H\"older's inequality it is immediate to show that for the linear functional $F : \mathcal{X}\ni v \mapsto \int \nabla v \cdot g$ we have $|F v|^2 \lesssim ||g||_{L^2}^2 B(v, v)$. Sobolev's inequality for $d> 2$ and Poincar\'e's inequality for $d=2$ imply 
that also $F: \mathcal{X} \ni v \mapsto \int v f $ satisfies $|F v|^2 \lesssim ||f||_{L^2}^2 B(u, u)$. Therefore, by Riesz's representation thoerem, there exists a unique solution in $\mathcal{X}$ of \eqref{approx.pb}.
It thus remains to show that the r.h.s. of \eqref{representation.f} solves the equation: An application of H\"older's inequality, together with  the assumptions on $f$, $g$ and the fact that 
$G_{\varepsilon, D}(a; \cdot , y) \in \mathcal{X}$, implies that $u$ in \eqref{representation.f} is well-defined and belongs to $\mathcal{X}$. By \eqref{PDa} and \eqref{PDay}, $u$ satisfies the boundary conditions and for every $v \in \mathcal{X}$ we have
\begin{align*}
\int& \nabla_x v(x) \cdot  a \nabla_x \bigl( \int G_{\varepsilon, D}(a; y, x) f(y) - \int \nabla_y G_{\varepsilon, D}(a; y, x) \cdot g(y) \bigr)\\
& \quad +  \varepsilon \sum_{i=1}^d\int \partial_{x_i}^m v(x) \, \partial_{x_i}^m\bigl(  \int G_{\varepsilon, D}(a; y, x) f(y) - \int \nabla_y G_{\varepsilon, D}(a; y, x) \cdot g(y) \bigr)\\
&\stackrel{\eqref{Ex.28}}{=} \int \biggl( \int \nabla_x v(x) \cdot a \nabla_x G_{\varepsilon, D}(a; x, y) + \varepsilon \sum_{i=1}^d \int \partial_{x_i}^m v(x) \, \partial_{x_i}^m G_{\varepsilon, D}(a; x, y) \biggr) f(y)  \\
&\quad - \int \biggl( \int \nabla_x v(x) \cdot a \nabla_x\nabla_y G_{\varepsilon, D}(a; x, y) + \varepsilon \sum_{i=1}^d \int \partial_{x_i}^m v(x) \, \partial_{x_i}^m\nabla_y G_{\varepsilon, D}(a; x, y) \biggr) \cdot g(y)\\
&\stackrel{\eqref{PDa}- \eqref{PDay}}{=} \int v(y) f(y) - \int \nabla v(y) \cdot g(y).
\end{align*} 
We thus established \eqref{representation.f}.
\bigskip

{{\sc Step 2}: Uniform bounds for $\{G_{\varepsilon,D}\}_{\varepsilon>0}$ if $d>2$.} We presently argue that the family $\{G_{\varepsilon, D}(a, \cdot, \cdot)\}$ constructed in the previous step satisfies
(\ref{Teo1B})-(\ref{Teo1A2}), and (\ref{Teo1B})-(\ref{Teo1C})-(\ref{Teo1D}). By the properties (\ref{Ex.29})-(\ref{Ex.24}), without loss of generality it is sufficient to fix $z=0$ and $R=1$, i.e. to prove that for all
$\alpha > d-2$,
\begin{align}
\int_{ |y | < 1}\int_{|x|>8}&|\nabla G_{\varepsilon, D}(a; x,y)|^2 \lesssim 1, \label{Ex.7h}\\
\int_{ |y | < 1}\int_{|x| < 1}&|x-y|^{\alpha}\bigl( |\nabla_y G_{\varepsilon,D}(a; x,y)|^2 +|\nabla G_{\varepsilon,D}(a; x,y)|^2 \bigr)\lesssim 1,  \label{Ex.7f}
\end{align}
and that for every $1 \leq p < \frac{d}{d-2}$ and $1 \leq q < \frac{d}{d-1}$
\begin{align}
\int_{|y | < 1 }\int_{|x| < 1}&|G_{\varepsilon,D}(a; x,y)|^p \lesssim 1, \label{Ex.7g} \\
\int_{|y | < 1 }\int_{|x| < 1}&|\nabla_x G_{\varepsilon, D}(a; x,y)|^q+|\nabla_y G_{\varepsilon, D}(a; x,y)|^q \lesssim 1 ,\label{Ex.7}\\
\int_{ |y | < 1}\int_{|x|> 4}&|\nabla\nabla G_{\varepsilon,D}(a; x,y)|^2\lesssim 1. \label{Ex.7e}
\end{align}
For a given $L^2$-vector field $g$ with support in $\{|x| > 2 \}\cap D$, the solution\footnote{Also in this case, we consider u to be the weak solution in the sense of Definition \ref{R.6}, this time with $R=+\infty$.} of 
\begin{align*}
\begin{cases}
 -\nabla\cdot a\nabla u + \varepsilon\mathcal{L}_n u=\nabla\cdot g \mbox{\hspace{0.1cm} \ \  in $D$}\\
 \ \ u= 0 ,\ \ \ \ \mbox{\hspace{3cm} in $\partial D$}
\end{cases}
\end{align*}
satisfies by (\ref{KC}) the energy estimate $\lambda \int |\nabla u|^2 \leq \int |g|^2$. In addition, since by H\"older's inequality
\begin{align*}
\int_{|y|< 1} |u|^2 \lesssim \bigl(\int_{|y|< 1} |u|^{\frac{2d}{d-2}} \bigr)^{\frac{d-2}{d}}
\end{align*}
and the scale-invariant Sobolev Inequality 
\begin{align*}
\bigl(\int |u|^{\frac{2d}{d-2}} \bigr)^{\frac{d-2}{d}}\lesssim \int |\nabla u|^2,
\end{align*}
we have that
\begin{align*}
\int_{|y|<1} |u|^2 \lesssim \int |\nabla u|^2,
\end{align*}
and thus infer
\begin{align*}
\int_{|y|< 1} |u(y)|^2 \lesssim \int |g(y)|^2.
\end{align*}
Thus, the previous estimate and the energy estimate respectively yield, thanks to the representation formula \eqref{representation.f}, that
\begin{align} 
\int_{|y|< 1}\big|\int_{|x|> 2}\nabla G_{\varepsilon,D}(a;x,y) g(x)\big|^2\lesssim\int |g|^2, \label{Ex.6}\\
\int_{|y|< 1}\big|\int_{|x|> 2}\nabla\nabla G_{\varepsilon,D}(a;x,y) g(x)\big|^2\lesssim\int |g|^2 \label{Ex.6a}.
\end{align}
We now apply Lemma \ref{BGO3} to the families $\{ G_{\varepsilon, D}(a; \cdot, y)\}_{|y|< 1}$ and $\{\nabla_yG_{\varepsilon, D}(a; \cdot , y) \}_{|y|< 1}$, with functionals given by $\int \cdot g$ and measure $\mu( dy)= dy_{| \{ |y|< 1 \}}$.
We specify that we are allowed to use Lemma \ref{BGO3} on both families since, by the previous step, they respectively solve the problems (\ref{PDa}) and (\ref{PDay}) and thus are a-harmonic in $\{|x|> 2\}\cap D$ for ($\mu$-)almost every $y$ such that $|y|<1$.
Therefore, from (\ref{Ex.6}) and (\ref{Ex.6a}) we get
\begin{align}\label{O.6}
\int_{|y|< 1}\int_{|x|> 8} |\nabla G_{\varepsilon,D}(a; x,y)|^2 + |\nabla\nabla G_{\varepsilon, D}(a;x,y)|^2 \lesssim 1,
\end{align}
which implies the bounds (\ref{Ex.7h}) and (\ref{Ex.7e}).\\
We now turn to inequality (\ref{Ex.7f}): By the shift invariant property (\ref{Ex.29}) and the invariance under scaling of the previous argument, it follows from (\ref{O.6}) that for all $w \in \mathbb{R}^d$ and $r > 0$ it holds
\begin{align}
\int_{|y-w|<r} \int_{|x-w|>8r}|\nabla G_{\varepsilon,D}(a;x,y)|^2 \lesssim r^2,\label{Ex.28d}\\
\int_{|y-w|<r} \int_{|x-w|>8r}|\nabla\nabla G_{\varepsilon, D}(a; x,y)|^2 \lesssim 1.\label{Ex.28c}
\end{align}
We appeal to the scale-invariant Sobolev Inequality in the exterior domain\footnote{To show Sobolev's inequality in the outer domain $\{|x|> R\}$ we argue as follows: By scale invariance, we may reduce ourselves to the domain $\{|x|>1\}$;
moreover, by standard approximation, we may assume $u$ to be smooth and zero outside a ball big enough. 
We now extend u inside $\{|x|<1\}$ using the radial reflection $x \rightarrow \frac{x}{|x|^2}$, apply Sobolev's inequality on the whole space and conclude by observing that, due to our choice of extension, the Dirichlet integral in $\{|x|< 1\}$
can be controlled by the Dirichlet integral in $\{|x|>1 \}$.}  $\{|x-w| > 8r \}$ to obtain from (\ref{Ex.28c}) that
\begin{align}\label{Ex.26}
\int_{|y-w|<r} \bigl(\int_{|x-w|>8r}|\nabla_y G_{\varepsilon, D}(a; x,y)|^{\frac{2d}{d-2}}\bigr)^{\frac{d-2}{d}} \lesssim 1.
\end{align}
Thus, H\" older's inequality in the x-variable yields
\begin{align}\label{Ex.44}
\int_{|y-w|< r}&\int_{8r<|x-w|< 16r}|\nabla_y G_{\varepsilon,D}(a; x,y)|^2 \notag\\
&\lesssim  r^2 \int_{|y-w|<r} \bigl(\int_{|x-w|> 8r}|\nabla_y G_{\varepsilon, D}(a; x,y)|^{\frac{2d}{d-2}}\bigr)^{\frac{d-2}{d}}\stackrel{(\ref{Ex.26})}{\lesssim} r^2.
\end{align}
We now cover the ball $\{ |y|<1 \}$ with the union of smaller balls of radius $0< r < 1$, each of them centred in $n \sim r^{-d}$ points $\{w_i\}_{i=1}^n$ of the lattice $\frac{r}{\sqrt{d}} \mathbb{Z}^d$. Then, estimates (\ref{Ex.44}) and (\ref{Ex.28d}) yield
\begin{align}\label{Ex.44b}
\int_{|y|< 1}&\int_{9r <|x-y|< 15r}|\nabla G_{\varepsilon,D}(a; x,y)|^2 +|\nabla_y G_{\varepsilon, D}(a; x,y)|^2\notag\\
\lesssim &\sum_{i=1}^n \int_{|y-w_i|<r} \int_{9r < |x-y| < 15r}|\nabla G_{\varepsilon,D}(a; x,y)|^2 +|\nabla_y G_{\varepsilon, D}(a; x,y)|^2\notag\\
\lesssim &\sum_{i=1}^n \int_{|y-w_i|<r} \int_{8r < |x-w_i| < 16r}\hspace{-0.2cm}|\nabla G_{\varepsilon,D}(a; x,y)|^2 +|\nabla_y G_{\varepsilon, D}(a; x,y)|^2 \stackrel{(\ref{Ex.44})-(\ref{Ex.28d})}{\lesssim} r^{2-d}.
\end{align}
It follows from this that for any $\alpha > d-2$ and $0< r < 1$
\begin{align}\label{Ex.44c}
\int_{|y|< 1}\int_{9r <|x-y|< 15r}&|x-y|^{\alpha}\bigl( |\nabla G_{\varepsilon,D}(a; x,y)|^2 +|\nabla_y G_{\varepsilon, D}(a; x,y)|^2\bigr)\notag\\
\lesssim r^{\alpha}\int_{|y|< 1}&\int_{9r <|x-y|< 15r}|\nabla G_{\varepsilon,D}(a; x,y)|^2 +|\nabla_y G_{\varepsilon, D}(a; x,y)|^2 \lesssim r^{\alpha-(d-2)}.
\end{align}
Since $\alpha -(d-2)>0$, summing over dyadic annuli in the x-variable we infer
\begin{align*}
\int_{|y|< 1}\int_{ |x-y|< 2}&|x-y|^{\alpha}\bigl( |\nabla G_{\varepsilon,D}(a; x,y)|^2 +|\nabla_y G_{\varepsilon, D}(a; x,y)|^2\bigr)\lesssim 1,
\end{align*}
and thus (\ref{Ex.7f}).\\
We now claim that from (\ref{Ex.7f}) we obtain (\ref{Ex.7}): If we smuggle in (\ref{Ex.7}) the weight $|x-y|^{\frac \alpha 2 q}$ and apply H\"older's inequality first in x and then in y, we get 
\begin{align*}
\int_{|x| < 1}&\int_{|y | < 1 }|\nabla_x G_{\varepsilon, D}(a; x,y)|^q + |\nabla_y G_{\varepsilon, D}(a; x,y)|^q\\
&\lesssim \biggl(\int_{|x| < 1}\int_{|y | < 1 }|x-y|^{\alpha}\bigl(|\nabla_x G_{\varepsilon, D}(a; x,y)|^2 + |\nabla_y G_{\varepsilon, D}(a; x,y)|^2\bigr)\biggr)^{\frac{q}{2}}\\
& \hspace{0.8cm}\times\biggl(\int_{|x| < 1}\int_{|y | < 1 }|x-y|^{-\frac{q}{2-q}\alpha}\biggr)^{\frac{2-q}{2}},
\end{align*}
and thus (\ref{Ex.7}), as our assumption $1 \leq q < \frac{d}{d-1}$ ensures that there exists an $\alpha > d-2$ such that
\begin{align*}
\int_{|x| < 1}\int_{|y | < 1 }|x-y|^{-\frac{q}{2-q}\alpha} < +\infty.
\end{align*}
It only remains to establish  (\ref{Ex.7g}): We first observe that if we prove an analogy of (\ref{Ex.44}) for $G_{\varepsilon, D}$, namely that for $r > 0$
\begin{align}\label{Ex.45}
\int_{|y|< r}\int_{8r< |x| < 16r}|G_{\varepsilon,D}(a; x,y)|^2 \lesssim r^{4},
\end{align}
then by a scaling and covering argument similar to the one in (\ref{Ex.44b}) and (\ref{Ex.44c}) for $\nabla G_{\varepsilon,D}$ and $\nabla_y G_{\varepsilon,D}$, we infer
\begin{align*}
\int_{|y|< 1}&\int_{6r <|x-y|< 9r}|G_{\varepsilon,D}(a; x,y)|^2 \lesssim r^{4-d},
\end{align*}
and thus, for any $\alpha > d-4$,
\begin{align*}
\int_{|y|< 1}\int_{ |x-y|< 2}|x-y|^{\alpha} |G_{\varepsilon,D}(a; x,y)|^2 \lesssim 1.
\end{align*}
From the previous inequality, we argue as for (\ref{Ex.7}) and obtain that for $1 \leq p < \frac{d}{d-1}$
\begin{align}\label{Ex.45b}
\int_{|x| < 1}&\int_{|y | < 1 }|G_{\varepsilon, D}(a; x,y)|^p \lesssim 1.
\end{align}
We thus established inequality (\ref{Ex.7g}) with $1 \leq q < \frac{d}{d-1}$. To extend the range of the admissible exponents, we apply Poincar\'e - Sobolev's inequality
\begin{align*}
\biggl(\int_{\{|x|<1 \}\cup \{|y| <1 \}} |u|^{p^*} \biggr)^{\frac{1}{p^*}} \lesssim \biggl(\int_{\{|x|<1 \}\cup \{|y| <1 \}} |\nabla u|^{p} \biggr)^{\frac{1}{p}} + \biggl(\int_{\{|x|<1 \}\cup \{|y| <1 \}} |u|^{p} \biggr)^{\frac{1}{p}} 
\end{align*}
and thus estimate
\begin{align*}
\bigl(\int_{|x|<1}\int_{|y|< 1} |G_{\varepsilon, D}(a; x, y)&|^{p^*}\bigr)^{\frac{p}{p^*}}\\
&\lesssim \int_{|x|<1}\int_{|y|<1}|\nabla G_{\varepsilon, D} (a; x, y)|^p + |\nabla_y G_{\varepsilon, D} (a; x, y)|^p\\
&\hspace{1cm} + \int_{|x|<1}\int_{|y|<1}|G_{\varepsilon, D} (a; x, y)|^p.
\end{align*}
Appealing to inequalities (\ref{Ex.45b}) and (\ref{Ex.7}) we get for every $1 \leq p < \frac{d}{d-1}$
\begin{align*}
\int_{|x|<1}\int_{|y|< 1} |G_{\varepsilon, D}(a; x, y)&|^{p^*} \lesssim 1,
\end{align*}
 and therefore the bound (\ref{Ex.7g}) also for $\frac{d}{d-1} \leq q < \frac{d}{d-2}$.

\medskip

We thus only need to prove (\ref{Ex.45}): By the scaling property (\ref{Ex.29}) we may fix $r=1$. We apply Sobolev's inequality in the outer domain $\{ |x| > 8\}$ to get from (\ref{O.6})
\begin{align*}
\int_{|y|< 1} \bigl(\int_{|x|>8}|G_{\varepsilon, D}(a; x,y)|^{\frac{2d}{d-2}}\bigr)^{\frac{d-2}{d}} \lesssim 1,
\end{align*}
and then combine this with H\"older's inequality in the x-variable to conclude
\begin{align*}
\int_{|y|< 1} \int_{8< |x| < 16}|G_{\varepsilon, D}(a; x,y)|^2 \lesssim 1.
\end{align*}

\bigskip

{{\sc Step 3}: Uniform bounds for $\{G_{\varepsilon,D}\}_{\varepsilon>0}$ if $d=2$.}
As in the case $ d > 2$, we prove for the approximate family $\{G_{\varepsilon,D}\}_{\varepsilon>0}$ the bounds  (\ref{Teo1B})-(\ref{Teo1A2}) and  (\ref{Teo1A})-(\ref{Teo1C})-(\ref{Teo1D}). It suffices,  by property (\ref{Ex.29}), to fix $z=0$.
To show (\ref{Teo1A}), we may use the same argument of Step 2: For a given $L^2$-vector field $g$ with support in $\{|x|<2\}\cap D$, the solution\footnote{Also in this case, we consider u to be the weak solution in the sense of Definition \ref{R.6}, this time with $R=+\infty$.} of 
\begin{align*}
\begin{cases}
 -\nabla\cdot a\nabla u + \varepsilon\mathcal{L}_n u=\nabla\cdot g \mbox{\hspace{1.15cm} \ \  in D}\\
 \ \ \partial^{\alpha}u= 0 ,\ \ \ \ \mbox{\small{$0\leq |\alpha| \leq n-1$} \hspace{0.9cm} in $\partial D$}
\end{cases}
\end{align*}
satisfies by (\ref{KC}) the energy estimate $\lambda \int_D |\nabla u|^2 \leq \int |g|^2$ and yields, thanks to the representation formula \eqref{representation.f},
\begin{equation}\label{Ex.6} 
\int_{|x|> 4}\big|\int_{|y|< 2}\nabla\nabla G_{\varepsilon,D}(a;x,y) g(y)\big|^2\lesssim \int |g|^2.
\end{equation}
We now apply Lemma \ref{BGO2} to the family $\{\nabla G_{\varepsilon, D}(a; x, \cdot) \}_{\{|x|> 4\}\cap D}$, with functionals given by $\int \cdot g$ and measure $\mu( dx)= dx_{\{|x|> 4\}\cap D}$. We observe that we are allowed to use Lemma \ref{BGO2} on this family since, by (\ref{Ex.28}), we can identify
\begin{align*}
\nabla G_{\varepsilon, D}(a; x, \cdot)= \nabla_xG_{\varepsilon, D}(a; \cdot, x),
\end{align*}
with $\nabla_xG_{\varepsilon, D}(a; \cdot, x)$ constructed in Step 1 (with exchanged roles of the x and y variable). It follows from (\ref{PDay}) that for $x\in \mathbb{R}^d$ with $|x|> 4$,
\ $\nabla_xG_{\varepsilon, D}(a; \cdot, x)$ is solution of (\ref{Ex.39}) in the domain $\{|y|< 2\} \cap D$. Therefore, from (\ref{Ex.6}) we get by Lemma \ref{BGO2} the desired bound (\ref{Teo1A}).
We remark that since the scale invariant Sobolev's inequality is no more available for $d=2$ we cannot infer also (\ref{Ex.26}).
Appealing to our assumption on $D$ to have at least one bounded direction, we may use as a replacement for Sobolev's inequality the following version of Poincar\`e-Sobolev's estimate\footnote{We postpone the proof to the Appendix.}: Let $D\subset \mathbb{R}^2$ be open and having
at least one bounded direction. Then, for every $2 \leq p < +\infty$, $z\in \mathbb{R}^2$ and $R>0$, it holds
\begin{align}\label{PSI}
\bigl(\int_{|x-z| > R}|u|^p \bigr)^{\frac 1 p} \lesssim_{D,p} \bigl(\int_{|x-z|> R}|\nabla u|^2\bigr)^{\frac 1 2}, 
\end{align}
for every $u \in W^{1,1}_{loc}(\mathbb{R}^2)$ and such that $u =0$ almost everywhere outside $D$. Here, the constant depends on the size of the smallest bounded component of $D$.\\
With the same reasoning used in Step 2, once that we show that for every $\delta > 0$, $z\in \mathbb{R}^2$ and $r>0$ we have
\begin{align}
\int_{|y-z|< r}\int_{4r < |x-z| < 8r}&  |\nabla G_{\varepsilon,D}(a; x,y)|^2 + |\nabla_y G_{\varepsilon,D}(a; x,y)|^2 \lesssim_{D} r^{2-\delta},\label{LE1}\\
\int_{|y-z|< r}\int_{4r < |x-z| < 8r}&  |G_{\varepsilon,D}(a; x,y)|^2 \lesssim_{D} r^{4-\delta }, \label{LE2}
\end{align}
it follows by a covering argument, that
\begin{align*}
\int_{|y-z|< 1}\int_{5r < |x-y| < 7r}&  |\nabla G_{\varepsilon,D}(a; x,y)|^2 + |\nabla_y G_{\varepsilon,D}(a; x,y)|^2 \lesssim_{D} r^{-\delta},\\
\int_{|y-z|< 1}\int_{5r < |x-y| < 7r}&  |G_{\varepsilon,D}(a; x,y)|^2 \lesssim_{D} r^{2-\delta},
\end{align*}
and thus that for every $\alpha> d-2 = 0$
\begin{align}\label{LE4}
\int_{|y-z|< 1}\int_{|x-y|< 2}& |x-y|^\alpha\bigl( |\nabla G_{\varepsilon,D}(a; x,y)|^2 + |\nabla_y G_{\varepsilon,D}(a; x,y)|^2 \bigr)\lesssim_{D} 1,
\end{align}
and for every $\alpha > d-4 = -2$
\begin{align}\label{LE5}
\int_{|y-z|< 1}\int_{|x-y|< 2}& |x-y|^\alpha|G_{\varepsilon,D}(a; x,y)|^2 \lesssim_{D} 1.
\end{align}
We may analogously argue for a general radius $R$ and establish (\ref{LE4})-(\ref{LE5}), and thus bound (\ref{Teo1B}) for any $R> 0$. As shown in Step 2, these estimates also yield (\ref{Teo1C})-(\ref{Teo1D}) by the standard Poincar\'e-Sobolev Inequality.

\medskip

We now give the argument for (\ref{LE1}) and (\ref{LE2}): Inequality (\ref{PSI}) on $\nabla_yG_{\varepsilon, D}(a, \cdot, y)$  yields for every $2 \leq p < +\infty$
\begin{align*}
\int_{|y-z| <R} \bigl(\int_{|x-z|> 4R}& |\nabla_yG_{\varepsilon, D}(a; x, y)|^p\bigr)^{\frac 2 p}\\
&\lesssim_{D} \int_{|y-z| <R}\int_{|x-z|> 4R} |\nabla\nabla G_{\varepsilon, D}(a; x, y)|^2 \stackrel{(\ref{Teo1A})}{\lesssim_{D}} 1,
\end{align*}
and thus by H\"older's inequality in $\{4R< |x-z|< 8R \}$ also 
\begin{align}\label{Ex.44d}
\int_{|y-z| <R}\int_{4R < |x-z| < 8R} |\nabla_yG_{\varepsilon, D}(a; x, y)|^2 \lesssim_{D} R^{2\frac{p-2}{p}}.
\end{align}
Since the exponent $p$ can be chosen arbitrarily large, we obtain (\ref{LE1}) for $\nabla_yG_{\varepsilon, D}$.\\
We now observe that 
\begin{align*}
\int_{|y-z|< R}\int_{|x-w|< R}|&\nabla_y G_{\varepsilon,D}(a; x,y)|^2 \lesssim_{D}R^{2\frac{p-2}{p}},
\end{align*}
for every $w$ and $z$ such that $\{5R < |z-w| < 7R\}$. Indeed, this is implied by (\ref{Ex.44d}) and the inclusion
\begin{align*}
 \{5R < |w-z| < 7R\} \cap &\{| x-w|<R\}\\
 &\subset  \{5R < |w-z| < 7R\} \cap \{4R<| x-z|< 8R\}.
\end{align*}
For a fixed $w \in \mathbb{R}^d$, we choose $n \lesssim 1$ balls of radius R which cover the annulus $\{ 5R < |y-w| < 7R\}$ and whose centres $\{ z_i\}_{i=1}^n$ are contained in  $\{5R \leq |z-w| \leq 7R \}$. Thus, from the previous inequality we infer
\begin{align}\label{Ex.27}
\int_{5R< |y-w| < 7R}&\int_{|x-w|< R}|\nabla_y G_{\varepsilon,D}(a; x,y)|^2\notag\\
& \leq \sum_{i=1}^n\int_{|y-z_i| < R}\int_{|x-w|< R}|\nabla_y G_{\varepsilon,D}(a; x,y)|^2 \lesssim_{D} R^{2\frac{p-2}{p}}.
\end{align}
By switching the labels x and y and using the symmetry property (\ref{Ex.28}), this may be rewritten as
\begin{align*}
\int_{5R< |x-w| < 7R}&\int_{|y-w|< R}|\nabla G_{\varepsilon,D}(a; x,y)|^2 \lesssim_{D} R^{2\frac{p-2}{p}},
\end{align*}
i.e. inequality (\ref{LE1}) thanks to the arbitrariness of $ 2\leq p < +\infty$.\\
It thus remains to prove (\ref{LE2}): By Poincar\'e's inequality in the x-variable we have
\begin{align}\label{Ex.21}
 \int_{|y|< R}&\int_{4R < |x| < 8R}| G_{\varepsilon,D}(a; x,y)|^2 \notag\\
 &\hspace{1cm}\lesssim R^{2}\int_{|y|< R }\int_{4R< |x|< 8R }\hspace{-0.2cm}|\nabla G_{\varepsilon,D}(a; x,y)|^2\notag \\
 & \hspace{3cm}+ \fint_{|y|< R }\bigl| \int_{ 4R< |x|< 8R }\hspace{-0.2cm} G_{\varepsilon,D}(a; x,y) \bigr|^2.
\end{align}
Therefore, thanks to (\ref{LE1}), we conclude (\ref{LE2}) once that we show that the second term on the r.h.s. of (\ref{Ex.21}) satisfies for $\delta >0$
\begin{align}\label{LE3}
 \fint_{|y|< R }\bigl| \int_{ 4R< |x|< 8R }\hspace{-0.2cm} G_{\varepsilon,D}(a; x,y) \bigr|^2\lesssim_{D} R^{4- \delta}.
\end{align}
To do so, let us fix p and consider any function $ g \in L^{2}( \mathbb{R}^d)$ with supp$(g) \subset \{ |x| < R\}\cap D$. 
Let u be the solution of
\begin{align*}
\begin{cases}
 -\nabla\cdot a\nabla u + \varepsilon\mathcal{L}_n u= g \mbox{\hspace{0.5cm} in $D$}\\
 \ \ u= 0 , \ \ \ \mbox{ \hspace{2.45cm} in $\partial D$}.
\end{cases}
\end{align*}
The energy estimate
\begin{align*}
\int |\nabla u|^2 &\lesssim \biggl(\int_{|x|< R} |g|^{2}\biggr)^{\oh}\biggl(\int_{|x|< R} |u|^{2}\biggr)^{\oh},
\end{align*}
together with H\"older's inequality for $ 2 \leq p < +\infty$
\begin{align*}
\biggl( \int_{|x|< R} |u|^2 \biggr)^\oh \leq R^{1 -\frac 1 p}\biggl( \int |u|^p \biggr)^{\frac 1 p}
\end{align*}
and the standard Poincar\'e-Sobolev's inequality
\begin{align*}
\biggl( \int |u|^p \biggr)^{\frac 1 p} & {\lesssim_{D}} \biggl( \int |\nabla u|^{2} \biggr)^{\frac 1 2},
\end{align*}
yields
\begin{align*}
\biggl(\int |\nabla u|^2 \biggr)^\oh \lesssim_{D} R^{1 -\frac 1 p} \biggl( \int_{|x|< R} |g|^{2} \biggr)^{\oh}.
\end{align*}
Applying again H\"older's inequality in $\{|x| < R \}$ and the Poincar\'e-Sobolev's inequality stated above, the previous estimate also implies that
\begin{align}\label{Ex.20b}
\bigl( \int_{4R < |x| < 8R} |u|^2 \bigr)^{\frac 1 2} &\lesssim_{D} R^{2-\frac{4}{p}} \bigl( \int |g|^{2} \bigr)^{\frac 1 2}.
\end{align}
By the representation formula \eqref{representation.f}, estimate (\ref{Ex.20b}) can be rewritten as
\begin{align*}
\bigl( \int_{4R < |x| < 8R } |\int_{|y| < R} G_{\varepsilon, D}(a; x, y) g(y) |^2 \bigr)^{\frac 1 2}  \lesssim_{D} R^{2-\frac{4}{p}} \bigl(\int |g|^2 dy \bigr)^{\frac{1}{2}},
\end{align*}
so that Jensen's inequality implies
\begin{align*}
| \int_{|y| < R}\biggl( \int_{4R < |x| < 8R}  G_{\varepsilon, D}(a; x, y) \biggr) g(y) | \lesssim_{D} R^{3-\frac{4}{p}} \bigl(\int |g|^2 \bigr)^{\frac{1}{2}}.
\end{align*}
The arbitrariness of $g$ allows to argue by duality that
\begin{align}\label{Ex.20c}
 \fint_{|y| < R} |\int_{4R < |x| < 8R}  G_{\varepsilon, D}(a; x,y) |^2  \lesssim_{D} R^{4- \frac 8 p},
\end{align}
i.e. the desired bound  (\ref{LE3}) thanks to the arbitrariness of $2 \leq p < +\infty$.

\medskip

At last, we prove that the bound (\ref{Teo1A}) implies (\ref{Teo1A2}): Modulo a change of coordinates, we may assume $D \subset I \times \mathbb{R}$, with $I$ a bounded interval. Moreover, since by construction for almost every $y \notin D$, $G_{\varepsilon, D}(a; \cdot, y)=0$ almost surely in $\mathbb{R}^2$, we reduce ourselves to those
$z\in \mathbb{R}^2$ and $R>0$ such that $\{|y-z| < R\}\cap D \neq \emptyset$ and, without loss of generality we fix $z=0$. Therefore, for every $R \gtrsim_D 1$ the rectangle $I \times (-2R, 2R)$ is such that
$$
\{|y| < R\}\cap D \subset I \times (-2R, 2R) \subset \{|y| < 4R\}\cap D.
$$
and thus
\begin{align*}
\int_{|y| < R}\int_{|x|> 8R} |\nabla G_{\varepsilon, D}(a; x, y)|^2 \leq \int_{I \times (-2R, 2R)}\int_{|x|> 16R} |\nabla G_{\varepsilon, D}(a; x, y)|^2.
\end{align*}
Since by (\ref{Ex.28}) and (\ref{PDay}) the application $\nabla G_{\varepsilon, D}(a ; x, y)$ vanishes outside D we may apply Poincar\'e's inequality in $I \times (-2R, 2R)$ and get from the previous inequality
\begin{align*}
\int_{|y| < R}\int_{|x|> 8R} |\nabla G_{\varepsilon, D}(a; x, y)|^2 \lesssim_D R^2 \int_{I \times (-2R, 2R)}\int_{|x|> 8R} |\nabla \nabla G_{\varepsilon, D}(a; x, y)|^2
\end{align*}
and thus also
\begin{align*}
\int_{|y| <R}\int_{|x|> 8R} |\nabla &G_{\varepsilon, D}(a; x, y)|^2 \\
&\lesssim_D R^2\int_{|y|  < 4R}\int_{|x|> 8R} |\nabla \nabla G_{\varepsilon, D}(a; x, y)|^2 \stackrel{(\ref{Teo1A})}{\lesssim_D} R^2.
\end{align*}

\bigskip

{{\sc Step} 4: Existence of $G_D(a, \cdot ,\cdot)$.} In this final step we do not distinguish between the cases $d>2$ and $d=2$.
The uniform bounds (in $\varepsilon$)  (\ref{Teo1A})-(\ref{Teo1B}) and (\ref{Teo1C})-(\ref{Teo1D}) for the family $\{G_{\varepsilon, D}(a ; \cdot, \cdot) \}_{\varepsilon \downarrow 0}$ allow us to argue by weak-compactness that, modulo a subsequence, for $1 \leq q < \frac{d}{d-1}$
\begin{align}
G_{\varepsilon,D}( a; \cdot, \cdot) &\rightharpoonup G_{D}(a; \cdot, \cdot)  \mbox{\hspace{1.4cm} in $ \ W^{1,q}_{\mbox{\tiny{loc}}}(\mathbb{R}^{d} \times \mathbb{R}^{d})$,  \,}\label{Ex.36}\\
\nabla_{x,y}G_{\varepsilon,D}(a; \cdot, \cdot)& \rightharpoonup \nabla_{x,y}G_{D}(a; \cdot, \cdot)\mbox{\ \ \ \ in $ \ L^{2}_{\mbox{\tiny{loc}}}(\mathbb{R}^{d}\times \mathbb{R}^{d} \setminus \{ {x}=y \})$,}\label{Ex.36a}\\
\nabla\nabla G_{\varepsilon, D}(a; \cdot, \cdot)& \rightharpoonup \nabla\nabla G_{D}(a; \cdot, \cdot)  \mbox{\hspace{0.7cm} in $ \ L^{2}_{\mbox{\tiny{loc}}}(\mathbb{R}^{d}\times \mathbb{R}^{d} \setminus \{{x}=y\})$.}\label{Ex.36b}
\end{align}
Since $G_{D}(a; \cdot, \cdot) \in \ W^{1,q}_{\mbox{\tiny{loc}}}(\mathbb{R}^{d} \times \mathbb{R}^{d})$ with $G_{D}(a; \cdot, \cdot) \equiv 0$ outside $D \times D$,
it follows respectively that for almost every $y\in \mathbb{R}^{d}$
\begin{align*}
G_{D}(a, \cdot,y)&\in W^{1,q}_{loc}(\mathbb{R}^{d}),\\
G_{D}(a, \cdot,y)&= 0 \text{\ \ almost everywhere outside $D$}.
\end{align*}
We now show that for almost every $y\in \mathbb{R}^{d}$, the application $G_{D}(a, \cdot, y)$ solves (\ref{TG}): By construction of $G_{\varepsilon, D}$, it holds indeed that for almost every $y\in \mathbb{R}^{d}$ 
and every $\zeta \in C^{\infty}_0(D)$
\begin{align*}
\int \nabla\zeta(x)\cdot a(x)& \nabla G_{\varepsilon, D}(a; x, y) \\
 -&\varepsilon \int \mathcal{L}_n\zeta(x) G_{\varepsilon, D}(a; x, y)= \zeta(y).
\end{align*}
For every $\rho \in C^\infty_0(\mathbb{R}^{d})$, the previous identity yields 
\begin{align*}
\int \rho(y)\int \nabla\zeta(x)\cdot a(x)& \nabla G_{\varepsilon, D}(a; x, y) \\
 -&\varepsilon\int \rho(y) \int \mathcal{L}_n\zeta(x) G_{\varepsilon, D}(a; x, y)=\int \rho(y)\zeta(y),
\end{align*}
so that for $\varepsilon \rightarrow 0$, by weak convergence (\ref{Ex.36}), we get 
\begin{align}\label{Ex.42}
\int \rho(y)\int \nabla\zeta(x)\cdot a(x)& \nabla G_{D}(a; x, y) = \int \rho(y)\zeta(y).
\end{align}
The arbitrariness of the test function $\rho \in C^{\infty}_0(\mathbb{R}^{d})$ implies that for almost every $y\in \mathbb{R}^{d}$
\begin{align}\label{Ex.41}
\int \nabla\zeta(x)\cdot a(x)& \nabla G_{D}(a; x, y)= \zeta(y).
\end{align}
We now appeal to the separability of $C^\infty_0(D)$ with respect to the $C^1$ topology to conclude that for almost every $y\in \mathbb{R}^{d}$ and for every $\zeta\in C^\infty_0(D)$
\begin{align}\label{Ex.25}
\int \nabla\zeta(x)\cdot a(x) \nabla G_{D}(a; x, y) = \zeta(y),
\end{align}
i.e. for almost every $y\in \mathbb{R}^{d}$ a solution $G_{D}(a; \cdot, y)$ of (\ref{TG}) exists.
Reasoning in the same way, from (\ref{PDay}) and weak convergence (\ref{Ex.36b}) we also obtain that for every $R >0$, $ z\in \mathbb{R}^{d}$ and almost every $y \in \{|y- z|> 2R\}$,
the function $\nabla_yG_{D}(a; \cdot, y)$ solves
\begin{align}\label{PDy}
\begin{cases}
 -\nabla\cdot a\nabla \nabla_{y} G_{D}(a; \cdot , y) = 0 \mbox{\ \ \hspace{0.8cm}  in $\{|x- z| < R \} \cap D$}\\
 \ \ \nabla_{y} G_{D}(a; \cdot , y) = 0 \mbox{ \hspace{2.2cm} in $\{|x- z|< R \}\cap \partial D$.}
\end{cases}
\end{align}
Furthermore, appealing to (\ref{Ex.36}), (\ref{Ex.36a}) and (\ref{Ex.36b}) and the lower semiconinuity of the bounds (\ref{Teo1B})-(\ref{Teo1A2}),  (\ref{Teo1A}) and (\ref{Teo1C})-(\ref{Teo1D}), 
we get that they hold also for $G_{D}(a; \cdot, \cdot)$; in particular, inequalities (\ref{Teo1B})-(\ref{Teo1A2}) imply that $G_{D}(a; \cdot, \cdot)$ satisfies bound (\ref{WB}) for any $\alpha \in (d-2, d)$ as well as bound 
(\ref{WB2}) for any $R>0$ if $d> 2$ and for any $R \gtrsim_D 1$ if $d=2$. Thus, $G_D(a; \cdot ,\cdot)$ is the Green function for the domain $D$. By uniqueness (cf. Lemma \ref{U}) and symmetry of the operator $-\nabla \cdot a\nabla$, cf. (\ref{Sym}), we also have that
for all $a\in \Omega$, $z\in \mathbb{R}^d$, $R> 0$ and almost every $x,y \in \mathbb{R}^d$ it holds
\begin{align}
G_{D}(a;x,y)&=G_{D}(a;y,x),\label{Ex.28a}\\
G_{D}(a(\cdot+z);x,y)&=G_{D+z}(a;x+z,y+z),\label{Ex.29a}\\ 
G_{D} (a; R{x}, R y) &= R^{2-d}G_{R^{-1}D}(a(R \, \cdot );  x, y). \label{Ex.24a}
\end{align}

\bigskip 

\begin{remark}
We observe that also for $G_D$ holds a representation formula for weak solutions of 
\begin{align}\label{Cp}
\begin{cases}
-\nabla\cdot a \nabla u = f \ \ \ \ \mbox{ in $D$}\\
\ \ u= 0 \ \ \ \ \mbox{\hspace{1.1cm} in $\partial D$,}
\end{cases}
\end{align}
with $f \in L^q(D)$, $q>d$ and compactly supported. This may easily follows by uniqueness of the solution $u$  (via Riesz's representation theorem) and the fact that the function
\begin{align*}
\hat u(x) = \int G_D(a; y, x) f(y) \, dy 
\end{align*}
is well defined and such that $\nabla \hat u \in L^2( D)$, thanks to the bounds \eqref{Teo1A2}, \eqref{Teo1C} and \eqref{Teo1D}. Note that $\hat u$ weakly solves \eqref{Cp}:
This may be shown as for \eqref{representation.f} first for smooth test functions and then extended by standard approximation.\\ 
We also have that for any $f \in L^2(D)$ and $g \in [L^2( D)]^d$ with compact support, the weak solution of
\begin{align}\label{Cp2}
\begin{cases}
-\nabla\cdot a \nabla u = f + \nabla \cdot g \ \ \ \ \mbox{ in $D$}\\
\ \ u= 0 \ \ \ \ \mbox{\hspace{2.2cm} in $\partial D$,}
\end{cases}
\end{align}
admits the representation
\begin{align*}
u(x) = \int G_D(a; y, x) f(y) \, dy - \int \nabla G_D(a; y, x) \cdot g(y) \, dy,
\end{align*}
whenever $x$ is outside the support of both $g$ and $f$. We first consider the family $\{ u_\varepsilon \}_{\varepsilon > 0}$ of solutions to the approximate problems
\eqref{approx.pb} with the same r.h.s. : By standard weak-compactness arguments, (up to a subsequence)  $\{ u_\varepsilon\}_{\varepsilon > 0}$  weakly converge in $W^{1,2}_{loc}(D)$ to the solution $u$ of \eqref{Cp2}.
We thus conclude the identity above by using \eqref{representation.f}, together with \eqref{Ex.36a}-\eqref{Ex.36b}, and the uniqueness of the (weak) limit.
\end{remark}

\section{Proof of Corollary \ref{Cor1} and \ref{Cor2}}\label{Corollaries}
{\sc Proof of Corollary \ref{Cor1}.}\\
Let $D\subset \mathbb{R}^d$ with $d \geq 2$ be as in the statement of Corollary \ref{Cor1}. Modulo a change of coordinates, we can assume that there exists a bounded interval $I \subset \mathbb{R}$ such that $D \subset I \times \mathbb{R}^{d-1}$. 
In addition, without loss of generality we may suppose that $|I|=1$: It will become clear along the proof that the estimates obtained depend on the size of $I$.
For $I \times \mathbb{R}^{d-1}$ as above, we write $\bar x= (x_1, x') \in I \times \mathbb{R}^{d-1}$.\\
The main ingredient for the argument of Corollary \ref{Cor1} is the following elliptic regularity result (\cite{O}, Lemma 2.2), adapted to elliptic systems with Dirichlet boundary conditions. We postpone its proof to the Appendix.
\begin{lemma}\label{exp}
Let $ D$ be as introduced above, and $a\in \Omega$ such that it satisfies (\ref{StE}). For $g\in L^2(D)^{d}$, let u solve (in the sense of Definition 1 with $\varepsilon =0$ and $R= +\infty$)
\begin{align*}
\begin{cases}
-\nabla\cdot a \nabla u = \nabla\cdot g \ \ \ \ \mbox{ in $D$}\\
\ \ u= 0 \ \ \ \ \mbox{\hspace{1.9cm} in $\partial D$.}
\end{cases}
\end{align*}
Then, there exists a constant $C_0$ depending on $d$, $\lambda$ (and the size of $I$) such that for any $x'_0 \in \mathbb{R}^{d-1}$, it holds
\begin{align}\label{Ex.32}
\int \exp{(\frac {|x'-x_{0}'|} { C_0 } )} |\nabla u|^2 \lesssim \int \exp{(\frac{ |x'-x_{0}'| }{C_0})} |g|^2.
\end{align}
\end{lemma}

\bigskip

We start by claiming that the previous lemma, together with an application of Lemma \ref{BGO2}, yields (\ref{Teo1F}). More precisely, we have for every $z\in \mathbb{R}^d$ and $R > 0$ that
\begin{align}
\int_{| y- z|> 8R}&\int_{\{|x-z |< R\}}|\nabla\nabla G_{D}( a; x, y)|^2 \lesssim \exp{(-\frac{2R}{C_0})},\label{Ex.31b}
\end{align}
with $C_0$ as in Lemma \ref{exp}.
Indeed, for a vector field $g\in L^2(D)^{d}$ with supp$(g) \subset \{|{x}- z| < 2R \}\cap D$, we apply Lemma \ref{exp} with $x_0' =z' $ to the solution of
\begin{align*}
\begin{cases}
 -\nabla\cdot a\nabla u =\nabla\cdot g \mbox{\ \ \  in \  $\ D$}\\
 \ \ u= 0  \hspace{2.35cm} \mbox{in \  $\partial D$}
\end{cases}
\end{align*}
and obtain, by using the formulas of Remark 2 in (\ref{Ex.32}), that
\begin{align}\label{Ex.33a}
\int_{|y- z|> 4R} &\exp{(\frac {|y'-z'|} {C_0})}\notag\\
\times | &\int_{|x- z|< 2R} \nabla\nabla G_{D}( a; x,y)\cdot g(x) |^2 \lesssim \exp{\bigl(\frac{2R}{C_0}\bigr)}\int |g|^2.
\end{align}
We now apply to (\ref{Ex.33a}) Lemma \ref{BGO2}, this time in the case $\varepsilon = 0$, with functionals given by $\int  g$, measure $\mu(dy)= \nolinebreak\exp{(\frac {|y'-z'|} {C_0})} dy|_{\{| y- z|> 4R\}}$ and to the family
of functions $\{\nabla G_{D}( a; y, \cdot)\}_{\{| y- z|> 4R\}}$, a-harmonic in $\{ | x- z| < 2R \} \cap D$ by (\ref{Ex.28a}) and (\ref{PDy}). 
From (\ref{Ex.33a}) we thus infer
\begin{align}\label{Ex.33b}
\int_{|y- z|> 4R }\int_{|x- z|< R }\hspace{-0.3cm}\exp{(\frac {|y'-z'|} {C_0} )}| \nabla\nabla G_{D}(a; x,y)|^2 \lesssim \exp{\bigl(\frac{2R}{C_0}\bigr)},
\end{align}
which implies inequality (\ref{Ex.31b}) since for $R \gtrsim_D 1$ it holds the inclusion  
$$
\{| y- z|> 8R\}\cap D \subset \mathbb{R} \times \{ |y' - z'| > 4R \} \cap D.
$$

\medskip

To obtain also (\ref{Teo1E}) we argue similarly to Step 3. of Theorem \ref{Teo1}; we first tackle the bound for the gradient of $G_D$: Without loss of generality we may reduce ourselves to consider the case  $\{|x-z| <R\}\cap D \neq \emptyset$ and fix $z=0$.
For every $R \gtrsim_D 1$ the rectangle $I \times (-2R, 2R)^{d-1}$ is such that
$$
\{|y| < R\}\cap D \subset I \times (-2R, 2R)^{d-1} \subset \{|y| < 4R\}\cap D.
$$
and thus
\begin{align*}
\int_{|y| < R}\int_{|x|> 8R} |\nabla G_{D}(a; x, y)|^2 \leq \int_{I \times (-2R, 2R)^{d-1}}\int_{|x|> 8R} |\nabla G_{D}(a; x, y)|^2.
\end{align*}
Since by (\ref{Ex.28a}) and (\ref{PDa}) the application $\nabla G_{\varepsilon, D}(a; x, y)$ vanishes outside D, we may apply Poincar\'e's inequality in $I \times (-2R, 2R)^{d-1}$ and get from the previous inequality that
\begin{align*}
\int_{|y| < R}\int_{|x|> 8R} |\nabla G_{D}(a; x, y)|^2 \lesssim_D R^2 \int_{I \times (-2R, 2R)^{d-1}}\int_{|x|> 8R} |\nabla \nabla G_{D}(a; x, y)|^2
\end{align*}
and thus that
\begin{align}\label{Ex.33}
&\int_{|y| <R}\int_{|x|> 8R} |\nabla G_{\varepsilon, D}(a; x, y)|^2 \notag\\
&\hspace{0.4cm}\lesssim_D R^2\int_{|y|  < 4R}\int_{|x|> 8R} |\nabla \nabla G_{\varepsilon, D}(a; x, y)|^2 \stackrel{(\ref{Ex.31b})}{\lesssim_D} R^2\exp{\bigl(-\frac{2R}{C_0}\bigr)}.
\end{align}
This trivially yields (\ref{Teo1E}) for $\nabla G_D$.\\
The bound (\ref{Teo1E}) for $G_D$ follows from (\ref{Ex.33}) by an application of Poincar\'e's inequality, this time in the domain $\{|x-z|> R\}\cap D$.
\footnote{The argument is analogous to the one for (\ref{PSI}) of Step 3. of Theorem \ref{Teo1}.}

\bigskip

{\sc  Proof of Corollary \ref{Cor2}.} Throughout this proof we assume $d > 2$ and recall that, for $a \in \Omega$,we adopt the notation $G(a, \cdot ,\cdot)$ for the Green function for the whole space $\mathbb{R}^d$.\\
{{\sc Step} 1: $\langle \cdot \rangle$-almost sure solutions of (\ref{TG}) and (\ref{TGy}).} We show that with the additional structure $\langle \cdot \rangle$ on $\Omega$, it holds that
\begin{align}
\text{ $\forall$ \ a.e. \ $y\in \mathbb{R}^d$ and $\langle \cdot \rangle$- a.e. \ $a \in \Omega$, $G(a; \cdot, y)$ solves (\ref{TG}),}&\label{statCor}\\
\text{ $\forall$ \ a.e. \ $y\in \mathbb{R}^d$ and $\langle \cdot \rangle$- a.e. \ $a \in \Omega$, $\nabla_yG(a; \cdot, y)$ solves (\ref{TGy}),}&\label{statCorb}
\end{align}
and
for every $R>0$, $z\in\mathbb{R}^d$, almost every $x,y \in \mathbb{R}^d$ and $\langle \cdot \rangle$-almost every $a\in\Omega$,
\begin{align}
G(a; x, y) &= G(a; y, x), \label{Ex.28b}\\
G(a, x+z, y+z) &= G(a(\cdot + z), x, y), \label{Ex.29b}\\
G(a; R x,R  y) &=R^{2-d}G(a(R \, \cdot );  x,  y) .\label{Ex.24b}
\end{align}
In other words, we prove that the ensemble on $\Omega$, chosen to be such that the $L^1(\Omega)$ space is separable (cf. Section \ref{Not} ), allows to exchange in (\ref{TG}) and (\ref{TGy}) as well as (\ref{Ex.28a}), (\ref{Ex.29a}) and (\ref{Ex.24a}) the order of
the quantors $a$ and $x,y$. This will be useful in the next steps, when we treat $G(\cdot ; \cdot, y)$ and $\nabla_yG(\cdot ; \cdot, y)$ as almost sure solutions of respectively (\ref{TG}) and (\ref{TGy}).

\medskip

From Theorem \ref{Teo1}, we have that for every $\phi\in L^1(\Omega)$ and $\zeta,\rho \in C^\infty_0(\mathbb{R}^d)$ it holds
\begin{align*}
\langle \phi(a)\int \rho(y)\int \nabla\zeta(x)\cdot a(x)& \nabla G(a; x, y)\rangle = \langle \phi(a) \int \rho(y)\zeta(y) \rangle,
\end{align*}
or equivalently by Fubini's theorem,
\begin{align*}
\int \rho(y)\langle \phi(a)\int \nabla\zeta(x)\cdot a(x)& \nabla G(a; x, y)\rangle =  \int \rho(y)\langle \phi(a)\zeta(y) \rangle.
\end{align*}
As the test function $\rho \in C^{\infty}_0(\mathbb{R}^d)$ is arbitrary, we infer that for almost every $y\in \mathbb{R}^d$
\begin{align}\label{Cor4}
\langle \phi(a) \int \nabla\zeta(x)\cdot a(x)& \nabla G(a; x, y)\rangle =  \langle \phi(a) \zeta(y) \rangle.
\end{align}
Since the space $L^1(\Omega)$ is separable, it also follows that for almost every $y\in \mathbb{R}^d$ and $\langle \cdot \rangle$-almost every $a\in \Omega$
\begin{align*}
\int \nabla\zeta(x)\cdot a(x)& \nabla G(a; x, y)=   \zeta(y).
\end{align*}
We now appeal to the separability of $C^\infty_0(\mathbb{R}^d)$ with respect to the $C^1$ topology, to conclude that for almost every $y\in \mathbb{R}^d$, $\langle \cdot \rangle$-almost every $a\in \Omega$ and for every $\zeta\in C^\infty_0(\mathbb{R}^d)$,
\begin{align*}
\int \nabla\zeta(x)\cdot a(x) \nabla G(a; x, y) = \zeta(y),
\end{align*}
i.e. claim (\ref{statCor}). With an analogous argument, from (\ref{PDy}) we also prove (\ref{statCorb}).

\medskip

In a similar way we obtain identities (\ref{Ex.28b}), (\ref{Ex.29b}) and (\ref{Ex.24b}): We show the argument only for (\ref{Ex.24b}) since the arguments for the other two are analogous. Identity (\ref{Ex.24a}) with a fixed $R>0$ yields for any triple $\phi\in L^1(\Omega)$, $\zeta,\rho \in C^\infty_0(\mathbb{R}^d)$
\begin{align*}
\langle \phi(a) \int \int \zeta(x)\rho(y) G(a, R x, R  y) \rangle = R^{2-d}\langle \phi(a) \int \int \zeta( x)\rho( y) G(\hat a, x, y) \rangle ,
\end{align*}
with $\hat \ : \Omega \rightarrow \Omega$ such that $\hat a (\cdot):= a(R \,\cdot)$. 
By Fubini's theorem we may exchange the order of integration in the previous identity and obtain that
\begin{align*}
\int \int \zeta(x)\rho(y)\langle \phi(a)  G(a, R x, R y) \rangle = R^{2-d}\int \int \zeta( x)\rho(y) \langle \phi(a) G(\hat a,x,y) \rangle.
\end{align*}
Therefore, separability of $L^1(\Omega)$ yields that for almost every $x, y \in \mathbb{R}^d$ and $\langle \cdot \rangle$-almost every $a\in \Omega$, identity (\ref{Ex.24b}) holds.

\bigskip

{{\sc Step }2: Spacially averaged annealed bounds.} We argue that for almost every $y\in \mathbb{R}^d$ and $R>0$
\begin{align}
\langle\int_{R<|x-y|<2R}|G(a;&x,y)|^2\rangle\lesssim R^{4-d},\label{O.5}\\
\langle\int_{|x-y|>R}|\nabla G(a&; x,y)|^2 + |\nabla_y G(a;x,y)|^{2}\rangle \lesssim R^{2-d},\label{O.5bis}\\
\langle\int_{|x-y|>R}|\nabla\nabla G(&a;x,y)|^2\rangle\lesssim R^{-d}.\label{O.5tris}
\end{align}
We claim that it is sufficient to prove (\ref{O.5}),(\ref{O.5bis}) and (\ref{O.5tris}) for $R =1$: Let us assume for instance that (\ref{O.5}) holds for a $R = 1$, namely that
\begin{align}\label{O.5b}
\langle\int_{1<|x-y|<2}|G(&a;x,y)|^2\rangle\lesssim 1.
\end{align}
Since for almost every $x,y \in \mathbb{R}^d$, $\langle \cdot \rangle$-almost every $a\in \Omega$ and every countable set of radii $\mathcal{R}$ identity $(\ref{Ex.24b})$ holds, we may infer from (\ref{O.5b}) that for almost every $y\in \mathbb{R}^d$, bound (\ref{O.5}) is true for every $R \in \mathcal{R}$. We now show that with an appropriate choice of $\mathcal{R}$, we extend (\ref{O.5}) to any $R>0$: Picking 
\begin{align*}
\mathcal{R}:= \{ 2^{-n}, \ \ n\in \mathbb{N} \} \cup \mathbb{N},
\end{align*}
for every $R>0$ there exist $R_1, R_2 \in \mathcal{R}$ such that $R_1 \leq R \leq R_2$ with
$\frac {R} {R_1}, \frac{R_2}{R} \leq 2 $. Thus
\begin{align*}
\langle \int_{R< |x-y|< 2R}& |G(a; x, y)|^2 \rangle \\
&\leq  \langle \int_{R_1< |x-y|< 2R_1} |G(a; x, y)|^2 \rangle + \langle \int_{R_2< |x-y|< 2R_2} |G(a; x, y)|^2 \rangle\\
&\lesssim R_1^{4-d} + R_2^{4-d} \lesssim R^{4-d}.
\end{align*}
The same reasoning holds for (\ref{O.5bis}) and (\ref{O.5tris}). Moreover, since the previous argument may be adapted to any fixed $R \simeq 1$, for convenience in the next estimates, we prove (\ref{O.5}),(\ref{O.5bis}) and (\ref{O.5tris}) with $R=3$.

\medskip

We start with inequality (\ref{O.5tris}): We claim that it is enough to prove that for almost every $y\in\mathbb{R}^d$ and $\delta <<1$,
\begin{align}
\fint_{|y'-y|< \delta} \langle \int_{|x-y'|> 3} |\nabla\nabla G(a; x, y')|^2 \rangle \lesssim 1,\label{Cor13a}
\end{align}
Indeed, using (\ref{Teo1A}) we may send $\delta \rightarrow 0$ and conclude by Lebesgue Differentiation Theorem.\\

\medskip

We thus prove (\ref{Cor13a}): We take the average $\langle \cdot \rangle$ into inequality (\ref{Teo1A}) with $z=y$, $R= 2$ and, after integrating in the  $x$ and $y'$-variables, we obtain
\begin{align}
\langle \int_{|y'-y|< 1}\int_{|x-y|> 2} |\nabla\nabla G( a ; x, y')|^2 \rangle \lesssim 1
\end{align}
and also
\begin{align}
\langle \int_{|y'-y|< 1}\int_{|x-y'|> 3} |\nabla\nabla G( a ; x, y')|^2 \rangle \lesssim 1.
\end{align}
We now consider  $n \sim \delta^{-d}$ disjoint balls of radius $\delta << 1$ centred in $\{w_i\}_{i=1}^{n}$ points and contained in the unitary ball centred at the origin:
The previous inequality yields
\begin{align}\label{Cor14}
\sum_{i=1}^n \langle \int_{|y'-w_i-y|< \delta} \int_{|x-y'|> 3} & |\nabla\nabla G( a ; x, y')|^2 \rangle\notag\\
&\leq \langle \int_{|y'-y|< 1}\int_{|x-y'|> 3} |\nabla\nabla G( a ; x, y')|^2 \rangle \lesssim 1.
\end{align}
Moreover, thanks to (\ref{Ex.29b}) and stationarity, we rewrite the l.h.s. of the previous inequality as
\begin{align*}
\sum_{i=1}^n \langle \int_{|y'-w_i-y|< \delta}& \int_{|x-y'|> 3} |\nabla\nabla G( a ; x, y')|^2 \rangle\\
= &\sum_{i=1}^n \langle \int_{|y'-w_i-y|< \delta} \int_{|x-y'|> 3} |\nabla\nabla G( a ; x-w_i, y'-w_i)|^2 \rangle, 
\end{align*}
and, by the change of coordinates $x= x-w_i$ and $y'= y'- w_i$, as
\begin{align*}
\sum_{i=1}^n \langle \int_{|y'-w_i-y|< \delta}& \int_{|x-y'|> 3} |\nabla\nabla G( a ; x, y')|^2 \rangle\\
= & n \langle \int_{|y'-y|< \delta} \int_{|x-y'|> 3} |\nabla\nabla G( a ; x, y')|^2 \rangle\\
\simeq &\delta^{-d} \langle \int_{|y'-y|< \delta} \int_{|x-y'|> 3} |\nabla\nabla G( a ; x, y')|^2 \rangle .
\end{align*}
Inserting this into the l.h.s. of (\ref{Cor14}) allows to conclude (\ref{Cor13a}) and thus establish (\ref{O.5tris}).\\

\medskip

The bound (\ref{O.5bis}) for $\nabla G$ follows analogously from inequality (\ref{Teo1A2}). To show (\ref{O.5bis}) also for $\nabla_y G$ we use Sobolev's inequality in $\{|z-y|> 2 \}$ 
\begin{align*}
\biggl( \int_{|x-y|> 2} |\nabla_yG(a;x,y)|^{\frac{2d}{d-2}} \biggr)^{\frac{d-2}{2d}} \lesssim \biggl( \int_{|x-y|> 2} |\nabla\nabla G(a;x,y)|^{2} \biggr)^{\frac{1}{2}},
\end{align*}
together with H\"older's inequality
\begin{align*}
\int_{2<|x-y|<4} |\nabla_yG(a;x,y)|^{2} \lesssim \biggl( \int_{|x-y|>2} |\nabla_y G(a;x,y)|^{\frac{2d}{d-2}} \biggr)^{\frac{d-2}{d}},
\end{align*}
and get
\begin{align}\label{C17}
\langle \int_{2<|x-y|<4} &|\nabla_y G(a;x,y)|^{2} \rangle \lesssim \langle \int_{|x-y|>2} |\nabla\nabla G(a;x,y)|^2 \rangle \stackrel{(\ref{O.5tris})}{\lesssim} 1.
\end{align}
Since as we argue above we may assume that (\ref{Ex.24b}) holds for almost every $x, y \in \mathbb{R}^d$, $\langle \cdot \rangle$-almost every $a\in \Omega$ and on a countable set of radii, we infer that from the above inequality we have also for every $n \in \mathbb{N}$ that
\begin{align*}
\langle \int_{2^n<|x-y|<2^{n+1}} |\nabla_yG(a;x,y)|^{2} \rangle\lesssim 2^{n(2-d)},
\end{align*}
so that summing over $n\in \mathbb{N}$ we conclude (\ref{O.5bis}) also for $\nabla_y G$. Inequality (\ref{O.5}) follows from (\ref{O.5bis}) for $\nabla G$, again by Sobolev's and H\"older's inequality.

\medskip

{ {\sc Step 3}: Spatially pointwise estimates.}
We now post-process (\ref{O.5}), (\ref{O.5bis}) and (\ref{O.5tris}) to obtain  (\ref{C2A}),(\ref{C2B}) and (\ref{C2C}). Reasoning as in Step 2, without loss of generality it suffices to prove (\ref{C2A}),(\ref{C2B}) and (\ref{C2C}) for almost every
$x,y \in \mathbb{R}^d$ with $10< |x-y| < 12$.
Let $w \in \mathbb{R}^d$ be fixed; We claim that if  $u(a;\cdot)$ is a-harmonic in $\{|y-w| \le 8 \}$ for $\langle \cdot \rangle$-almost every $a\in \Omega$, then for almost every $y \in \mathbb{R}^d$ with $|y-w|<1$,
it holds
\begin{align}
\langle |u(a;y)| \rangle+ \langle |\nabla &u(a;y)|\rangle \notag\\
&\lesssim \langle\int_{2<|z-y|< 6}|\nabla u(a; z)|^2\rangle^\frac{1}{2} + \langle\int_{2<|z-y|< 6}|u(a;z)|^2\rangle^\frac{1}{2}.\label{L.20}
\end{align}
Before proving (\ref{L.20}), we show how to conclude the argument for (\ref{C2A}),(\ref{C2B}) and (\ref{C2C}). We start with (\ref{C2A}): By symmetry (\ref{Ex.28b}) and (\ref{statCor}), for almost every $x \in \mathbb{R}^d$ and
$\langle \cdot \rangle$-almost every $a\in \Omega$ the application $u(a;y)=G(a;x,\cdot)$ is a-harmonic in $ |x-y| > 2$. Moreover, since we may select in this domain $N \lesssim 1$ balls of radius $8$, centred in $\{ w_i\}_{i=1}^N$ 
points such that their union covers the annulus $10< |x-y| < 12$, estimate (\ref{L.20}) yields that for almost every $y$ such that $\{10< |x-y| < 12\}$ 
\begin{align*}
\langle |G(a;x,y)&|\rangle\\
\lesssim & \langle\int_{2<|z-y|<6}|G(a;x,z)|^2\rangle^\frac{1}{2}+\langle\int_{2<|z-y|<6}|\nabla_zG(a;x,z)|^2\rangle^\frac{1}{2}\\
\stackrel{(\ref{Ex.28b})}{\lesssim}& \langle\int_{2 <|z-y|< 6}|G(a;x,z)|^2 \rangle^\frac{1}{2}+ \langle \int_{2 <|z-y|< 6}|\nabla G(a;z,x)|^2\rangle^\frac{1}{2}.
\end{align*}
In addition, by the inclusion 
$$
\{10 <|x-y| < 12\} \cap \{2<|z-y|< 6\} \subset \{10 <|x-y| < 12\} \cap \{ 4<|z-x|< 18\}
$$
we conclude from the previous inequality that
\begin{align*}
\langle |G(a;x,y)&|\rangle \stackrel{}{\lesssim}& \langle\int_{4 <|z-x|< 18}\hspace{-0.3cm}|G(a;x,z)|^2 \rangle^\frac{1}{2}+ \langle \int_{|z-x|> 4}\hspace{-0.3cm}|\nabla G(a;z,x)|^2\rangle^\frac{1}{2}\stackrel{(\ref{O.5})-(\ref{O.5bis})}{\lesssim 1},
\end{align*}
i.e. bound (\ref{C2A}).

\medskip

In order to have also (\ref{C2B})-(\ref{C2C}), we consider $u(a;y)={\nabla}G(a;x,y)$ which, thanks to symmetry (\ref{Ex.28b}) and (\ref{statCorb}), for almost every $x\in \mathbb{R}^d$ and $\langle \cdot \rangle$-almost every $a\in \Omega$
is a-harmonic in $\{|y-x| > 2\}$. Therefore, reasoning as for bound (\ref{C2A}), we may apply estimate (\ref{L.20}) to this choice of $u$ and get that for almost every $y$ such that $\{10< |x-y| < 12\}$
\begin{align*}
\langle&|\nabla G(a;x,y)| \rangle + \langle |\nabla\nabla G(a;x,y)| \rangle \\
&\lesssim \langle\int_{2<|z-y|<6}|\nabla G(a;x,z) |^2\rangle^\frac{1}{2}+\langle\int_{2<|z-y|<6}|\nabla\nabla G(a; x,z)|^2\rangle^\frac{1}{2}\\
&\stackrel{(\ref{Ex.28b})}{\lesssim} \langle\int_{2<|z-y|<6}\hspace{-0.3cm}|\nabla_x G(a;z,x) |^2\rangle^\frac{1}{2}+\langle\int_{2<|z-y|<6}\hspace{-0.3cm}|\nabla\nabla G(a; x,z)|^2\rangle^\frac{1}{2}\\
&{\lesssim} \langle\int_{|z-x|> 4}\hspace{-0.3cm}|\nabla_x G(a;z,x) |^2\rangle^\frac{1}{2}+\langle\int_{|z-x|> 4}\hspace{-0.3cm}|\nabla\nabla G(a; x,z)|^2\rangle^\frac{1}{2}\stackrel{(\ref{O.5bis})-(\ref{O.5tris})}{\lesssim} 1.
\end{align*}

\medskip

We now argue that (\ref{L.20}) is implied by the following deterministic result: Let $w \in \mathbb{R}^d$ and a family of applications $\{u(a; \cdot)\}_{a\in \Omega}$ a-harmonic in $\{|y-w| < 8\}$- Then for any fixed $a\in \Omega$ we have for almost every $\{|y-w|<1\}$
\begin{align}
|u(a;y)|\lesssim\biggl(\int_{2<|z-y|< 6}&|\nabla G(a;z,y)|^2+|G(a;z,y)|^2\biggr)^{\frac{1}{2}}\notag\\
&\times \biggl(\int_{2<|z-y|< 6}|u(a;z)|^2+|\nabla{u}(a;z)|^2\biggr)^{\frac{1}{2}}\label{W10}.
\end{align}
and
\begin{align}\label{W11}
|\nabla u(a;y)|\lesssim  \biggl(\int_{2<|z-y|<6}&|\nabla\nabla G(a;z,y)|^2+|\nabla_y G(a;z,y)|^2\biggr)^{\frac{1}{2}}\notag \\
 & \times\biggl(\int_{2< |z-y|<6}|u(a;z)|^2+|\nabla{u}(a;z)|^2\biggr)^{\frac{1}{2}}.
\end{align}
Indeed, arguing again by separability of $L^1(\Omega)$, we also infer that the previous bounds hold for almost every $y\in \mathbb{R}^d$ such that $|y-w| <1$ and for $\langle\cdot \rangle$-almost every $a\in \Omega$. Therefore, we may take  in (\ref{W10}) and (\ref{W11}) the expected value, use Cauchy-Schwarz's inequality in $\langle\cdot\rangle$ and estimate
\begin{align*}
\langle |u(a;y)| \rangle \lesssim & \biggl(\langle \int_{2<|z-y|<6}|G(a;y,z)|^2 + |\nabla_zG(a;y,z)|^2 \rangle \biggr)^{\frac{1}{2}}\\
&\hspace{1cm}\times \biggl(\langle \int_{2<|z-y|<6}|\nabla u(a;z)|^2 + |u(a;z)|^2 \rangle \biggr)^{\frac{1}{2}}\\
\stackrel{(\ref{Ex.28b})}{\lesssim }& \biggl(\langle \int_{2<|z-y|<6}|G(a;z,y)|^2 + |\nabla_zG(a;y,z)|^2 \rangle \biggr)^{\frac{1}{2}}\\
&\hspace{1cm}\times \biggl(\langle \int_{2<|z-y|<6}|\nabla u(a;z)|^2 + |u(a;z)|^2 \rangle \biggr)^{\frac{1}{2}},
\end{align*}
\begin{align*}
\langle |\nabla u(a;y)| \rangle \lesssim & \biggl(\langle \int_{2<|z-y|<6}|\nabla G(a;y,z)|^2 + |\nabla\nabla G(a;y,z)|^2 \rangle\biggr)^{\frac{1}{2}}\\
& \hspace{1cm}\times \biggl(\langle\int_{2<|z-y|<6}|u(a;z)|^2 + |\nabla u(a;z)|^2 \rangle\biggr)^\frac{1}{2}\\
\stackrel{(\ref{Ex.28b})}{\lesssim} &\biggl(\langle \int_{2<|z-y|<6}|\nabla_y G(a;z,y)|^2 + |\nabla\nabla G(a;y,z)|^2 \rangle\biggr)^{\frac{1}{2}}\\
& \hspace{1cm}\times \biggl(\langle\int_{2<|z-y|<6}|u(a;z)|^2 + |\nabla u(a;z)|^2 \rangle\biggr)^\frac{1}{2}.
\end{align*}
Inequality (\ref{L.20}) follows by (\ref{O.5}),(\ref{O.5bis}) and (\ref{O.5tris}).

\medskip

It thus only remains to tackle (\ref{W10}) and (\ref{W11}): Without loss of generality, let us fix $w=0$. For a cut-off function  $\eta$ of $\{|y|\le 3\}$ in $\{|y|\le 5\}$ we may define\footnote{ Since we are working in the systems' setting, to be notationally rigorous we should write $u\otimes \nabla\eta$ instead of $u\nabla\eta$.}
\begin{equation}\nonumber
v=\eta u,\quad f=-\nabla\eta\cdot a\nabla u,\quad g=-a(\nabla\eta u)
\end{equation}

and have, by our assumption that $u(a; \cdot)$ is a-harmonic in $\{|y| < 8 \}$, that 
\begin{align*}
-\nabla\cdot a\nabla v =-\nabla\cdot a( \nabla\eta u+\eta\nabla u)=\nabla\cdot g+f.
\end{align*}
Hence, the representation formula (see Remark 2) yields for almost every $y$ with $|y|<1$ that
\begin{align*}
u(a;y)=&\int\bigl(G(a;y,z)f(a;z) - \nabla_zG(a;y,z)\cdot g(a;z)\bigr)dz,\\
\nabla u(a;y)=&\int\bigl(\nabla G(a;y,z)f(a;z) - \nabla\nabla G(a;y,z)\cdot g(a;z)\bigr)dz.
\end{align*}
By definition of $\eta$, $v$, $f$, and $g$ together with conditions (\ref{KC})-(\ref{bdd}), H\"older's inequality implies that
\begin{align*}
|u(a;y)|\lesssim & \biggl(\int_{3<|z|<5}|G(a;y,z)|^2\biggr)^{\frac{1}{2}}\biggl(\int_{3<|z|<5}|\nabla u(a;z)|^2\biggr)^{\frac{1}{2}}\\
&+\biggl(\int_{3<|z|<5}|\nabla_zG(a;y,z)|^2 \biggr)^{\frac{1}{2}}\biggl(\int_{3<|z|< 5}|u(a;z)|^2\biggr)^{\frac{1}{2}},
\end{align*}
\begin{align*}
|\nabla u(a;y)|\lesssim & \biggl(\int_{3<|z|<5}|\nabla G(a;y,z)|^2\biggr)^{\frac{1}{2}}\biggl(\int_{3<|z|<5}|\nabla u(a;z)|^2\biggr)^\frac{1}{2}\\
&+\biggl(\int_{3<|z|<5}|\nabla\nabla G(a;y,z)|^2\biggr)^{\frac{1}{2}}\biggl(\int_{3<|z|<5}|u(a;z)|^2\biggr)^\frac{1}{2},
\end{align*}
and thus (\ref{W10}) and (\ref{W11}).

\bigskip

\section{Fourier Approach}\label{SF}
Here we summarise how the Fourier method developed in \cite{cn2} can be used to prove Corollary \ref{Cor2} provided the  system is uniformly elliptic, so we shall assume that both (\ref{bdd}) and (\ref{StE}) hold.  
The method is based on a representation of the Fourier transform of $G$ in terms of a function $\Phi:\Omega\times\mathbb{R}^d\rightarrow \mathcal{L}(Y^d,Y)$, which satisfies an elliptic PDE on $\Omega$.

 To define the PDE for $\Phi$ we introduce some notation. First observe that  $\xi \in\mathbb{R}^d$  can  be regarded as being  in the space $\mathcal{L}(Y,Y^d)$. In that case we denote its adjoint by $\xi^*\in\mathcal{L}(Y^d,Y)$. Similarly the gradient operator $D$ acts on functions $F:\Omega\rightarrow Y$ to yield a function $DF: \Omega\rightarrow Y^d$ with $(DF)_i:= D_iF$ for $i=1,...,d$ defined as
\begin{align*}
D_iF(a):= \lim_{h \downarrow 0}\frac{F(a( \cdot +he_i))- F(a)}{h},
\end{align*} 
with $e_i$ denoting the standard ith-versor in $\mathbb{R}^d$.
We denote by $D^*$ the corresponding divergence operator, which takes a function $F:\Omega\rightarrow Y^d$ to  a function $D^*F: \Omega\rightarrow Y$. Using this notation, the function $\Phi$ is the solution to the equation
\begin{equation} \label{A10}
\mathcal{P}(D^*+i\xi^*)a(D-i\xi)\Phi(a,\xi)=-\mathcal{P}(D^*+i\xi^*)a \ , \quad a\in \Omega, \ \xi \in\mathbb{R}^d,
\end{equation}
where $\mathcal{P}$ is the projection operator on $L^2(\Omega)$ orthogonal to the constant.   We can see using (\ref{StE}) that the function  $a\rightarrow (D-i\xi)\Phi(a,\xi)\in \mathcal{L}(Y^d,Y^d)$ is in  $L^2(\Omega,\mathcal{L}(Y^d,Y^d))$.  To do this we apply the adjoint $\Phi(a,\xi)^*\in \mathcal{L}(Y,Y^d)$ to (\ref{A10}) and take the expectation. This yields the inequality
\begin{equation} \label{B10}
\|(D-i\xi)\Phi(\cdot,\xi)\|_{L^2(\Omega,\mathcal{L}(Y^d,Y^d))} \ \le \ \frac{1}{\lambda} \quad \forall \xi\in\mathbb{R}^d \ .
\end{equation}
Next we define a function $q:\mathbb{R}^d\rightarrow \mathcal{L}(Y^d,Y^d)$ by
\begin{equation} \label{C10}
q(\xi)= \langle  \ a \ \rangle+\langle  \ a(D-i\xi)\Phi(a,\xi) \ \rangle \ .
\end{equation}
Then from (\ref{KC}), (\ref{StE}) it follows that $q(\xi)$ is Hermitian for $\xi\in\mathbb{R}^d$ and
\begin{equation} \label{D10}
\lambda |y|^2 \le \ y\cdot q(\xi)y \le |y|^2 \quad \forall y\in Y^d \  .
\end{equation}
From (\ref{D10}) we conclude that $\xi^* q(\xi)\xi\in \mathcal{L}(Y,Y)$ is invertible provided $\xi\ne 0$. \\

Generalising the representation of  \cite{cn2} (see equation (2.4) of \cite{cn2} or equation (8.1) of \cite{cs})  to the case of systems, we see that $\nabla_x G(a;x,y)$ is given by the Fourier inversion formula
\begin{multline} \label{E10}
\nabla_xG(a;x,y) \ = \\
 -\frac{i}{(2\pi)^d}\int_{\mathbb{R}^d} d\xi \  e^{-i(x-y)\cdot\xi} \ 
\left\{\xi+(D-i\xi)\Phi(a(\cdot+x), \xi)\xi \ \right\}\left[\xi^*q(\xi)\xi\right]^{-1}  \ .
\end{multline}
Let $\mathcal{H}$ be a Hilbert space with norm $\|\cdot\|$ and consider functions  $f:\mathbb{R}^d\rightarrow \mathcal{H}$.  For $1\le p<\infty$ we define the weak spaces $L_w^p(\mathbb{R}^d,\mathcal{H})$ in the usual way.  That is $f\in L_w^p(\mathbb{R}^d,\mathcal{H})$ if
\begin{equation} \label{F10}
m\{ \xi\in\mathbb{R}^d \ : \ \|f(\xi)\| > k \ \} \ \le \ \frac{C}{k^p} \quad \forall k>0. 
\end{equation}
The norm of $f$, which we denote by $\|f\|_{w,p}$, is the smallest constant $C$ for which (\ref{F10}) holds.  The following lemma can be proved in the the same way as Lemma 3.5 of \cite{cn2}.
\begin{lemma} \label{G10}
Let $\mathcal{H}=L^2(\Omega,\mathcal{L}(Y^d,Y^d))$ and $f:\mathbb{R}^d\rightarrow \mathcal{H}$ be the function 
\begin{equation} \label{H10}
f(\xi) \ = \  \partial_\xi^n(D-i\xi)\Phi(\cdot,\xi) \ = \ \left[\prod_{i=1}^d \left(\frac{\partial}{\partial \xi_i}\right)^{n_i}\right] (D-i\xi)\Phi(\cdot,\xi) \ ,
\end{equation}
where $n=(n_1,\cdots,n_d)$ satisfies $0<|n|<d/2$.  Then $f\in L_w^p(\mathbb{R}^d,\mathcal{H})$ with $p=d/|n|$ and $\|f\|_{w,p}\le C_d(\lambda)$, where the constant $C_d(\lambda)$ depends only on $\lambda$ and $d$. 
\end{lemma}
{\sc Comments on Lemma \ref{G10}}.\\
We give below the main ideas to prove Lemma \ref{G10} as in \cite{cn2} (Lemma 3.5). We use a scalar notation, but all the arguments and techniques hold also for systems.
The proof mainly relies on the representation formula for $\Phi$ [see Lemma 3.2 in \cite{cn2}]
\begin{equation}\label{Q10}
(D-i\xi)\Phi= (1-\mathcal{P}T_{b,\xi})^{-1}T_{\xi}\mathcal{P}(a)
\end{equation}
where
\begin{eqnarray}\label{U10}
T_{\xi}(\rho)(a):=\int_{\mathbb{R}^d} \nabla^2 G(I;x,0) e^{-ix\cdot\xi}\left[\rho\right](a(\cdot + x)) dx,\\\notag
T_{b,\xi}(\rho)(a):=\int_{\mathbb{R}^d} \nabla^2 G(I;x,0) e^{-ix\cdot\xi}\left[b\rho\right](a(\cdot + x)) dx\\\notag
\end{eqnarray}
if $\rho$ is a random variable $\rho: \Omega\rightarrow \mathbb{R}^d$ and $b= I-a$.
This can be obtained from (\ref{A10}) by a standard perturbation argument applied to the operator $(D^*+i\xi^*)a(D-i\xi)$.
Being $||T_{b,\xi}||\leq ||b||_{L^\infty(\Omega)}<1$, (\ref{Q10}) is well defined and its Neumann series converges.
Moreover, derivatives with respect to $\xi$ of (\ref{U10}) can be explicitly written as
\begin{eqnarray}\label{V10}
\partial^n_{\xi}T_{\xi}(\rho)(a):=(-i)^{|n|}\int_{\mathbb{R}^d} x^{n}\nabla^2 G(I;x,0) e^{-ix\cdot\xi}\left[\rho\right](a(\cdot + x)) dx\\\notag
\partial^n_{\xi}T_{b,\xi}(\rho)(a):=(-i)^{|n|}\int_{\mathbb{R}^d} x^{n}\nabla^2 G(I;x,0) e^{-ix\cdot\xi}\left[b\rho\right](a(\cdot + x)) dx.\\\notag
\end{eqnarray}
Note that in a rigorous formulation, to assure the convergence of the integrals in (\ref{U10}), one should first work with the massive Green function $G_T$ associated to the operator $T^{-1} - \nabla\cdot a \nabla$ (\cite{GO}, Definition 2.4) and then pass to the limit $T \rightarrow +\infty$ and obtain $G$. To keep notation lean, we neglect this issue. For the same reason, we restrict our attention to the case $d=3$.\\
Due to (\ref{Q10}) and (\ref{V10}), $f(\xi)$ in (\ref{H10}) is equal to a sum of terms containing derivatives as in (\ref{V10}).
Hence, for $d=3$ we get ($|n|=1<\frac{3}{2}$)
\begin{align}\label{W10b}
\partial_{\xi_i}(D-i\xi)\Phi=&\ (1-\mathcal{P}T_{b,\xi})^{-1}\partial_{\xi_i}T_{\xi}\mathcal{P}(a)\notag\\
& +(1-\mathcal{P}T_{b,\xi})^{-1}\mathcal{P}\partial_{\xi_i}T_{b,\xi}(1-\mathcal{P}T_{b,\xi})^{-1}T_{\xi}\mathcal{P}(a).
\end{align}
More precisely, each term on the r.h.s. of (\ref{W10b}) may be rewritten an operator acting on $\partial_{\xi_i}\left[\hat{G}(I;\xi,0)\xi_j\xi_k\right]\in L^3_w(\mathbb{R}^3)$, i.e.
\begin{align}
&(1-\mathcal{P}T_{b,\xi})^{-1}\partial_{\xi_i}T_{\xi}\mathcal{P}(a)= \mathcal{S}^1\bigl(\partial_{\xi_i}\left[\hat{G}(I;\xi,0)\xi_j\xi_k\right] \bigr), \label{W10c}\\
&(1-\mathcal{P}T_{b,\xi})^{-1}\mathcal{P}\partial_{\xi_i}T_{b,\xi}(1-\mathcal{P}T_{b,\xi})^{-1}T_{\xi}\mathcal{P}(a)\notag\\
&\hspace{4.2cm}= \mathcal{S}^2\bigl(\partial_{\xi_i}\left[\hat{G}(I;\xi,0)\xi_j\xi_k\right]\bigr). \label{W10d}
\end{align}
Lemma \ref{G10} follows once it is proved that $\mathcal{S}^1$ and $\mathcal{S}^2$ are bounded from $L^p_w(\mathbb{R}^3)$ to $L^p_w(\mathbb{R}^3, L^2(\Omega))$ for every $p\in (2,+\infty)$.
The most challenging operator is $\mathcal{S}^2$: The one associated to the second term on the r.h.s. of \eqref{W10b}, where the derivative falls on $(1-\mathcal{P}T_{b,\xi})^{-1}$. To deal with it, it is convenient to first prove its boundedness from $L^p(\mathbb{R}^3)$ to $L^p(\mathbb{R}^3, L^2(\Omega))$ for $p\in \{2,+\infty\}$ and then use Hunt's interpolation theorem.
The case $p=+\infty$ follows by an application of Bochner's theorem (cf. \cite{cn2}, formula (3.14)), while for $p=2$ the main idea relies on the fact that $(1-\mathcal{P}T_{b,\xi})^{-1}$ can be written in Neumann series and every term can be explicitly expressed. Once explicit, one can recognise that each term acts on a function by essentially taking multiple convolutions of its Fourier transform with the Hessian of the standard Green function $G(I;x,y)$. Such a convolution kernel does not increase the (Frobenius) norm of the function.\\
A generalisation to higher dimensions is in the same spirit but has to deal with more involved operators $\mathcal{S}^{1}, \mathcal{S}^2,...\mathcal{S}^N$ (for $N=N(d)$).
The upper bound for the number of derivatives $|n|$ is related to the strict condition $p>2$ which ensures the boundedness of the operators between the $L^p$-weak spaces (see Lemma 3.7 and Lemma 3.9 in \cite{cn2}).

\medskip

The following lemma implies (\ref{O.5bis}) of Corollary \ref{Cor2} provided $d$ is odd. In order to prove (\ref{O.5bis}) when $d$ is even we would need to extend Lemma \ref{G10} to include fractional derivatives, something that is also required in  \cite{cn2}. 
\begin{lemma} \label{HH10}
Let $d\ge 3$ and $n=(n_1\cdots,n_d)$  be a non-negative integer tuple such that $d/2-1<|n|<d/2$. Then for any $R>0$ there exists a constant $C_d(\lambda)$ depending only on $\lambda,d$ such that 
\begin{equation} \label{I10}
\left\langle\int_{|x|<R} x^{2n}|\nabla_xG(\cdot;x,0)|^2 \ dx\right\rangle \ \le \   C_d(\lambda) R^{2(|n|+1)-d} \ .
\end{equation} 
\end{lemma}
\begin{proof}
We have from (\ref{E10}) on integration by parts that
\begin{multline} \label{J10}
x^n\nabla_xG(a;x,0) \ = \ \int_{\mathbb{R}^d} d\xi \  e^{-ix\cdot\xi} \ \frac{f(a(\cdot+x), \xi)}{|\xi|} \ , \quad {\rm where} \\
f(a,\xi)   \ = \   |\xi|\frac{i^{3-|n|}}{(2\pi)^d}\partial_\xi^n\left[\left\{\xi+(D-i\xi)\Phi(a, \xi)\xi \ \right\}\left\{\xi^*q(\xi)\xi\right\}^{-1}\right] \ .
\end{multline}
Taking $\mathcal{H}=L^2(\Omega,\mathcal{L}(Y,Y^d))$, it follows from (\ref{C10}),  (\ref{D10}) and Lemma \ref{G10} that $f:\mathbb{R}^d\rightarrow\mathcal{H}$ is in  $ L_w^p(\mathbb{R}^d,\mathcal{H})$ with $p=d/|n|$, and $\|f\|_{w,p}\le C_d(\lambda)$ for some constant $C_d(\lambda)$ depending only on $\lambda,d$. Let $\phi$ be a cut-off function for $\{|x|<1\}$ in $\{|x|<2\}$.  Then from  (\ref{J10}) we have that
\begin{multline} \label{K10}
\left\langle\int_{\mathbb{R}^d} \phi(x/R)x^{2n}|\nabla_xG(\cdot;x,0)|^2 \ dx\right\rangle \ = \\
 \int_{\mathbb{R}^d\times\mathbb{R}^d} d\xi \ d\xi'  \ R^d \hat{\phi}(R(\xi-\xi')) \left\langle \frac{f(\cdot, \xi)^*}{|\xi|}\frac{f(\cdot, \xi')}{|\xi'|} \right\rangle \ .
\end{multline}
It follows from (\ref{K10}) that
\begin{multline} \label{L10}
\left| \left\langle\int_{\mathbb{R}^d} \phi(x/R)x^{2n}|\nabla_xG(\cdot;x,0)|^2 \ dx\right\rangle\right| \ \le \\ 
 \int_{\mathbb{R}^d\times\mathbb{R}^d} d\xi \ d\xi'  \ R^d |\hat{\phi}(R(\xi-\xi'))| \frac{g(\xi)}{|\xi|}\frac{g(\xi')}{|\xi'|} \ ,
\end{multline} 
where $g:\mathbb{R}^d\rightarrow \mathbb{R}^+$ is in $L_w^p(\mathbb{R}^d)$ with $p=d/|n|$. 

We can estimate the RHS of (\ref{L10}) by using the inequality
\begin{equation} \label{M10}
\int_E g(\xi)^q \ d\xi \ \le \ C_qm(E)^{1-q/p} \|g\|_{p,w}^{q/p} \ ,  \quad \forall  \ {\rm measurable \ } E\subset \mathbb{R}^d, \ 1\le q<p,
\end{equation}
where the constant $C_q$ diverges as $q\rightarrow p$.  We consider for any $A>0$ the integral
\begin{multline} \label{N10}
R^d \int_{|\xi-\xi'|<A/R} d\xi \ d\xi'  \ \frac{g(\xi)}{|\xi|}\frac{g(\xi')}{|\xi'|} \ \le \ 
R^d\left[ \int_{|\xi|<4A/R} d\xi  \ \frac{g(\xi)}{|\xi|}\right]^2+ \\
2R^d \int_{|\xi|>2A/R, \ |\zeta|<A/R} d\xi \ d\zeta  \ \frac{g(\xi)g(\xi+\zeta)}{|\xi|^2} \ .
\end{multline}
Taking $q=1$ in (\ref{M10}), we see that the first term on the RHS of (\ref{N10}) is bounded by
\begin{equation} \label{O10}
CR^d \left(\frac{A}{R}\right)^{2(d-1-d/p)}\left[ \sum_{j=0}^\infty 2^{-j(d-1-d/p)}\right]^2 \ ,
\end{equation}
where $C$ is a constant depending only on $d$ and $\|g\|_{p,w}$.  Taking $p=d/|n|$, we see that the sum in (\ref{O10}) converges provided  $|n|<d-1$. If this is the case then   the first term on the RHS of (\ref{N10}) is bounded by $CA^{2(d-1-|n|)}R^{2(|n|+1)-d}$, where $C$ depends only on $\lambda,d$.  To estimate the second term in (\ref{N10}) we take $q=2$ in (\ref{M10}). Thus it is bounded by 
\begin{equation} \label{P10}
CR^d \left(\frac{A}{R}\right)^{2(d-1-d/p)} \sum_{j=0}^\infty 2^{j(d-2-2d/p)} \ ,
\end{equation}
where $C$ is a constant depending only on $d$ and $\|g\|_{p,w}$. Taking $p=d/|n|$ as before, we see that the sum in (\ref{P10}) converges provided  $|n|>d/2-1$.  We have therefore shown that if $d/2-1<|n|<d/2$ then the LHS of (\ref{N10}) is bounded by  $CR^{2(|n|+1)-d}A^{2(d-1-|n|)}$, where $C$ is a constant depending only on $\lambda,d$.  Now using the fact that for any $k=1,2,..,$ there is a constant $C_k$ such that $|\hat{\phi}(\zeta)|\le C_k/[1+|\zeta|^k]$, we conclude from (\ref{L10}) that (\ref{I10}) holds.  
\end{proof}

\section{Appendix}\label{Ax}
{ \sc Proof of Lemma \ref{U}.} Let $a\in\Omega$ and the domain $D$ be fixed, and let us assume that $G^{(1)}_D(a;\cdot,\cdot)$ and $G^{(2)}_{D}(a; \cdot, \cdot)$ are two Green's functions. For the sake of simplicity, we skip the argument $a$ in $G_D^{(1)}$ and $G_D^{(2)}$.
We argue similarly to \cite{GO}, Subsection A.3. : Let $\rho\in C^{\infty}_0(\{|y| < 1\})$. For $u := G^{(1)}_D- G^{(2)}_D$, we define for $x,y \in \mathbb{R}^d$
\begin{align*}
\hat u( x, y)&:=  \int \rho(y-y') u(x,y') \, dy'\\
&= \int \rho(y'-y) \bigl(G^{(1)}_{D}(a; x, y')-  G^{(2)}_{D}(a; x, y')\bigr) \, dy'.
\end{align*}
We first argue that for every $y\in \mathbb{R}^d$,
\begin{align}\label{Pu2}
\begin{cases}
 -\nabla\cdot a\nabla \hat u(\cdot, y) = 0 \mbox{ \hspace{0.5cm} in $D$}\\
 \ \ \hat u(\cdot , y) = 0 , \hspace{1.85cm} \mbox{in $\partial D$}.
\end{cases}
\end{align}
By definition of $u$, it indeed holds that for almost every $y'\in \mathbb{R}^d$ and every $\zeta \in C^\infty_0(D)$
\begin{align*}
\int \nabla \zeta(x) \cdot a(x) \nabla u(x,&y')\, dx\\
& = \int \nabla \zeta(x) \cdot a(x) \bigl(\nabla G^{(1)}_D(x,y')- \nabla G^{(2)}_D(x,y') \bigr)\, dx\stackrel{(\ref{TG})}{ = }0.
\end{align*}
Multiplying this with $\rho(y' - y )$ and integrating in $y'$, we thus have for every $\zeta \in C^\infty_0(D)$ that
\begin{align*}
 0= \int  \int \nabla \zeta(x)& \cdot a(x) \nabla\bigl(\rho(y'-y)u(x,y')\bigr) \, dx \, dy'\\
 & = \int \nabla \zeta(x) \cdot a(x) \nabla \bigl(\int  \rho(y'-y )u(x,y')\, dy'\bigr)\, dx,
\end{align*}
i.e. for every $y\in \mathbb{R}^d$, $\hat u(\cdot, y)$ solves the equation in (\ref{Pu2}). The boundary conditions are trivially satisfied since by definition of $G^{(1)}$ and $G^{(2)}$ we have
\begin{align*}
\int_{\mathbb{R}^d\setminus D} |\hat u(x,y)&| dx = \int_{\mathbb{R}^d\setminus D} |\int \rho(y'-y )u(x,y') dy'| dx\\
&\leq  \int |\rho(y'-y)| \int_{\mathbb{R}^d\setminus D}\bigl| G^{(1)}(x,y')-  G^{(2)}(x,y')\bigr|\, dx \, dy' \stackrel{(\ref{TG})}{=} 0.
\end{align*}

\medskip

We now claim that for every $y\in \mathbb{R}^d$ fixed it holds
\begin{align}\label{U1}
\int |\nabla \hat u(x,y)|^2\, dx < +\infty.
\end{align}
We start by noting that since
\begin{align*}
\int |\nabla \hat u(x,y)|^2\, dx=  \int |\int \rho(y'-y) \nabla u(x, y') \, dy' |^2 \, dx, 
\end{align*}
after smuggling into the r.h.s. the weight \mbox{$ \bigl( |x-y'|^{\alpha} \wedge 1 \bigr)^{\frac 1 2}$} and using H\"older's inequality in the $y'$-variable combined with \mbox{$\rho\in C^\infty_0(\{|y| < 1\})$}, this may be estimated by
\begin{align}\label{W2}
\int& |\nabla \hat u(x,y)|^2\notag\\
\leq & ||\rho||_\infty \int \int_{|y'-y|< 1}\bigl( |x-y'|^{\alpha} \wedge 1 \bigr)^{-1} \int_{|y' - y|< 1} \bigl( |x-y'|^{\alpha} \wedge 1 \bigr) |\nabla u (x, y')|^2\notag \\
=&||\rho||_\infty\int \int_{|y'-y|< 1} ( |x-y'|^{-\alpha} \vee 1)\int_{|y' - y|< 1}\bigl( |x-y'|^{\alpha} \wedge 1 \bigr) |\nabla u (x, y')|^2 .
\end{align}
For $i\in \{1,2 \}$, let $\alpha_i < d$ be the exponent in the inequality (\ref{WB}) for $G^{(i)}$. Without loss of generality, we assume that $\alpha_1 < \alpha_2 < d$ so that for every $z\in \mathbb{Z}^d$ we have
\begin{align}\label{WBu}
&\int_{|y-z|< 2 }\int_{|x-z|< 2} |x-y|^{\alpha_2}|\nabla u(x,y)|^2 \notag \\
&\leq \int_{|y-z|< 2 }\int_{|x-z|< 2} |x-y|^{\alpha_2}\bigl( |\nabla G^{(1)}(x,y)|^2 + |\nabla G^{(2)}(x,y)|^2 \bigr)\notag \\
&\stackrel{|x-y| < 4 , \ \alpha_1 < \alpha_2}{\lesssim}\int_{|y-z|< 2 }\int_{|x-z|< 2}|x-y|^{\alpha_1}|\nabla G^{(1)}(x,y)|^2 \notag\\
&\hspace{3.2cm}+ \int_{|y-z|< 2 }\int_{|x-z|< 2}|x-y|^{\alpha_2}|\nabla G^{(2)}(x,y)|^2 \stackrel{(\ref{WB})}{<} +\infty.
\end{align}
Since $\alpha= \alpha_2 <d$ we estimate for every $x\in \mathbb{R}^d$
\begin{align*}
\int_{|y'-y|< 1} ( |x-y'|^{-\alpha} \vee 1) \lesssim 1,
\end{align*}
we may reduce (\ref{W2}), with this choice of $\alpha=\alpha_2$, to
\begin{align*}
\int |\nabla \hat u(x,y)|^2 &\lesssim ||\rho||_\infty\int \int_{|y'- y|< 1} \bigl( |x-y'|^{\alpha} \wedge 1 \bigr) |\nabla u (x, y')|^2,
\end{align*}
and, by choosing $z\in \mathbb{Z}^d$ such that $\{|y'-y|<1 \} \subset \{|y' -z|<2\}$, to
\begin{align*}
&\int |\nabla \hat u(x,y)|^2 \lesssim ||\rho||_\infty\int \int_{|y'- z|< 2} \bigl( |x-y'|^{\alpha} \wedge 1 \bigr) |\nabla u (x, y')|^2.
\end{align*}
We now appeal to (\ref{WB}) and (\ref{WB2}) to conclude (\ref{U1}).\\
Thanks to (\ref{U1}), we may test the equation in (\ref{Pu2}) with $\hat u(\cdot, y)$ itself, and obtain by (\ref{KC}) and (\ref{TG}), that $\hat u(\cdot, y)=0$ almost everywhere in $\mathbb{R}^d$. 
It thus follows by definition of $\hat u$ that for every $y\in \mathbb{R}^d$ and $\zeta \in C^\infty_0(\mathbb{R}^d)$ we have
\begin{align*}
\int \int \rho(y'-y) \zeta(x) \bigl( G_D^{(1)}(x,y') - G_D^{(2)}(x,y') \bigr) \, dx \, dy'= 0.
\end{align*}
Therefore, by the arbitrariness of $y\in \mathbb{R}^d$ and of both $\zeta \in C^\infty_0(\mathbb{R}^d)$ and $\rho \in C^\infty_0(\{|y| < 1\})$ we conclude that $G^{(1)}_D= G^{(2)}_D$ for almost every $x,y \in \mathbb{R}^d$.

\bigskip

{ \sc Proof of Remark \ref{R.1} }.
We start  by showing that (\ref{R1}) implies that 
\begin{align}\label{r.3}
\sup_{|z|< R}|u(z)|\lesssim \biggl(\fint_{|z|< 8R}| u|^2 \biggr)^{\frac{1}{2}}
\end{align}
if $-\nabla\cdot a\nabla u=0$ in $\{ |x|< 8R \}$. And that (\ref{R2}), together with (\ref{R1}), yields in addition
\begin{align}\label{r.4}
\sup_{|z|< R}|\nabla u(z)|\lesssim \biggl(\fint_{|z|< 4R}|\nabla u|^2 \biggr)^{\frac{1}{2}}.
\end{align}
From this, it immediately follows that $u$ and $\nabla u$ are bounded in $\{|z|< R\}$, provided u has finite energy in $\{|z| < 8R\}$.

\medskip

To show (\ref{r.3}) and (\ref{r.4}), we use an argument similar to the one of Corollary \ref{Cor2}, Step 3. Moreover, with a standard approximation argument we can assume that u and a are smooth. By scaling, without loss of generality we can reduce ourselves to consider the case $R=1$. For a cut-off function $\eta$ of $\{|x|<2\}$ in $\{|x|<4\}$, we have\footnote{Similarly to Corollary \ref{Cor2} the correct notation would be $u\otimes \nabla\eta$ instead of $u\nabla \eta$.}
\begin{align}\label{R.A}
-\nabla\cdot a\nabla{(\eta{u})}=-\nabla\cdot a (\nabla\eta u )-\nabla\eta\cdot a\nabla{u}-\eta \nabla\cdot a \nabla u.
\end{align}  
Therefore, since u is assumed $a$-harmonic, $v:=\eta u$ satisfies in $\{|x|<8\}$ the equation
\begin{equation*}
 -\nabla\cdot a\nabla{v}=-\nabla\cdot a (\nabla\eta u )-\nabla\eta\cdot a\nabla{u}.
\end{equation*}
By the representation formula and the choice of the cut-off $\eta$, we get for almost every $\{|x| < 1\}$
\begin{align*}
u(x)=&\int u(z) \nabla G(a;z,x)\cdot a\nabla\eta(z)dz -\int G(a;z,x)\nabla \eta(z) \cdot a\nabla u(z)dz,\\
\nabla u(x)=&\int u(z) \nabla\nabla G(a;z,x)\cdot a\nabla\eta(z)dz -\int \nabla_xG(a;z,x)\nabla \eta(z) \cdot a\nabla u(z)dz,
\end{align*}
and since
\begin{align*}
\int \nabla\nabla G(a; z, x) \cdot a \nabla \eta(z)\stackrel{(\ref{Sym})-(\ref{TGy})}{=} \nabla\eta(x) \stackrel{|x|<1}{=} 0,
\end{align*}
it also follows
\begin{align*}
u(x)=&\int u(z)\nabla G(a;z,x)\cdot a\nabla\eta(z)dz -\int G(a;z,x)\nabla \eta(z) \cdot a\nabla u(z)dz,\\
\nabla u(x)=&\int (u(z)-c) \nabla\nabla G(a;z,x)\cdot a\nabla\eta(z)dz -\int \nabla_xG(a;z,x)\nabla \eta(z) \cdot a\nabla u(z)dz,
\end{align*}
for any constant $c \in \mathbb{R}$.
We now apply H\"older's inequality in the integrals on the r.h.s. of the previous inequalities and obtain by (\ref{KC}) and the choice of the test-function $\eta$
\begin{align}
 &|u(x)|\lesssim\int_{2<|z|<4}|\nabla G(a;z,x)||u(z)| +\int_{1<|z|<2}|G(a;z,x)||\nabla{u(z)}| \label{r.6}\\
&\lesssim \biggl(\int_{2<|z|<4}|\nabla G(a;z,x)|^2+|G(a;z,x)|^2\biggr)^{\frac{1}{2}}\biggl(\int_{|z|<4}|u|^2+|\nabla{u}|^2\biggr)^{\frac{1}{2}}\label{r.5},
\end{align}
and
\begin{align}\label{r.10}
 |\nabla u(x)|\lesssim  \biggl(\int_{2<|z|<4}|\nabla\nabla G(a;z,x)|^2&+|\nabla_x G(a;z,x)|^2\biggr)^{\frac{1}{2}}\notag \\
 & \times\biggl(\int_{|z|<4}|u-c|^2+|\nabla{u}|^2\biggr)^{\frac{1}{2}}.
\end{align}
H\"older's inequality in $\{2<|z|<4\}$ and Sobolev's inequality in $\{|z|>2\}$ imply
\begin{align*}
\int_{2<|z|<4}&|G(a;z,x)|^2\lesssim \biggl(\int_{2<|z|<4}|G(a;z,x)|^{\frac{2d}{d-2}}\biggr)^{\frac{d-2}{d}}\lesssim \int_{|z|>2}|\nabla{G(a;z,x)}|^2,
\end{align*}
and analogously
\begin{align*}
\int_{2<|z|<4}&|\nabla_x G(a;z,x)|^2\lesssim \int_{|z|>2}|\nabla_z\nabla_x{G(a;z,x)}|^2 { =} \int_{|z|>2}|\nabla\nabla{G(a;z,x)}|^2 ,
\end{align*}
where in the last identity we use the symmetry of G, cf. (\ref{Ex.28a}).
By plugging the two previous inequalities in (\ref{r.5}) and (\ref{r.10}) respectively, we reduce to
\begin{align*}
 |u(x)|\lesssim& \biggl(\int_{|z|> 2}|\nabla G(a;z,x)|^2\biggr)^{\frac{1}{2}}\biggl(\int_{2< |z|<4}|u|^2+|\nabla{u}|^2\biggr)^{\frac{1}{2}},\\
 |\nabla u(x)|&\lesssim \biggl(\int_{|z|>2}|\nabla\nabla G(a;z,x)|^2\biggr)^{\frac{1}{2}}\biggl(\int_{|z|<4}|u-c|^2+|\nabla{u}|^2\biggr)^{\frac{1}{2}}.
\end{align*}
By the assumptions $ |x| < 1$ and $d>2$, we can appeal to (\ref{R1}) and (\ref{R2}) to respectively get
\begin{align*}
|u(x)| &\lesssim \biggl(\int_{|z|<4}|u|^2+|\nabla{u}|^2\biggr)^{\frac{1}{2}},\\
|\nabla u(x)| &\lesssim \biggl(\int_{|z|<4}|u-c|^2+|\nabla{u}|^2\biggr)^{\frac{1}{2}}.\\
\end{align*}
By choosing $c= \fint_{|z|< 4}u$ in the last estimate and applying Poincar\'e inequality in $\{|z|< 4\}$, we conclude (\ref{r.4}). To obtain (\ref{r.3}) from the second to last inequality, we use that u is $a$-harmonic in $\{|z|<8\}$ and apply the Caccioppoli's estimate 
\begin{align}\label{Ca}
\int_{|z|<4}|\nabla u|^2 \lesssim \int_{|z|<8}| u|^2.
\end{align}
This last one is proven by testing the equation $-\nabla\cdot a\nabla u=0$ with $\bar\eta^2 u$ where $\bar \eta$ cuts $\{|z| < 4\}$ in $\{|z|<8 \}$ and using\footnote{Also here, we should write $u\otimes \nabla\bar\eta$ and $u\otimes u$ instead of $u\nabla \eta$ and $u\times u$.}
\begin{align}\label{comp}
\nabla(\bar\eta u)\cdot a\nabla(\bar\eta u) &=\bar\eta^2\nabla{u}\cdot a \nabla{u}+2 \bar\eta \bigl( u \nabla\bar\eta \bigr)\cdot a \nabla u + |u|^2 \nabla\bar \eta\cdot a \nabla\bar\eta\\
&\stackrel{(\ref{Sym})}{=}\nabla(\bar\eta^2u )\cdot a \nabla u + |u|^2 \nabla\bar\eta\cdot a \nabla\bar\eta
\end{align}
together with properties (\ref{KC}) and \eqref{bdd}.

\medskip

It remains to prove that (\ref{R3}) implies the Liouville property (\ref{L}): Let us consider $a\in\Omega$ smooth and let $u$ solve $-\nabla\cdot a\nabla u=0$ in $\mathbb{R}^d$. Arguing analogously as for (\ref{r.4}) (cf. also \eqref{r.10}), (\ref{R3}) yields for every $r \leq R$
 \begin{align*}
\sup_{|z|< r}|\nabla u(z)|\lesssim R^{\frac{\alpha}{2}} \biggl(\fint_{|z|<4R}|\nabla u|^2 \biggr)^{\frac{1}{2}},
\end{align*}
and thus
 \begin{align*}
\sup_{|z|< r}|\nabla u(z)|\lesssim R^{\beta} \biggl(\fint_{|z|<4R}|\nabla u|^2 \biggr)^{\frac{1}{2}},
\end{align*}
for every $R \geq 1$ and $\beta \in [\frac \alpha 2 , 1)$. 
We now apply the rescaled Caccioppoli's Inequality (\ref{Ca}), namely
\begin{equation}\label{r.9}
 \int_{|z|< 4R}|\nabla u|^2\lesssim R^{-2}\int_{|z|< 8R}|u|^2,
\end{equation}
and get
 \begin{align*}
\sup_{|z|< r}|\nabla u(z)|\lesssim R^{-1+\beta} \biggl(\fint_{|z|<8R}| u|^2 \biggr)^{\frac{1}{2}},
\end{align*}
Therefore, the Liouville property $(L)$ is obtained by sending $R\rightarrow +\infty$ in this last inequality and by the arbitrariness of r.\\

\bigskip

{\sc Proof of Lemma \ref{BGO2}}. Using translation and scaling invariance, up to a relabelling of the domain $D$, an elementary covering argument shows that it is enough to establish 
(\ref{b210}) with radius $R$ replaced by $\frac{R}{\sqrt{d}}$. Therefore, 
it suffices to show the result with 
the inner ball $\{|x|<\frac{R}{\sqrt{d}}\}\cap D$ replaced by the cube $(- R, R)^d\cap D$ and the outer
ball $\{|x|<2R\}\cap D$ in (\ref{b29}) by the cube $(-2R,2R)^d\cap D$. By scale invariance, again up to a relabelling of D, we may reduce to
$(-\frac{\pi}{4},\frac{\pi}{4})^d \cap D$ and $(-\frac{\pi}{2},\frac{\pi}{2})^d \cap D$, respectively. Thus the proof boils down to showing that
\begin{equation}\label{b5}
\int \int_{(-\frac{\pi}{4},\frac{\pi}{4})^d}|\nabla u|^2dx \mathrm{d}\mu \lesssim
\sup_{F}{\int |F u|^2 \mathrm{d}\mu},
\end{equation}
where the supremum is taken over all linear functionals $F$ satisfying
\begin{equation}\label{b3}
|Fv|^2 <\int_{(-\frac{\pi}{2},\frac{\pi}{2})^d}|\nabla v|^2dx.
\end{equation}
with v satisfying (i) and (ii) of Definition \ref{R.6} in  $(-\frac{\pi}{2},\frac{\pi}{2})^d\cap \partial D$.

\medskip

We consider the Fourier cosine series coefficients 
\begin{equation}\label{b4}
{\mathcal F}u(k):={\textstyle\sqrt\frac{2}{\pi^d}}
\int_{(-\frac{\pi}{2},\frac{\pi}{2})^d} u(x)\Pi_{i=1}^d\cos(k_ix_i)dx
\quad\mbox{for}\;k\in\mathbb{Z}^d-\{0\}.
\end{equation}
We observe that if we show that for any even $l\in\mathbb{N}$ we have
\begin{equation}\label{b2}
\int_{(-\frac{\pi}{4},\frac{\pi}{4})^d}|\nabla u|^2dx
\lesssim\sum_{k\in\mathbb{Z}^d-\{0\}}\frac{1}{|k|^{2l}}|{\mathcal F}u(k)|^2,
\end{equation}
then we conclude (\ref{b5}). Indeed, for every $k\in \mathbb{Z}^d - \{0\}$, the linear functional $\mathcal{F}u (k)$ has the boundedness property (\ref{b3}): since $k\neq 0$, it holds for any $c\in \mathbb{R}^d$

\begin{align*}
{\mathcal F}u(k) = \int_{(-\frac{\pi}{2},\frac{\pi}{2})^d} (u(x) - c)\Pi_{i=1}^d\cos(k_ix_i)dx.
\end{align*}
Choosing $c = \fint_{(-\frac{\pi}{2},\frac{\pi}{2})^d} u$, we may apply Poincar\'e's inequality in $(-\frac{\pi}{2},\frac{\pi}{2})^d$ and thus get
\begin{equation}\nonumber
{\mathcal F}u(k) \lesssim  \int_{(-\frac{\pi}{2},\frac{\pi}{2})^d} |\nabla u(x) |^2 dx
\end{equation}
and also (\ref{b3}) by renormalising the definition of the linear functionals given by (\ref{b4}).
Hence, after integrating in $\mu$, we may reformulate (\ref{b2}) as
\begin{equation}\nonumber
\int\int_{(-\frac{\pi}{4},\frac{\pi}{4})^d}|\nabla u|^2dx \mathrm{d}\mu
\lesssim\sum_{k\in\mathbb{Z}^d-\{0\}}\frac{1}{|k|^{2l}} \int|\mathcal{F}u(k)|^2\mathrm{d}\mu.
\end{equation}
Now picking $l\in\mathbb{N}$ with $l>\frac{d}{2}$ so that 
$\sum_{k\in\mathbb{Z}^d-\{0\}}\frac{1}{|k|^{2l}}\lesssim 1$, we obtain (\ref{b5}).

\medskip
We now turn to the argument for (\ref{b2})
and introduce the abbreviation $\|\cdot\|$ for the $L^2((-\frac{\pi}{2},\frac{\pi}{2})^d)$-norm. Let $\eta$ be a cut-off function for $(-\frac{\pi}{4},\frac{\pi}{4})^d$ in
$(-\frac{\pi}{2},\frac{\pi}{2})^d$ with
\begin{equation}\label{b8}
\sum_{1 \leq |\alpha| \leq n}|\partial^\alpha\eta|\lesssim 1 .
\end{equation}
The two ingredients for (\ref{b2}) are the interpolation inequality for any function $v$ of zero spatial average and even $l\in \mathbb{N}$
\begin{eqnarray}\label{b7}
\|\eta^l v\|
&\lesssim&\|\eta^{l+1}\nabla v\|^\frac{l}{l+1}
\Big(\sum_{k\in\mathbb{Z}^d-\{0\}}\frac{1}{|k|^{2l}}|{\mathcal F}v(k)|^2\Big)^\frac{1}{2(l+1)}\nonumber\\
&&+\Big(\sum_{k\in\mathbb{Z}^d-\{0\}}\frac{1}{|k|^{2l}}|{\mathcal F}v(k)|^2\Big)^\frac{1}{2},
\end{eqnarray}
and the Caccioppoli estimate

\begin{align}\label{Ex.11}
|| \eta^{m}\nabla  u||
\lesssim \inf_{c\in\mathbb{R}^d}||\eta^{(m-n)}(u-c)||,
\end{align}
for $m \geq 2n$.\\

Before proving (\ref{b7}) and (\ref{Ex.11}) we show how to obtain (\ref{b2}) from them: We insert the Caccioppoli estimate (\ref{Ex.11}) with  $m=n^2$ in
the interpolation inequality (\ref{b7}) with $v=u- \fint_{(-\frac{\pi}{2},\frac{\pi}{2})^d} u$, $\eta$ replaced by $\eta^n$ and $l=n-1$. Appealing to Young's inequality we obtain
\begin{align*}
|| \eta^{n(n-1)} ( u - \fint_{(-\frac{\pi}{2},\frac{\pi}{2})^d} u)||^2 \lesssim \sum_{k\in\mathbb{Z}^d-\{0\}}\frac{1}{|k|^{2l}}|{\mathcal F}u(k)|^2
\end{align*}
and thus inequality (\ref{b2}) by another application of the Caccioppoli estimate (\ref{Ex.11}) and the choice of the support of $\eta$.

\medskip

To obtain the interpolation estimate (\ref{b7}), we rewrite it without Fourier transform, appealing to the representation of the Laplacian $-\triangle_N$ with Neumann boundary
conditions through the Fourier cosine series by ${\mathcal F}(-\triangle_N)w(k)=|k|^2{\mathcal F}w(k)$:
\begin{eqnarray*}
\Big(\sum_{k\in\mathbb{Z}^d-\{0\}}\frac{1}{|k|^{2l}}|{\mathcal F}v(k)|^2\Big)^\frac{1}{2}
&=&\|w\|\quad\mbox{where}\quad(-\triangle_N)^\frac{l}{2}w=v.
\end{eqnarray*}
For (\ref{b7}) it thus suffices to show that for an arbitrary function $w$  
\begin{align}\label{int0}
\|\eta^{l}\triangle^\frac{l}{2}w\|
&\lesssim \|\eta^{l+1}\nabla\triangle^\frac{l}{2}w\|^\frac{l}{l+1}
\|w\|^\frac{1}{l+1}+\|w\|.
\end{align}
It is easily seen that this family of interpolation
estimates indexed by even $l$ follows from the following two-tier family of
interpolation inequalities indexed by $m\in\mathbb{N}$
\begin{align}\label{int1}
\|\eta^{2m}\triangle^m& w\|\notag\\
&\lesssim\|\eta^{2m+1}\nabla\triangle^mw\|^\frac{1}{2}
\|\eta^{2m-1}\nabla\triangle^{m-1} w\|^\frac{1}{2}
+\|\eta^{2m-1}\nabla\triangle^{m-1} w\|
\end{align}
and
\begin{align}\label{int2}
\|\eta^{2m-1}\nabla \triangle^{m-1}& w\|\notag\\
&\lesssim \|\eta^{2m}\triangle^{m}w\|^\frac{1}{2}
\|\eta^{2m-2}\triangle^{m-1} w\|^\frac{1}{2}
+\|\eta^{2m-2}\triangle^{m-1} w\|.
\end{align}
Indeed, plugging \eqref{int2} in \eqref{int1} yields
\begin{align*}
\|\eta^{2m}\triangle^m& w\|\notag\\
&\lesssim\|\eta^{2m+1}\nabla\triangle^mw\|^\frac{1}{2}
\|\eta^{2m}\triangle^{m}w\|^\frac{1}{4}
\|\eta^{2m-2}\triangle^{m-1} w\|^\frac{1}{4}\\
& \quad + \|\eta^{2m+1}\nabla\triangle^mw\|^\frac{1}{2}\|\eta^{2m-2}\triangle^{m-1} w\|^\frac{1}{2} \\
&
\quad +\|\eta^{2m}\triangle^{m}w\|^\frac{1}{2}
\|\eta^{2m-2}\triangle^{m-1} w\|^\frac{1}{2} +\|\eta^{2m-2}\triangle^{m-1} w\|
\end{align*}
and after an application of Young's inequality
\begin{align*}
\|\eta^{2m}\triangle^m& w\|\notag\\
&\lesssim \|\eta^{2m+1}\nabla\triangle^mw\|^\frac{2}{3}\|\eta^{2m-2}\triangle^{m-1} w\|^\frac{1}{3}\\
&  \quad +\|\eta^{2m+1}\nabla\triangle^mw\|^\frac{1}{2}\|\eta^{2m-2}\triangle^{m-1} w\|^\frac{1}{2} + \|\eta^{2m-2}\triangle^{m-1} w\|\\
&=\|\eta^{2m+1}\nabla\triangle^mw\|^\frac{2}{3}\|\eta^{2m-2}\triangle^{m-1} w\|^\frac{1}{3}\\
&  \quad +\|\eta^{2m+1}\nabla\triangle^mw\|^\frac{1}{2}\|\eta^{2m-2}\triangle^{m-1} w\|^\frac{1}{4}\|\eta^{2m-2}\triangle^{m-1} w\|^\frac{1}{4} + \|\eta^{2m-2}\triangle^{m-1} w\|.
\end{align*}
We apply once more Young's inequality to the first to last term on the r.h.s and get
\begin{align*}
\|\eta^{2m}\triangle^m& w\|\notag\\
&\lesssim \|\eta^{2m+1}\nabla\triangle^mw\|^\frac{2}{3}\|\eta^{2m-2}\triangle^{m-1} w\|^\frac{1}{3} +  \|\eta^{2m-2}\triangle^{m-1} w\|.
\end{align*}
By iterating the previous estimates we conclude \eqref{int0} from \eqref{int1} and \eqref{int2}.

\medskip

Obviously, the two-tier family \eqref{int1}-\eqref{int2} reduces to the two estimates
\begin{eqnarray*}
\|\eta^{2m}\triangle v\|
&\lesssim&\|\eta^{2m+1}\nabla\triangle v\|^\frac{1}{2}
\|\eta^{2m-1}\nabla v\|^\frac{1}{2}
+\|\eta^{2m-1}\nabla v\|,\\
\|\eta^{2m-1}\nabla v\|
&\lesssim&\|\eta^{2m}\triangle v\|^\frac{1}{2}
\|\eta^{2m-2} v\|^\frac{1}{2}
+\|\eta^{2m-2} v\|,
\end{eqnarray*}
which by Young's inequality follow from
\begin{eqnarray*}
\|\eta^{2m}\triangle v\|
&\lesssim&(\|\eta^{2m+1}\nabla\triangle v\|+\|\eta^{2m}\triangle v\|)^\frac{1}{2}
\|\eta^{2m-1}\nabla v\|^\frac{1}{2},\\
\|\eta^{2m-1}\nabla v\|
&\lesssim&(\|\eta^{2m}\triangle v\|+\|\eta^{2m-1}\nabla v\|)^\frac{1}{2}
\|\eta^{2m-2} v\|^\frac{1}{2}.
\end{eqnarray*}
Thanks to (\ref{b8}) and the choice of the support of $\eta$, these two last estimates immediately follow
from integration by parts, the Cauchy-Schwarz and the triangle inequalities.

\medskip

In proving (\ref{Ex.11}) we also introduce the notation  $(\cdot, \cdot)$ for the usual $L^2$-inner product in $(-\frac{\pi}{2},\frac{\pi}{2})^d$.  We first observe that by identity (\ref{comp}) in the proof of Remark \ref{R.1} and properties (\ref{KC}) and (\ref{bdd}), it holds for any test function $\bar\eta$ in $(-\frac{\pi}{2},\frac{\pi}{2})^d$
\begin{align}\label{Ex.4a}
\int \nabla(\bar\eta^2u )\cdot a \nabla u &\gtrsim \int |\nabla(\bar \eta u )|^2 - \int |\nabla\bar\eta|^2 |u|^2\notag\\
 &\gtrsim \int \bar\eta^2|\nabla u |^2 - \int |\nabla\bar\eta|^2 |u|^2.
\end{align}
We now test (\ref{Ex.39}) with $\eta^{2m}(u-c)$; by the invariance of the equation under translations, we may fix without loss of generality $c=0$. Thanks to the cut-off function $\eta$, we obtain
\begin{align*}
\int \nabla(\eta^{2m}u )\cdot a \nabla u+&\varepsilon \sum_{i=1}^d (\partial_i^{n}\bigl(\eta^{2m}u\bigr),\partial_i^{n}u)=0,
\end{align*}
and by (\ref{Ex.4a})
\begin{align}\label{CC1}
\lambda ||\eta^m\nabla u||^2 +&\varepsilon \sum_{i=1}^d (\partial_i^{n}\bigl(\eta^{2m}u\bigr),\partial_i^{n}u)\lesssim ||\nabla\eta^{m}u||^2 \stackrel{(\ref{b8})}{\lesssim} ||\eta^{(m-1)}u||^2.
\end{align}
This last inequality implies (\ref{Ex.11}) provided that we have that for every $i= 1, ..., d$ it holds
\begin{align}\label{Ex.4}
\varepsilon ( \partial_i^{n}(\eta^{2m}u) , \partial_i^{n}u ) \gtrsim - ||\eta^{(m-n)}u||^2.
\end{align}
To simplify the notation, we drop the index $i$. We claim that to obtain (\ref{Ex.4}) it is sufficient to show that
\begin{align}\label{Ex.12a}
||\eta^{\alpha}\partial^k u|| &\lesssim  || \eta^{(\alpha+l)} \partial^{k+l}u||^{\frac{k}{k+l}} ||\eta^{(\alpha-k)}u||^{\frac{l}{k+l}} + ||\eta^{(\alpha-k)}u||.
\end{align}
Indeed, the l.h.s of (\ref{Ex.4}) can be estimated from below by
\begin{align*}
\varepsilon (\partial^{n}(\eta^{2m}u \bigr), \partial^{n}u)\gtrsim& \ \varepsilon ||\eta^{m}\partial^nu||^2- \varepsilon \sum_{i=1}^n (|\partial^{i}\eta^{2m}|| \partial^{n-i}u|, |\partial^n u|)\\
\stackrel{(\ref{b8})}{\gtrsim}& \varepsilon ||\eta^{m}\partial^nu||^2- \varepsilon \sum_{i=1}^n (\eta^{m-i}|\partial^{n-i}u|, \eta^{m}|\partial^nu|),
\end{align*}
and after an application of Cauchy-Schwarz's inequality followed by Young's inequality, by
\begin{align}\label{Ex.14}
\varepsilon (\partial^{n}(\eta^{2m}u \bigr), \partial^{n}u) \notag\\
&\gtrsim \varepsilon||\eta^{m}\partial^n u||^2  - \varepsilon C\sum_{i=1}^{n}|| \eta^{m-i}\partial^{n-i}u||^2 \notag\\
&\gtrsim  \varepsilon||\eta^{m}\partial^n u||^2 -\varepsilon C ||\eta^{(m-n)}u||^2 -\varepsilon C\sum_{i=1}^{n-1}|| \eta^{m-i}\partial^{n-i}u||^2.
\end{align}
We now may plug (\ref{Ex.12a}) with $\alpha= l= i$ and $k= n-i$ in the terms of the sum in (\ref{Ex.14})  and apply Young's inequality to conclude (\ref{Ex.4}).

\medskip

Inequality (\ref{Ex.12a}) is implied by the interpolation estimate
\begin{align}\label{Ex.30}
||\eta^{\alpha}\partial u|| \lesssim ||\eta^{(\alpha-1)}u|| + ||\eta^{(\alpha+l)}\partial^{l+1} u||^{\frac {1}{(l+1)} }||\eta^{(\alpha-1)} u||^{\frac {l}{(l+1)}},ì
\end{align}
for all $l \geq 1$, combined with an iterated application of Young's inequality, similar to the one used to infer \eqref{int0} from \eqref{int1}-\eqref{int2}. We prove (\ref{Ex.30}) by induction on l: For $l=1$, we estimate
\begin{align*}
||\eta^{\alpha}\partial u||^2 = -(\partial\eta^{2\alpha} \partial u, u)& - (\eta^{2\alpha} \partial^2 u, u)\\
&\stackrel{(\ref{b8})}{\lesssim} (\eta^{\alpha-1}|u|, \eta^\alpha |\partial u |) + ( \eta^{\alpha-1} |u|, \eta^{\alpha+1}|\partial^2 u|),
\end{align*}
and use Cauchy-Schwarz's inequality to get
\begin{align*}
||\eta^{\alpha}\partial u||^2 \lesssim ||\eta^{(\alpha-1)} u|| ||\eta^{\alpha}\partial u|| + ||\eta^{(\alpha+1)}\partial^2 u|| ||\eta^{(\alpha-1)} u||.
\end{align*}
Young's inequality yields
\begin{align*}
||\eta^{\alpha}\partial u||^2 \lesssim ||\eta^{(\alpha-1)}u||^2 + ||\eta^{(\alpha+1)}\partial^2 u|| ||\eta^{(\alpha-1)} u||,
\end{align*}
i.e. (\ref{Ex.30}) with $l=1$. Let us now assume that (\ref{Ex.30}) holds for every $i \leq l$ and $l > 1$: First using (\ref{Ex.30}) with $i=1$, we  have
\begin{align*}
||\eta^{\alpha}\partial u||^2 \lesssim  \  ||\eta^{(\alpha-1)} u||^2 &+ ||\eta^{(\alpha+1)}\partial^2 u|| ||\eta^{(\alpha-1)} u||,
\end{align*}
then, another application of (\ref{Ex.30}) with $i=l$ on $||\eta^{(\alpha+1)}\partial^2 u||$ yields
\begin{align*}
||\eta^{\alpha}\partial u||^2{\lesssim } ||\eta^{(\alpha-1)} u||^2 &+ ||\eta^{\alpha}\partial u||||\eta^{(\alpha-1)} u||\\
&+||\eta^{(\alpha+1+l)} \partial^{l+2} u||^{\frac {1}{(l+1)} }||\eta^{\alpha}\partial u||^{\frac {l}{(l+1)}}||\eta^{(\alpha-1)} u||.
\end{align*}
By Young's inequality we can absorb the term $||\eta^{\alpha}\partial u|| $ on the r.h.s and thus obtain (\ref{Ex.30}) with $l+1$ . 

\bigskip

{\sc Proof of Lemma \ref{BGO3}}.
Using translation and scaling invariance, up to a relabelling of the domain $D$, we may reduce ourselves to $R \simeq 1$. 
As in the proof of Lemma \ref{BGO2} it is convenient to work with boxes instead of balls and thus argue that for a function $u$ satisfying (\ref{Ex.39}) in the domain $\mathbb{R}^d \setminus [-\frac{\pi}{8},\frac{\pi}{8}]^d \cap D$ we have 
\begin{equation}\label{b210b}
\int \int_{\mathbb{R}^d \setminus [-\frac{\pi}{2},\frac{\pi}{2}]^d}|\nabla u|^2 \dx \mathrm{d}\mu \lesssim
\sup_{F}{\int |F u|^2 \mathrm{d}\mu},
\end{equation}
where the supremum runs over all linear functionals $F$ bounded in the sense of
\begin{equation}\label{b29b}
|Fv|^2\leq \int_{\mathbb{R}^d \setminus [-\frac{\pi}{8},\frac{\pi}{8}]^d}|\nabla v|^2 \dx,
\end{equation}
with v satisfying (i) and (ii) in the sense Definition \ref{R.6} in $\mathbb{R}^d \setminus [-\frac{\pi}{8},\frac{\pi}{8}]^d$.

\medskip

We observe that Lemma \ref{BGO2} also implies by Poincar\'e's inequality in $(-\frac{\pi}{4},\frac{\pi}{4})^d$ that for all $w$ satisfying (\ref{Ex.39a}) in $(-\frac{\pi}{2},\frac{\pi}{2})^d$
\begin{align*}
\int \inf_{c\in \mathbb{R}}{ \int_{(-\frac{\pi}{4},\frac{\pi}{4})^d} |w- c|^2 dx} d\mu \lesssim \sup_{\mathcal{F}}\int |\mathcal{F}(w)|^2 d\mu ,
\end{align*}
where the linear functionals $\mathcal{F}$ satisfy (\ref{b29a}) in $(-\frac{\pi}{2},\frac{\pi}{2})^d$.
We want to argue that a similar estimate holds also for $u$ solving (\ref{Ex.39}) in the domain $\mathbb{R}^d \setminus [-\frac{\pi}{8},\frac{\pi}{8}]^d \cap D$, namely that
\begin{align}\label{Annulus}
\int \inf_{c\in \mathbb{R}}{ \int_{(-\frac{\pi}{4},\frac{\pi}{4})^d \setminus (-\frac{\pi}{6},\frac{\pi}{6})^d} |u- c|^2 dx} d\mu \lesssim \sup_{\mathcal{F}}\int |\mathcal{F}(u)|^2 d\mu,
\end{align}
for all linear functionals satisfying (\ref{b29b}).

\medskip

Before tackling (\ref{Annulus}) we show how to conclude the proof: Thanks to (\ref{Annulus}), if we define $\mathcal{F}_0(u):= \fint_{(-\frac{\pi}{4},\frac{\pi}{4})^d \setminus (-\frac{\pi}{6},\frac{\pi}{6})^d} u$, it holds by the triangle inequality
\begin{align}\label{Ann1}
\int \int_{(-\frac{\pi}{4},\frac{\pi}{4})^d \setminus (-\frac{\pi}{6},\frac{\pi}{6})^d} |u|^2 dx d\mu \lesssim \sup_{\mathcal{F}}\int |\mathcal{F}(u)|^2 d\mu + \int |\mathcal{F}_0 u|^2 d\mu.
\end{align}
We now observe that since $d>2$, we may apply first H\"older's inequality in $(-\frac{\pi}{4},\frac{\pi}{4})^d \setminus (-\frac{\pi}{6},\frac{\pi}{6})^d$ and then Sobolev's inequality in the outer domain
$\{|x| > \frac \pi 6\}$ to infer that
\begin{align*}
|\mathcal{F}_0(u)|^2 \lesssim \bigl(\int_{(-\frac{\pi}{4},\frac{\pi}{4})^d \setminus (-\frac{\pi}{6},\frac{\pi}{6})^d} |u|^{\frac{2d}{d-2}}\bigr)^{\frac{d-2}{d}} \lesssim \int_{\mathbb{R}^d \setminus [-\frac{\pi}{8},\frac{\pi}{8}]^d } |\nabla u|^2,
\end{align*}
i.e. $\mathcal{F}_0$ satisfies (\ref{b29b}). Therefore, it follows from (\ref{Ann1}) that 
\begin{align}\label{Annulus2}
\int \int_{(-\frac{\pi}{4},\frac{\pi}{4})^d \setminus (-\frac{\pi}{6},\frac{\pi}{6})^d} |u|^2 dx d\mu \lesssim \sup_{\mathcal{F}}\int |\mathcal{F}(u)|^2 d\mu.
\end{align}
We obtain  (\ref{b210b}) by plugging in the previous estimate the Caccioppoli's inequality
\begin{align}\label{CC3}
 \int \hat\eta^{2m}|\nabla  u|^2 dx\lesssim \int  |u|^2 \mathbf{1}_{(-\frac{\pi}{4},\frac{\pi}{4})^d\setminus(-\frac{\pi}{6},\frac{\pi}{6})^d }dx
\end{align}
where $\hat\eta$ is any cut-off function for the set $\mathbb{R}^d\setminus (-\frac{\pi}{5},\frac{\pi}{5})^d$ in $\mathbb{R}^d\setminus (-\frac{\pi}{6},\frac{\pi}{6})^d$ satisfying (\ref{b8}) and $m \in\mathcal{N}$, $m > n$. We remark that this last inequality is obtained in a similar way to (\ref{Ex.11}),
this time testing the equation in (\ref{Ex.39}) with $\hat\eta^{2m} u$ which is an admissible test function since $u$ solves (\ref{Ex.39}). This time, with a more careful look, we rewrite (\ref{CC1}) as
\begin{align}\label{CC2}
\lambda ||\hat\eta^{m}\nabla u||^2 +\varepsilon \sum_{i=1}^d (\partial_i^{n}\bigl(\hat\eta^{2m}u\bigr),\partial_i^{n}u)&\lesssim ||\nabla\hat\eta^{m}u||^2 \\\notag
&{\lesssim} ||\hat\eta^{(m-1)}u \mathbf{1}_{(-\frac{\pi}{5},\frac{\pi}{5})^d\setminus(-\frac{\pi}{6},\frac{\pi}{6})^d }||^2
\end{align}
and the second term on the l.h.s of (\ref{CC2}) as
\begin{align*}
\varepsilon (\partial^{n}& (\hat\eta^{2m}u \bigr), \partial^{n}u)\gtrsim \ \varepsilon ||\hat\eta^{m}\partial^nu||^2- \varepsilon \sum_{i=1}^n (|\partial^{i}\hat\eta^{2m}|| \partial^{n-i}u|, |\partial^n u|)\\
{\gtrsim}& \varepsilon ||\hat\eta^{m}\partial^nu||^2- \varepsilon \sum_{i=1}^n (\hat\eta^{m-i}\mathbf{1}_{\mbox{supp}(\nabla\hat\eta)}|\partial^{n-i}u|, \hat\eta^{m}|\partial^nu|)\\
\gtrsim &\varepsilon||\hat\eta^{m}\partial^n u||^2  - \varepsilon C\sum_{i=1}^{n}|| \hat\eta^{m-i}\partial^{n-i}u \mathbf{1}_{\mbox{supp}(\nabla\hat\eta)}||^2 \\
\gtrsim & \varepsilon||\hat\eta^{m}\partial^n u||^2 -\varepsilon C ||\hat\eta^{(m-n)}u \mathbf{1}_{\mbox{supp}(\nabla\hat\eta)}||^2 -\varepsilon C\sum_{i=1}^{n-1}|| \hat\eta^{m-i}\partial^{n-i}u \mathbf{1}_{\mbox{supp}(\nabla\hat\eta)}||^2.
\end{align*}
We now choose another function $\bar \eta$, satisfying (\ref{b8}) and such that it cuts off the set supp$(\nabla\hat\eta)$ in  $(-\frac{\pi}{4},\frac{\pi}{4})^d\setminus(-\frac{\pi}{8},\frac{\pi}{8})^d$. We thus have for any $\alpha>0$
\begin{align}
\hat\eta^\alpha \mathbf{1}_{\mbox{supp}(\nabla\hat\eta)} \lesssim (\hat\eta \bar\eta)^\alpha \lesssim \hat\eta^\alpha,\ \ \ \ \ \ \mbox{supp}(\bar\eta\hat\eta) \subset (-\frac{\pi}{4},\frac{\pi}{4})^d\setminus(-\frac{\pi}{6},\frac{\pi}{6})^d\label{CC4}
\end{align}
Therefore, we may bound
\begin{align*}
\varepsilon (\partial^{n}& (\hat\eta^{2m}u \bigr), \partial^{n}u)\\
\gtrsim & \varepsilon||\hat\eta^{m}\partial^n u||^2 -\varepsilon C ||(\hat\eta \bar\eta)^{(m-n)}u ||^2 -\varepsilon C\sum_{i=1}^{n-1}|| (\hat\eta \bar \eta)^{m-i}\partial^{n-i}u||^2.
\end{align*}
We now may apply the interpolation inequality (\ref{Ex.12a}) with $\eta= \hat\eta\bar\eta$ to the second and third term on the r.h.s. and appeal to (\ref{CC4}) to estimate
and conclude that
\begin{align*}
\varepsilon (\partial^{n}& (\hat\eta^{2m}u \bigr), \partial^{n}u)\gtrsim -\varepsilon ||\hat\eta^{(m-n)}u  \mathbf{1}_{(-\frac{\pi}{4},\frac{\pi}{4})^d\setminus(-\frac{\pi}{6},\frac{\pi}{6})^d}||,
\end{align*}
so that the Caccioppoli's inequality (\ref{CC3}) follows by the previous inequality and (\ref{CC2}).

\medskip

It thus remains only to prove inequality (\ref{Annulus}): We fix a $\tilde \eta$ to be a cut-off function of $\mathbb{R}^d \setminus (-\frac{\pi}{6},\frac{\pi}{6})^d$ in $\mathbb{R}^d \setminus (-\frac{\pi}{8},\frac{\pi}{8})^d$ and set 
$$
\tilde u:= \tilde\eta (u- \fint_{(-\frac{\pi}{2},\frac{\pi}{2})^d \setminus (-\frac{\pi}{8},\frac{\pi}{8})^d}u).
$$
In the same spirit of Lemma \ref{BGO2}, we define for $k\in\mathbb{Z}^d-\{0\}$
\begin{align}\label{Fourier}
\tilde{ \mathcal F}u(k)&:={\textstyle\sqrt\frac{2}{\pi^d}}
\int_{(-\frac{\pi}{2},\frac{\pi}{2})^d} \tilde\eta (u- \fint_{(-\frac{\pi}{2},\frac{\pi}{2})^d \setminus (-\frac{\pi}{8},\frac{\pi}{8})^d}u)\Pi_{i=1}^d\cos(k_ix_i)dx\notag\\
&={\textstyle\sqrt\frac{2}{\pi^d}}\int_{(-\frac{\pi}{2},\frac{\pi}{2})^d} \tilde u \ \Pi_{i=1}^d\cos(k_ix_i)dx = \mathcal{F}(\tilde u)(k),
\end{align}
with $\mathcal{F}(\tilde u)(k)$ the k-th coefficient of the cosine Fourier series of  the function $\tilde u$ defined in (\ref{b4}). We first note that  the functional $\tilde{\mathcal{F}}(u)(k)$ satisfies (\ref{b29b}) for all $k \in \mathbb{Z}^d \setminus \{0 \}$: Indeed, similarly to the proof of Lemma \ref{BGO2} we may write
\begin{align*}
|\tilde{ \mathcal F}u(k)|^2= \bigl({\textstyle\sqrt\frac{2}{\pi^d}}\int_{(-\frac{\pi}{2},\frac{\pi}{2})^d} (\tilde u - \fint_{(-\frac{\pi}{2},\frac{\pi}{2})^d \setminus (-\frac{\pi}{8},\frac{\pi}{8})^d}\tilde u)\Pi_{i=1}^d\cos(k_ix_i)dx\bigr)^2, 
\end{align*}
so that first Cauchy-Schwarz's inequality, then Poincar\'e's inequality and the definition of $\tilde u$ yield
\begin{align*}
|\tilde{ \mathcal F}u(k)|^2&\lesssim  \int_{(-\frac{\pi}{2},\frac{\pi}{2})^d}|\nabla \tilde u|^2 dx\\
& \lesssim  \int_{(-\frac{\pi}{2},\frac{\pi}{2})^d\setminus(-\frac{\pi}{8},\frac{\pi}{8})^d } |\nabla u|^2 dx + \int_{(-\frac{\pi}{2},\frac{\pi}{2})^d\setminus(-\frac{\pi}{8},\frac{\pi}{8})^d }\hspace{-0.8cm}|u -\fint_{(-\frac{\pi}{2},\frac{\pi}{2})^d \setminus (-\frac{\pi}{8},\frac{\pi}{8})^d}u |^2 dx
\end{align*}
Another application of Poincar\'e's inequality on the second term of the r.h.s implies (\ref{b29b}). Therefore, analogously to Lemma \ref{BGO2}, if we show that for any even $l\in\mathbb{N}$ we have
\begin{equation}\label{Annulus3}
\inf_{c\in \mathbb{R}}\int_{(-\frac{\pi}{4},\frac{\pi}{4})^d \setminus (-\frac{\pi}{6},\frac{\pi}{6})^d}|u -c|^2dx
\lesssim\sum_{k\in\mathbb{Z}^d-\{0\}}\frac{1}{|k|^{2l}}|{\tilde{\mathcal F}}u(k)|^2,
\end{equation}
we may conclude (\ref{Annulus}) by integrating in $d \mu$ and taking $l > \frac d 2$. 

\medskip

To show (\ref{Annulus3}), we first apply the interpolation inequality (\ref{b7}) to the function $\tilde u - \fint_{(-\frac{\pi}{2},\frac{\pi}{2})^d} \tilde u$ and observe that, if we choose as cut-off $\eta$ in (\ref{b7}) any smooth function cutting
$ (-\frac{\pi}{4},\frac{\pi}{4})^d \setminus (-\frac{\pi}{5},\frac{\pi}{5})^d$ in $ (-\frac{\pi}{2},\frac{\pi}{2})^d \setminus (-\frac{\pi}{6},\frac{\pi}{6})^d$, then it holds by our choice of $\tilde u$ that
\begin{align*}
 \inf_{c \in \mathbb{R}}\int \eta^{2l} |u - c|^2 dx &\lesssim \bigl( \int \eta^{2(l+1)}|\nabla u|^2 \bigr)^\frac{l}{l+1}
\Big(\hspace{-0.1cm}\sum_{k\in\mathbb{Z}^d-\{0\}}\frac{1}{|k|^{2l}}|{\mathcal F}\tilde u(k)|^2\Big)^\frac{1}{(l+1)}\hspace{-0.2cm}\\
&+\hspace{-0.1cm}\sum_{k\in\mathbb{Z}^d-\{0\}}\frac{1}{|k|^{2l}}|{\mathcal F}\tilde u(k)|^2\\
&\stackrel{(\ref{Fourier})}{=} \bigl( \int \eta^{2(l+1)}|\nabla u|^2 \bigr)^\frac{l}{l+1}
\Big(\hspace{-0.1cm}\sum_{k\in\mathbb{Z}^d-\{0\}}\frac{1}{|k|^{2l}}|\tilde{\mathcal F} u(k)|^2\Big)^\frac{1}{(l+1)}\hspace{-0.2cm}\\
&+\hspace{-0.1cm}\sum_{k\in\mathbb{Z}^d-\{0\}}\frac{1}{|k|^{2l}}|\tilde{\mathcal F} u(k)|^2.
\end{align*}
Since our choice of $\eta$ implies that the Caccioppoli inequality (\ref{Ex.11}) holds also for $u$, we may buckle up the previous estimate and get (\ref{Annulus3}).

\bigskip

{\sc Poincar\'e-Sobolev's inequality (\ref{PSI})}. Without loss of generality, we assume that $R=1$ and $z=0$. We start by observing that (\ref{PSI}) immediately follows if we combine the estimate
\begin{align}\label{PSI1}
\bigl( \int_{|x|> 1} |v|^p \bigr)^{\frac 1 p} \lesssim \bigl( \int_{|x|> 1} |\nabla v|^2 \bigr)^{\frac 1 2 } + \bigl(\int_{|x|> 1} |v|^2\bigr)^\oh
\end{align}
for every $2 \leq p < +\infty$ and $v\in W^{1,2}(\mathbb{R}^2)$, with the Poincar\'e's inequality
\begin{align}\label{PSI3}
\bigl( \int_{|x|> 1} |w|^2 \bigr)^{\frac 1 2} \lesssim_D  \bigl( \int_{|x|> 1} |\nabla w|^2 \bigr)^{\frac 1 2 }
\end{align}
for $w \in W^{1,2}(\mathbb{R}^2)$ which vanishes outside $D$.

\medskip

By standard approximation, we may assume in both inequalities that $w$ and $v$ are $C^1(\mathbb{R}^2)$. We first show (\ref{PSI1}): Thanks to Young's inequality and by interpolation it is enough to prove that for every integer $p \geq 1$
\begin{align}\label{PSI4}
\bigl( \int_{|x| > 1} |v|^{2p} \bigr)^{\frac{1}{2p}} \lesssim \bigl( \int_{|x|> 1} |\nabla v|^2 \bigr)^{\frac{p-1}{2p}} \bigl( \int_{|x|>1} |v|^2 \bigr)^{\frac{1}{2p}}.
\end{align}
This, in turn, is implied by the same inequality on the whole space $\mathbb{R}^2$
\begin{align}\label{PSI5}
\bigl( \int  |v|^{2p} \bigr)^{\frac{1}{2p}} \lesssim \bigl( \int |\nabla v|^2 \bigr)^{\frac{p-1}{2p}} \bigl( \int  |v|^2 \bigr)^{\frac{1}{2p}}.
\end{align}
Indeed, arguing as for the Sobolev's inequality on outer domains, if $v$ is defined in $\{|x| >1 \}$, we may extend it to the whole space by the radial reflection $x \rightarrow \frac{x}{|x|^2}$, apply (\ref{PSI5}) and controlling
\begin{align*}
\int_{|x|<1} |v|^2 + |\nabla v|^2 \lesssim   \int_{|x|> 1} |v|^2 + |\nabla v|^2,
\end{align*}
thanks to our choice of extension.\\
Inequality (\ref{PSI5}) for $p=1$ is trivial; We thus show it for every $p \geq 2$ by induction. For $p=2$, the Isoperimetric Inequality (\cite{EG}, Section 5.6.1. Theorem 1) applied to $v^2$ and the chain rule yield
\begin{align*}
\bigl( \int |v|^4 \bigr)^{\frac 1 2} \lesssim \int |\nabla(v^2)| \lesssim \int |v| |\nabla v|,
\end{align*}
so that after an application of H\"older's inequality this turns into
\begin{align*}
 \bigl( \int |v|^4 \bigr)^{\frac 1 2} \lesssim \bigl(\int |v|^2 \bigr)^{\frac 1 2}\bigl( \int|\nabla v|^2 \bigr)^{\frac 1 2},
\end{align*}
i.e. estimate (\ref{PSI5}) for $p=2$. Let us now assume that (\ref{PSI5}) holds for $p \geq 2$. Appealing to the Isoperimetric Inequality applied this time to $v^{p+1}$, we obtain by the chain rule and H\"older's inequality that
\begin{align*}
\bigl( \int |v|^{2(p+1)} \bigr)^{\frac 1 2} \lesssim  \bigl(\int |v|^{2p} \bigr)^{\frac 1 2}\bigl( \int|\nabla v|^2 \bigr)^{\frac 1 2}.
\end{align*}
Plugging (\ref{PSI5}) with $p$ in the previous inequality it follows
\begin{align*}
\bigl( \int |v|^{2(p+1)} \bigr)^{\frac 1 2} &\lesssim \bigl( \int |\nabla v|^2 \bigr)^{\frac{p-1}{2}} \bigl( \int  |v|^2 \bigr)^{\frac{1}{2}}\bigl( \int|\nabla v|^2 \bigr)^{\frac 1 2}\\
&\lesssim \bigl( \int |\nabla v|^2 \bigr)^{\frac{p}{2}} \bigl( \int  |v|^2 \bigr)^{\frac{1}{2}}.
\end{align*}
This establishes (\ref{PSI5}) for $p+1$ and thus concludes the proof of (\ref{PSI4}).

\medskip

We now turn to Poincar\'e's inequality (\ref{PSI3}): Without loss of generality, we prove (\ref{PSI3}) on the cylindrical domain $(-1, 1) \times \mathbb{R}$.
In the following estimates we write $x =(x_1, x_2) \in (-1, 1) \times  \mathbb{R}$. Thanks to the geometry of the domain and our assumption on $w$, for every $(x_1, x_2) \in (-1, 1) \times \mathbb{R} \setminus \{|x| <1\}$, there exists
a $x_1' \in (-1, 1)$ such that $w(x_1', x_2) =0$. We may thus bound
\begin{align*}
 | w(x_1, x_2)|^2 &= |w(x_1, x_2) - w(x_1', x_2)|^2 \\
 &\lesssim | \int_{x_1}^{x_1'} \nabla w( z, x_2) dz |^2 \lesssim \int_{(-1, 1)}|\nabla w( z, x_2)|^2 dz, 
\end{align*}
where in the last inequality we applied H\"older's inequality in $z$ and the fact that $x_1, x_1' \in (-1, 1)$. Therefore, estimate (\ref{PSI3}) follows once we integrate the previous inequality in $\{|x| >1\}$ and appeal again to the Dirichlet boundary conditions.

\bigskip

{\sc Proof of Lemma \ref{exp}.} \ Throughout this proof, we use the same notation of Corollary \ref{Cor1}. We recall that $\lesssim_D$ stands for $\leq C$ with a constant depending on d, $\lambda$ and the smallest bounded direction of D.
Moreover, without loss of generality, we may reduce ourselves to work in a cylindrical domain \mbox{$D = I \times \mathbb{R}^{d-1}$} with $I \subset \mathbb{R}$ bounded. Therefore since in this setting the implicit multiplicity constant in \eqref{Ex.32} depends on $D$ through the size of $I$, we may substitute the notation $\lesssim_D$ with $\lesssim_I$.
We start by observing that it is sufficient to prove that there exists a constant $C_0$ such that for every $i=1,...,d-1$
\begin{align}
\int \exp{\bigl( -\frac{x'_i-x_{0,i}'}{C_0} \bigr)} |\nabla u|^{2} \lesssim_I \int \exp{\bigl( -\frac{x'_i-x_{0,i}'}{C_0} \bigr)} |g|^2,\label{exp00}\\
\int \exp{\bigl( \frac{x'_i-x_{0,i}'}{C_0} \bigr)} |\nabla u|^{2} \lesssim_I \int \exp{\bigl( \frac{x_i'-x_{0,i}'}{C_0} \bigr)} |g|^2, \label{exp01}
\end{align}
where $x'=( x'_1, .., x_{d-1}')$ and $x'_0=( x'_{0,1}, .., x_{0,d-1}')$. Indeed, from \eqref{exp00}-\eqref{exp01} it follows that
\begin{align*}
\int \sum_{i=1}^{d-1} \exp{\bigl( \frac{|x'_i-x_{0,i}'|}{C_0} \bigr)} |\nabla u|^{2} \lesssim_I \int \sum_{i=1}^{d-1} \exp{\bigl( \frac{|x'_i-x_{0,i}'|}{C_0} \bigr)} |g|^2
\end{align*}
and thus also \eqref{Ex.32} thanks to the convexity of the function $\exp( \frac{|s|}{C_0} )$ and the equivalence of the norms in $\mathbb{R}^{d-1}$. 

\medskip

Without loss of generality, we fix $x_0'=0$ in \eqref{exp00}-\eqref{exp01}. We start with the argument for (\ref{exp00}) for a fixed i, say $i=1$. As $D$ is unbounded, we are not allowed a priori to test the equation before (\ref{Ex.32}) with $\exp{\bigl( \pm \frac{x'_1}{C_0} \bigr)} u$. We first need to consider the following approximate problem: For every $M \in \mathbb{N}$, let $ D_M:= I \times (-M ; +\infty )^{d-1}$ and $u_M$ be the solution of
\begin{align}\label{PM}
\begin{cases}
 -\nabla\cdot a\nabla u_M = \nabla\cdot g_M \mbox{\ \ \  in \  $D_M$}\\
 \ \ u_M= 0  \hspace{2.35cm} \mbox{in \  $\partial D_M $,}
\end{cases}
\end{align}
where $g_M(x_1, x') = \mathcal{X}_{(-(M-1), +\infty)^{d-1}}(x')g(x_1, x')$. 
We start by showing that if we prove (\ref{exp00}) for every $u_M$, namely
\begin{align}\label{exp002}
\int \exp{\bigl( -\frac{x'}{C_0} \bigr)} |\nabla u_M|^{2} \lesssim_I \int \exp{\bigl( -\frac{x'}{C_0} \bigr)} |g_M|^2,
\end{align}
then we can conclude (\ref{exp00}).
Indeed, thanks to the geometry of the domain and the Dirichlet boundary conditions in $\partial I \times (-M; +\infty)^{d-1}$  for every $x' \in (-M; +\infty)^{d-1}$ it holds by Poincar\`e's inequality in $I$ that
\begin{align}\label{Poi}
\int_D |u(x_1, x')|^2 d x_1 \leq |I|^2 \int |\nabla_{x_1} u(x_1, x')|^2.
\end{align}
 Therefore, by integrating over $x'\in\mathbb{R}$, it follows
\begin{align*}
\int_{D} |u_M|^2 \leq |I|^2 \int |\nabla u_M|^2 \lesssim_I \int |\nabla u_M|^2 ,
\end{align*}
and by the energy estimate and the definition of $g_M$,
\begin{align*}
\int_{D} |u_M|^2 \lesssim_I \int |\nabla u_M|^2 \lesssim_I \int |g_M|^2 \lesssim_I \int |g|^2
\end{align*}
Thus, $\{ u_M \}_{M \in \mathbb{N}}$ is uniformly bounded in $W^{1,2}(\mathbb{R}^d)$ and, up to subsequences, weakly converges to u, the solution of the equation before (\ref{Ex.32}). 
In addition, since by construction $D_1 \subset \dots \subset D_N \subset D_{N+1} \subset \dots \subset D$, from (\ref{exp002}) it holds that for every $N, M \in \mathbb{N}$ with $N \leq M$
\begin{align*}
\int_{D_N}\exp{\bigl( -\frac{x'_1}{C_0} \bigr)} |\nabla u_M|^2 \lesssim_I \int \exp{\bigl( -\frac{x'_1}{C_0} \bigr)} |g|^2.
\end{align*}
If we send $M \rightarrow +\infty$ we obtain by weak lower semi-continuity
\begin{align*}
\int_{D_N}\exp{\bigl( -\frac{x'_1}{C_0} \bigr)} |\nabla u|^2 \lesssim_I \int \exp{\bigl( -\frac{x'_1}{C_0} \bigr)} |g|^2.
\end{align*}
Taking also the limit $N\rightarrow +\infty$, we conclude (\ref{exp00}).

\medskip

We turn now to the argument for (\ref{exp002}): In the rest of the proof, we write $u, g$ respectively for $u_M$ and $g_M$.
By multiplying the inequality (\ref{Poi}) with $\rho(x'_1):= \exp{\bigl( -\frac{x'_1}{C_0}\bigr)}$, $C_0 > 0$, and then integrating over $x'\in \mathbb{R}^{d-1}$, we obtain
\begin{align}\label{exp1}
\int_{D} \rho |u|^2\lesssim_I \int_{D} \rho|\partial_{x_1} u|^2 \lesssim_I \int_{D} \rho|\nabla u|^2,
\end{align}
which implies, being $\nabla u \in L^2(D)$ and $\rho \in L^\infty (D_M)$, that also $u$ and $\rho u \in L^2(D_M)$. Therefore, we can test the equation in (\ref{PM}) with $\rho u$ and use assumption (\ref{StE}) on $a\in \Omega$ to bound
\begin{align}\label{exp3}
\lambda \int \rho |\nabla u|^2 &\leq \int\rho \nabla u \cdot a \nabla u \notag\\
&\stackrel{(\ref{bdd})}{\lesssim} 
 \int |u|\ |\partial_{x'_1}\rho||\nabla u| + \int \rho |\nabla u| | g| + \int |u|\ |\partial_{x'_1}\rho||g |.
\end{align}
By our choice of $\rho$, we thus obtain
\begin{align}
\int \rho |\nabla u|^2 \lesssim \frac{1}{C_0} \int \rho |u| |\nabla u| + \frac{ 1}{C_0} \int \rho |u| |g | + \int \rho |\nabla u| | g|,
\end{align}
so that H\"older's inequality in every term of the r.h.s yields
\begin{align}\label{exp4}
\int \rho |\nabla u|^2 \lesssim & \ \frac{1}{C_0}\biggl( \int \rho|u|^2 \biggr)^{\frac 1 2}\biggl( \int \rho|\nabla u|^2 \biggr)^{\frac 1 2}\\
&+ \frac{1}{C_0} \biggl(\int \rho |u|^2 \biggr)^{\frac 1 2} \biggl( \int \rho |g |^2 \biggr)^{\frac 1 2} + \biggl( \int \rho|\nabla u|^2  \biggr)^{\frac 1 2}\biggl( \int \rho| g|^2 \biggr)^{\frac 1 2}. \notag
\end{align}
Appealing to (\ref{Poi}), (\ref{exp4}) turns into
\begin{align*}
\int \rho |\nabla u|^2 \lesssim_I & \ \frac {1}{ C_0 } \int \rho|\nabla u|^2 \notag\\
&+ \frac{1}{ C_0 }\biggl(\int \rho |\nabla u|^2 \biggr)^{\frac 1 2} \biggl( \int \rho |g |^2 \biggr)^{\frac 1 2} + \biggl( \int \rho|\nabla u|^2  \biggr)^{\frac 1 2}\biggl( \int \rho| g|^2 \biggr)^{\frac 1 2}.
\end{align*}
Therefore, after an application of Young's inequality in the last two terms on the r.h.s., we may choose $C_0 > C(\lambda,|I|)$, absorb the term $\int \rho|\nabla u|^2$ and thus establish (\ref{exp002}).

\medskip

Inequality (\ref{exp01}) follows similarly by considering an approximate problem in $D\times (-\infty; M)^{d-1}$ and $\rho= \exp{\bigl( \frac{x'}{C_0} \bigr)}$.

\end{document}